\documentclass[11pt,reqno]{article}

\usepackage[margin=1.15in]{geometry}
\usepackage{amsmath,amssymb,amsthm,mathtools}
\usepackage{xcolor}
\usepackage[colorlinks=true,linkcolor=blue!50!black,citecolor=blue!50!black]{hyperref}
\usepackage{microtype}

\theoremstyle{plain}
\newtheorem{theorem}{Theorem}[section]
\newtheorem{lemma}[theorem]{Lemma}
\newtheorem{proposition}[theorem]{Proposition}

\theoremstyle{definition}
\newtheorem{definition}[theorem]{Definition}
\theoremstyle{remark}
\newtheorem{remark}[theorem]{Remark}

\newcommand{\T}{\mathbb{T}}
\newcommand{\R}{\mathbb{R}}
\newcommand{\Z}{\mathbb{Z}}
\newcommand{\E}{\mathbb{E}}
\newcommand{\Var}{\operatorname{Var}}
\newcommand{\Cov}{\operatorname{Cov}}
\newcommand{\dd}{\,\mathrm{d}}
\newcommand{\weak}{\Longrightarrow}
\DeclareMathOperator{\supp}{supp}

\newcommand{\norm}[1]{\left\lVert #1\right\rVert}
\newcommand{\abs}[1]{\left\lvert #1\right\rvert}
\newcommand{\ip}[2]{\langle #1,#2\rangle}
\newcommand{\1}{\mathbf{1}}
\newcommand{\df}{\mathfrak{d}}

\begin{document}

\title{Fluctuations of the quadratic matching cost on the flat torus:\\ dimensions two, three and four}
\author{Shi Feng\thanks{Cornell University.
Email: \texttt{sf599@cornell.edu}} \and Gilles Mordant\thanks{Yale University. Email: \texttt{gilles.mordant@yale.edu}}}
\date{\today}
\maketitle

\begin{abstract}
We establish sharp variance asymptotics and limiting distributions for
the quadratic optimal matching cost between the uniform measure and an empirical measure counterpart  on the flat torus $\T^d$ in
dimensions $d=2,3,4$. More precisely, let $\mu$ be Haar probability measure and
$\mu_n$ the empirical measure of $n$ independent samples drawn from $\mu$.
Then, for $d=2,3$, 
\begin{align*}
    n\bigl(W_2^2(\mu_n,\mu)-\E[W_2^2(\mu_n,\mu)]\bigr)
    &\weak L_d,\\
    n^2\Var(W_2^2(\mu_n,\mu))
    &\longrightarrow
    \frac{1}{8\pi^4}
    \sum_{k\in\Z^d\setminus\{0\}}\frac{1}{\abs{k}^4},
\end{align*}
where $L_d$ is an explicit non-Gaussian weighted sum of independent
centered exponential random variables. In dimension four,
\begin{align*}
    \frac{n}{\sqrt{\log n}}
    \bigl(W_2^2(\mu_n,\mu)-\E[W_2^2(\mu_n,\mu)]\bigr)
    &\weak \mathcal N\left(0,\frac{1}{16\pi^2}\right).%,\\
   % \frac{n^2}{\log n}\Var(W_2^2(\mu_n,\mu))
 %   &\longrightarrow\frac{1}{16\pi^2}.
\end{align*}
The main idea of the proof is to rely on a Hoeffding decomposition of the optimal transport cost, which turns out to be asymptotically equivalent to U-statistics of order 2.
\end{abstract}

\tableofcontents

%------------------------------------------------
%                 Section Introduction
%------------------------------------------------
\section{Introduction}
\label{sec:setting}

The theory of optimal transport (OT) connects probability, analysis, and geometry
\cite{Vil09}.  In recent years, it further has become an important tool in machine learning and data
science \cite{PC19}. An important fact is that Wasserstein distances (rooted in the OT theory) provide metrics on
spaces of probability measures with suitable finite moments, while taking
the geometry of the sample space into account. Their applications to statistics, for instance, include
minimum-distance parameter estimation \cite{BJGR19}, goodness-of-fit testing
\cite{HMS21}, and the estimation of counterfactual distributions in causal
inference \cite{TGR24}. For these applications, distributional limit results for empirical transport costs are key. 

Let $\rho_n$ be the empirical measure of an
i.i.d.\ sample from a probability measure $\rho$, and let $\nu$ be a fixed
reference measure. In $\R^d$, under suitable moment and regularity
assumptions, del Barrio and Loubes \cite[Theorem~4.1]{dBL19} establish
\[
 \sqrt{n}\left(W_2^2(\rho_n,\nu)
 -\E[W_2^2(\rho_n,\nu)]\right)
 \weak \mathcal N(0,\sigma_{\rho,\nu}^2).
\]
Extensions to more general costs and weaker assumptions are obtained in
\cite{dBGL24}; an analogous result on the flat torus is proved in
\cite[Theorem~2.6]{GGCN23}. These results describe first-order fluctuations,
but their limiting variance vanishes under the null hypothesis
$\rho=\nu$; see \cite[Remark~4.4]{dBL19} and the discussion following
\cite[Theorem~2.6]{GGCN23}.

The null setting is precisely the one relevant for calibrating
goodness-of-fit tests against a specified reference distribution, to state but one application. Nontrivial
null limits are known in one dimension under suitable regularity and
integrability assumptions \cite{dBGU05}, and for measures on finite spaces
\cite{SM18}. For continuous distributions in higher dimensions, identifying
the fluctuation scale and limiting law of the unsmoothed quadratic cost
remains a substantially more delicate problem; see
\cite[Section~7, Problem~3]{dBG26}. In this paper, we prove the first limiting distribution theorem for that setting, albeit in the case of the uniform distribution on the torus in
dimensions two, three, and four. We expect the ideas here to generalize to other settings, at the cost of additional technical \textit{tours de force}. 
\medskip

More precisely, consider the $d$-dimensional torus $\T^d=(\R/\Z)^d$ endowed with its geodesic distance
$d_{\T^d}$. Let $\mu$ be the Haar measure on $\T^d$ and let $X_1,\dots,X_n$, $n\in\mathbb N$, be
an i.i.d.\ sample drawn from $\mu$. The empirical measure for that sample is
$\mu_n=n^{-1}\sum_i\delta_{X_i}$. In this work, we consider the squared 2-Wasserstein distance between an
empirical measure and the continuous reference, i.e.,
\begin{equation}
\label{eq: W2}
  W_{n,d} \;:=\; W_2^2(\mu_n,\mu)\;=\;\inf_{\pi\in\Pi(\mu_n,\mu)}\int d_{\T^d}(x,y)^2\dd\pi(x,y),
\end{equation}
where $\Pi(\mu_n,\mu)$ is the set of couplings between $\mu_n$ and $\mu$.

\subsection{Main results}
Our main result
completely characterizes the limiting distribution of $W_{n,d}$ for dimensions $d\le 4$, a long-awaited result. The case $d\geq5$ is the subject of ongoing work by the first author;  the higher dimensions indeed requires  different techniques.

Before stating the main result, let $\mathcal P_d=(\Z^d\setminus\{0\})/\{k\sim-k\}$, $\lambda_{[k]}:=\lambda_k = 4\pi^2|k|^2$
(well defined since $\lambda_{-k}=\lambda_k$), and let $(\zeta_{[k]})$ be independent
$\mathrm{Exp}(1)$. For $d\in \{2,3\}$, define the random variable
\begin{equation}\label{eq:limitlaw}
  L_d=\sum_{[k]\in\mathcal P_d}\frac{2}{\lambda_{[k]}}\bigl(\zeta_{[k]}-1\bigr),
\end{equation}
convergent in $L^2$ and a.s.\ by \eqref{eq:degenerate} below.

\begin{theorem}\label{thm:main}
Under the model above, for $d\in\{2,3\}$ and as $n\to \infty$, 
\[
  n\bigl(W_{n,d}-\E W_{n,d}\bigr)\weak L_d,
  \qquad
  n^2\Var(W_{n,d})\longrightarrow 2\sum_{k\ne0}\lambda_k^{-2}.
\]
When $d=2$, the constant is $\beta(2)/(12\pi^2)$ where $\beta$ is the Catalan $\beta$ function.
\end{theorem}

\begin{theorem}\label{thm:main2}
For $d=4$, as $n \to \infty$, it holds that
\[
  \frac{n}{\sqrt{\log n}}\bigl(W_{n,4}-\E W_{n,4}\bigr)\weak N\Bigl(0,\frac1{16\pi^2}\Bigr),
  \qquad
  \frac{n^2}{\log n}\Var(W_{n,4})\longrightarrow\frac1{16\pi^2}.
\]
\end{theorem}

\medskip

\subsection{Intuition and structure of proof}
\label{ssec:intuition}
We start by providing the intuition of Theorem \ref{thm:main} and Theorem \ref{thm:main2}. Let $T_n$ be the optimal
transport map from $\mu$ to $\mu_n$, and let $u$ be the mean-zero solution of
\begin{align*}
    \Delta u=\mu_n-\mu,
\end{align*}
understood as a weak solution.

Formally, the Poisson equation is the linearization of the
Monge--Ampère equation around the identity map. Indeed, transport
from a density $f$ to a density $g$ satisfies
\[
    g(T)\det DT=f.
\]
Forgetting the absence of density for now and writing $T_n=\mathrm{Id}+v_n$ and linearizing around the uniform
density, we use $\det(I+Dv_n)\approx1+\operatorname{div}v_n$ to obtain
\[
    \mu_n-\mu\approx-\operatorname{div}v_n.
\]
Since Transport maps are gradients of convex functions by the Brenier--McCann theorem, the mean-zero
solution $u$ of $\Delta u=\mu_n-\mu$ therefore satisfies
\[
    v_n=T_n-\mathrm{Id}\approx-\nabla u.
\]
Thus,
\begin{align*}
    W_{n,d}
    \approx
    \int_{\T^d}\abs{\nabla u}^2\dd\mu
    =
    \left\langle
        \mu_n-\mu,
        (-\Delta)^{-1}(\mu_n-\mu)
    \right\rangle.
\end{align*}
We call the expression on the right the linearized transport energy. Let   $e_k(x) := e^{2\pi ik\cdot x}$ be the Fourier basis functions on the torus. For $1\leq K\leq \infty$, define
\begin{align*}
G_K(x):=\sum_{\substack{k\in\Z^d\\0<\abs{k}\leq K}}
    \frac{e_k(x)}{\lambda_k},
\end{align*}
to be the Green kernel of the negative Laplacian on the torus, truncated at frequency $K$.   Similarly, for any function $f$, $\hat f $ is its Fourier transform, while $\widehat{f_K}(k) = \1_{|k| \leq K} \hat{f}(k)$ is the Fourier truncated function. Define the Green U-statistics
\begin{align}\label{eq:Ustat}
    U_{n,K}
    :=\frac{1}{n^2}\sum_{i\neq j}G_K(X_i-X_j).
\end{align}

To relate this $U_{n,K}$ to the linearized transport energy, set $\nu_n:=\mu_n-\mu$. We first observe
\begin{align*}
    (-\Delta^{-1} \nu_n)_K = G_K * \nu_n.
\end{align*}
Therefore, as $G_K( x - \cdot)$ has a zero expectation under $\mu$ for each $x$, 
\begin{align*}
    \left\langle\nu_n,G_K*\nu_n\right\rangle
    &=
    \int_{\T^d\times\T^d}
    G_K(x-y)\dd\nu_n(x)\dd\nu_n(y)\notag\\
    &=
    \frac{1}{n^2}\sum_{i,j=1}^nG_K(X_i-X_j)
    =
    U_{n,K}+\frac{G_K(0)}{n}.
\end{align*}
The diagonal term $G_K(0)/n$ is deterministic. Hence $U_{n,K}$ is exactly
the centered random part of the truncated linearized transport energy.
Moreover, $G_K(x-y)$ is canonical (Definition \ref{def:Canonical function}) in both variables (the zero expectation property just mentioned), so $U_{n,K}$ belongs
to the second Hoeffding subspace (Appendix \ref{app:hoeffding}). Think of the latter as the space of sums of kernels of order 2. 

Since we try to approximate $W_{n,d}$ by a function that lives in the second Hoeffding subspace, it suggests separating the second-order component of
$W_{n,d}-\E W_{n,d}$ from the higher-order terms. 
Denoting by $P_r$ the Hoeffding projection of order $r$,  
since
$P_1W_{n,d}=0$ (Lemma \ref{lem:P1=0}), our goal is to prove
\begin{align*}
    nP_2W_{n,d}
    &=
    nU_{n,K}+o_{L^2}(\sigma_n),\\
    n\sum_{r\geq3}P_rW_{n,d}
    &=
    o_{L^2}(\sigma_n),
\end{align*}
where $K$ and $\sigma_n$ are chosen according to the dimension.
\newpage 

\noindent
This intuition guides the structure of the proof, which consists of three main steps:
\begin{enumerate}
    \item show that only the second-order Hoeffding component contributes;
    \item identify this component with $U_{n,K}$;
    \item prove the limiting laws for $U_{n,K}$.
\end{enumerate}
In the developments of the display above, we choose $ K=\infty, 
    \sigma_n=1 $
   \text{ for }$ d\in\{2,3\},$
 and   $ K=n^{1/4},
   \sigma_n=\sqrt{\log n},$ for $d=4.$ Still, in both regimes, we prove the reduction
\begin{align}\label{eq:reduction-chain}
    n(W_{n,d}-\E W_{n,d})
    =nP_2W_{n,d}+o_{L^2}(\sigma_n)
    =nU_{n,K}+o_{L^2}(\sigma_n).
\end{align}
The limiting laws of the Green statistics are
\begin{align}\label{eq:green-limits}
    \frac{nU_{n,K}}{\sigma_n}
    \weak
    \begin{cases}
        L_d, & d\in\{2,3\},\\
        \mathcal N\left(0,\dfrac{1}{16\pi^2}\right), & d=4.
    \end{cases}
\end{align}
The reason we can directly choose $K = \infty$ for $d = 2,3$ is that $G \in L^2(\mu)$ as $\sum_{k\in \Z^d} \frac{1}{\lambda_k^2} <\infty$.  In dimension four however, $G\notin L^2(\mu)$, so a cutoff is necessary. The choice 
\begin{align*}
    K=n^{1/4},
\end{align*}
is the correct one as $K^{-1}$ is the typical transport-cell diameter. Frequencies much
larger than $K$ correspond to distances below the cell scale, where the
linear Poisson approximation no longer describes the individual transport
cells. By Appendix \ref{app:lattice}, we have the variance
\begin{align*}
    \Var(nU_{n,K})
    &=
    \frac{4}{n^2}\binom{n}{2}\norm{G_K}_2^2
    =
    2\frac{n-1}{n}
\sum_{\substack{k\in\Z^4\\0<\abs{k}\leq n^{1/4}}}
    \frac{1}{\lambda_k^2}
    =
    \frac{\log n}{16\pi^2}+O(1).
\end{align*}
As we shall see, the most technical part of the work is establishing the asymptotic equivalences.

\subsection{Related work}
\label{ssec:related-work}

Random Euclidean matching has a long history in combinatorial probability,
geometric optimization, and empirical process theory. A starting point is
the theorem of Ajtai, Koml\'os, and Tusn\'ady \cite{AKT84}: for two independent
samples of $n$ uniform points in the unit square, the minimum total
Euclidean matching length has order $\sqrt{n\log n}$. With probability
measures, and hence a factor $1/n$ in the cost, this corresponds to the
scale $\sqrt{\log n/n}$ for the $1$-Wasserstein distance. Shor
\cite{Sho86} used lower bounds for planar up-right matching in the
average-case analysis of bin-packing algorithms. Leighton and Shor
\cite{LS89} revealed a different planar logarithm for
\emph{bottleneck} matching: the smallest possible longest edge when
matching random points to a regular grid has order
$n^{-1/2}(\log n)^{3/4}$, with high probability. Shor and Yukich
\cite{SY91} established the corresponding higher-dimensional theory,
where the bottleneck scale is $(\log n/n)^{1/d}$ for $d\geq3$.
These results distinguish the effects of dimension and of the matching
objective; the longest edge and the average quadratic cost have different
logarithmic corrections.

Talagrand developed a complementary approach through empirical processes,
majorizing measures, and transportation estimates. His work on matching
in many dimensions \cite{Tal92}, his unified treatment of matching and
empirical discrepancy \cite{Tal94a}, and his transportation estimates for
the uniform measure \cite{Tal94b} substantially extended the classical
matching theory; see also the systematic account in \cite{Tal14}.
In dimension two, Talagrand and Yukich \cite{TY93} established
exponential-square integrability estimates for matched distances. These
give upper bounds of order $(\log n/n)^{p/2}$ for the average $p$th-power
matching cost, for every fixed finite $p\geq1$.
In a related direction, Dobri\'c and Yukich \cite{DY95} proved
an almost sure limit theorem for the minimum total bipartite Euclidean
matching length in dimensions $d\geq3$, including nonuniform sampling laws.
Yukich's monograph \cite{Yuk98} develops the broader theory of Euclidean
optimization functionals and the subadditive methods underlying many
such laws of large numbers. These results provide essential background
for the study of matching costs, while the centered quadratic cost
considered here requires information on a finer scale than its leading
law of large numbers.

For quadratic matching in dimension two, the identification of the
leading constant led to a different set of methods. Caracciolo,
Lucibello, Parisi, and Sicuro \cite{CLPS14} proposed a linearization through
the Poisson equation and predicted the leading constant for bipartite
matching. Ambrosio, Stra, and Trevisan \cite{AST19} made the PDE approach
rigorous and proved, in the empirical-to-uniform setting of the present
paper,
\[
    \lim_{n\to\infty}\frac{n}{\log n}\E W_{n,2}
    =\frac{1}{4\pi}.
\]
For two independent empirical measures, the corresponding constant is
$1/(2\pi)$. Ambrosio and Glaudo \cite{AG19} refined these estimates and
extended the precise bipartite asymptotics to closed surfaces. Bobkov and
Ledoux \cite{BL21} gave an elementary Fourier-analytic approach to the
AKT bounds, further clarifying the role of smoothing and spectral
estimates. More recently, Armegioiu, Goldman, Grotto, and Trevisan
\cite{AGGT26} established in a preprint the existence and characterization
of the leading expected-cost constants for all finite power costs on the
two-dimensional torus in the bipartite problem.

In dimensions $d\geq3$, Goldman and Trevisan \cite{GT21} combined
subadditivity with PDE estimates to prove convergence of the suitably
rescaled expected matching costs for every power $p\geq1$, both for
bipartite matching and for transport between the empirical and uniform
measures on a cube. In particular, the quadratic mean cost has order
$n^{-2/d}$. This extends the range of powers covered by earlier
subadditive matching arguments. It also illustrates the distinction
between the critical dimension two for the mean quadratic cost and the
critical dimension four for the variance studied here.

The PDE approach also gives information on optimal displacements.
Ambrosio, Glaudo, and Trevisan \cite{AGT19} proved that, in the planar
empirical-to-uniform problem, the optimal map is approximated in $L^2$
by the identity plus the gradient of a regularized Poisson potential.
Goldman, Huesmann, and Otto \cite{GHO21} developed quantitative
linearization estimates for the Monge--Amp\`ere equation with rough data.
Building on this theory, Goldman and Huesmann \cite{GH22} established
fluctuation results for the displacement in dimensions two and three,
with a curl-free Gaussian free field as the limiting object. Their
quantitative estimates are a principal analytic input to the present
work. Goldman, Huesmann, and Otto \cite{GHO25} subsequently proved a
mean-cost error of order $O(\log\log n/n)$ for empirical-to-uniform
matching on the two-dimensional flat torus and obtained sharp
displacement estimates. 

Distributional results for transport costs also arise in statistical
optimal transport. The central limit theorems of del Barrio and Loubes
\cite{dBL19}, their extensions to general costs \cite{dBGL24}, and the
torus results of Gonz\'alez-Delgado, Gonz\'alez-Sanz, Cort\'es, and Neuvial
\cite{GGCN23} describe fluctuations, centered at the expected empirical
cost, at rate $\sqrt n$ under appropriate assumptions. In the uniform
null setting, their first-order limiting
variance vanishes. Nontrivial null limits in one dimension
\cite{dBGU05} and on finite spaces \cite{SM18} demonstrate that a finer
analysis is then necessary, but do not determine the fluctuations for
continuous distributions in higher dimensions. The survey
\cite[Section~7, Problem~3]{dBG26} explicitly highlights this remaining
question.

The closest predecessor for the scalar fluctuations considered here is
Ledoux \cite{Led19}. He proves a limit theorem for the renormalized dual
Sobolev energy and, as a consequence, shows that for $d\leq3$ one can
choose heat-smoothing times $t_n\to0$ such that
\[
    n\left(
        W_2^2(\mu_n^{t_n},\mu)
        -\E W_2^2(\mu_n^{t_n},\mu)
    \right)
    \weak\chi.
\]
Here $\mu_n^{t_n}$ is the heat-smoothed empirical measure. On the flat
torus, for $d=2,3$, the limiting quadratic Gaussian functional $\chi$
has the law of $L_d$ in \eqref{eq:limitlaw}. Ledoux also explains why the
available smoothing estimates do not allow one to replace
$\mu_n^{t_n}$ by the atomic measure $\mu_n$ at the required fluctuation
scale. Thus, removing the smoothing is a substantive additional step.

The present work addresses this step for the centered quadratic cost on
the flat torus. Theorems~\ref{thm:main} and~\ref{thm:main2} give the
distributional limits and sharp variance asymptotics for the original
empirical measure in dimensions two, three, and four. The argument
identifies the second Hoeffding component with a Green-kernel statistic
and controls the higher-order components at the fluctuation scale.
This makes the spectral distinction explicit: the squared Green
coefficients are summable in dimensions two and three, producing the
non-Gaussian law $L_d$, whereas their logarithmic divergence in dimension
four produces the normalization $n/\sqrt{\log n}$ and a Gaussian limit.

\subsection{Organization of the work and notation}
Section~\ref{sec:hoeffding} recalls the Hoeffding decomposition and develops
the basic identities for the Wasserstein cost used throughout the paper.
Section~\ref{sec:kernel} contains the main analysis needed for both results.
Section~\ref{sec:clt} completes the proofs of the two main theorems. The
appendices collect the required standard material and a remark on superconcentration.

\paragraph{Notation.}
Throughout, $\mu$ denotes the Haar probability measure on $\T^d$, $\mu_n$ the
empirical measure, and $P_r$ the Hoeffding projection of order $r$. For $a \in \mathbb{C}$, let $\bar a$ denote its complex conjugate.

For
$k\in\Z^d$, set
\begin{align*}
    e_k(x):=e^{2\pi i k\cdot x},
    \qquad
    \lambda_k:=4\pi^2\abs{k}^2,
    \qquad
    \widehat h(k):=\int_{\T^d}h(x)\overline{e_k(x)}\dd\mu(x),
\end{align*}
so that $-\Delta e_k=\lambda_ke_k$. For $B\subset\Z^d$, we denote by $h_B$ the Fourier projection of $h$ onto
$B$, defined by
\begin{align*}
    \widehat{h_B}(k)
    :=
    \mathbf{1}_{\{k\in B\}}\widehat h(k),
    \qquad
    h_B(x)
    =
    \sum_{k\in B}\widehat h(k)e_k(x).
\end{align*}
Further, $\norm{\cdot}_p$ denotes the corresponding $L^p(\mu)$ norm. For a random variable $R_n$, we write 
$R_n=o_{L^2}(a_n)$ when $\E[R_n^2]^{1/2}=o(a_n)$. We write $A\lesssim B$ if $A\leq C_dB$ for constant $C_d$ depending only on the dimension $d$, $(n)_r:=n(n-1)\cdots(n-r+1)$, and $L:=n^{1/d}$.

For $1\leq K<\infty$, recall from the introduction that
\begin{align}
\label{eq:green}
G_K(x)=\sum_{\substack{k\in\Z^d\\0<\abs{k}\leq K}}
    \frac{e_k(x)}{\lambda_k},
    \qquad
    G:=G_\infty,
    \qquad
    -\Delta G=\delta_0-\mu.
\end{align}
Each $G_K$ is real, even, and centered. Moreover,
\begin{align}\label{eq:degenerate}
    \int_{\T^d}G_K(x-y)\dd\mu(y)=0,
    \qquad
    \norm{G}_2^2
    =\sum_{k\in\Z^d\setminus\{0\}}\lambda_k^{-2}<\infty
    \quad\Longleftrightarrow\quad d\leq3.
\end{align}
Thus $G\in L^2(\mu)$ for $d=2,3$, while in dimension four we use only the truncated kernels $G_K$. We say a function $f:(\T^d)^r\to\R$ is symmetric if
\begin{align*}
    f(x_1,\ldots,x_r)
    =
    f(x_{\sigma(1)},\ldots,x_{\sigma(r)})
\end{align*}
for every permutation $\sigma$ of $[r]$. The next definition is also key for our purposes.
\begin{definition}[Canonical functions] \label{def:Canonical function}
A function $f:(\T^d)^r\to\R$ is canonical
if, for every $j\in[r]$ and almost every fixed choice of the other
variables,
\begin{align*}
    \int_{\T^d}
    f(x_1,\ldots,x_{j-1},y,x_{j+1},\ldots,x_r)\dd\mu(y)
    =0.
\end{align*}
Equivalently, $f$ is canonical if and only if
$\widehat f(k_1,\ldots,k_r)=0$ whenever there exists $j\in[r]$ such that $k_j=0$.
\end{definition}
 Consequently, any Fourier multiplier preserves canonicity. In
particular, whenever they are well defined, derivatives, inverse derivatives,
Laplacians, inverse Laplacians, convolutions, and inverse convolutions preserve
canonicity. Canonicity is the key definition for the Hoeffding decomposition and the whole paper.
Given $x, y \in \mathbb{T}^d$, choose lifts $\widetilde x, \widetilde y \in \mathbb{R}^d$ for which
$|\widetilde x - \widetilde y|$ is minimal, using a fixed Borel tie-break on the cut locus, and set
\begin{equation}
\label{eq:defD}
\df_x(y) := \widetilde x - \widetilde y, \qquad |\df_x(y)| = d_{\mathbb{T}^d}(x,y). 
\end{equation}
For square-integrable scalar-, vector-, or matrix-valued functions $F$ and
$H$ on $\T^d$, we use the same notation $\langle F,H\rangle$ for the
corresponding $L^2$ inner product:
\begin{align*}
    \langle F,H\rangle
    &:=
    \int_{\T^d}F(x)H(x)\dd\mu(x),
    &&F,H:\T^d\to\R,\\
    \langle F,H\rangle
    &:=
    \int_{\T^d}F(x)\cdot H(x)\dd\mu(x),
    &&F,H:\T^d\to\R^d,\\
    \langle F,H\rangle
    &:=
    \int_{\T^d}
    \operatorname{Tr}\bigl(F(x)^\top H(x)\bigr)\dd\mu(x),
    &&F,H:\T^d\to\R^{d\times d}.
\end{align*}

%------------------------------------------------
%                 Section Hoeffding
%------------------------------------------------
\section{Hoeffding decomposition of the squared Wasserstein distance}
\label{sec:hoeffding}

In this section, we develop the exact identities used to analyze the Hoeffding components of $W_{n,d}$. 

We first recall the definition of the Hoeffding decomposition and show that the first-order component vanishes. We then compute the derivative of $W_{n,d}$ with respect to one sample point in terms of its average displacement. Instead of computing
the higher-order Hoeffding kernels directly, we test them against canonical statistics and express their variances through the response functionals. For the second-order component, this gives a simpler formula. Finally, the product Poincar\'e inequality controls the part
of the variance carried by the higher order terms of the Fourier decomposition, thereby showing that they are asymptotically negligible. These results will then be combined with the transport estimates in Section \ref{sec:kernel}.

\subsection{The decomposition and the vanishing first-order term}
\label{ssec:P1}
In Appendix~\ref{app:hoeffding}, we develop the fact that a random variable admits a Hoeffding decomposition obtained by projections of different orders $r$. The components are orthogonal and we can write
\begin{align*}
    W_{n,d}-\E W_{n,d}
    =\sum_{r=1}^{n}P_rW_{n,d},
    \qquad \Var(W_{n,d})
=\sum_{r=1}^{n}\Var(P_rW_{n,d}).
\end{align*}
We first show that the first-order component vanishes.

\begin{lemma}
\label{lem:P1=0}
For every $d\geq 2$ and $n\geq 1$,
\begin{align}\label{eq:P1=0}
    P_1W_{n,d}=0.
\end{align}
\end{lemma}

\begin{proof}
Fix $i\in[n]$ and define
\(
    \phi_i(x):=\E[W_{n,d}\mid X_i=x],
     x\in\T^d.
\)
For any $t\in\T^d$, translation invariance of $\mu$ and joint translation
invariance of the matching cost yield
\begin{align*}
    \phi_i(x+t)
    &=\int_{(\T^d)^{n-1}}
        W_{n,d}(x+t,y_{-i}+t)\,
        d\mu^{\otimes(n-1)}(y_{-i})\\
    &=\int_{(\T^d)^{n-1}}
        W_{n,d}(x,y_{-i})\,
        d\mu^{\otimes(n-1)}(y_{-i})\\
    &=\phi_i(x)
\end{align*}
for every $x,t\in\T^d$. Hence $\phi_i$ is constant. Integrating with respect to $x$ shows that this constant is
$\E W_{n,d}$. Therefore,
\begin{align*}
    P_1W_{n,d}
    =\sum_{i=1}^{n}
        \left(\E[W_{n,d}\mid X_i]-\E W_{n,d}\right)
    =0. & \qedhere
\end{align*}
\end{proof}

\subsection{Perturbing one sample point}
\label{ssec:derivative}

The optimal transport problem between a continuous distribution and an empirical measure gives rise to a tessellation of space where each cell has mass $1/n$ and its content is sent to a unique Dirac mass. We will denote the cell transported to $X_i$ by $A_i$.
The optimal coupling measure is thus given by 
\[
\pi_n(\dd x , \dd y) = \sum_{i=1}^n \delta_{X_i}(\dd x ) \1_{A_i}(y)\mu(\dd y).
\]
Recalling the notation  $\df_{X_i}$ from  \eqref{eq:defD}, we further introduce the average source-minus-target displacement
\begin{equation}
d_i := n \int_{A_i} \df_{X_i}(y) \dd\mu(y).
\end{equation}

\begin{lemma}[Gradient of cost with respect to a sample point]
\label{lem:gradient}
The function $(X_1, \ldots, X_n) \mapsto W_{n,d}$ is Lipschitz and, at almost every configuration,
\begin{align}
\label{eq:derivative_identity}
    \nabla_{X_i}W_{n,d} = \frac{2d_i}{n}
\end{align}
Moreover,
\begin{align}
\label{eq:disp_bound}
nW_n = \sum_i |d_i|^2 + n\sum_i \int_{A_i} |\df_{X_i}(y) - d_i|^2 \dd \mu (y), \qquad \sum_i |d_i|^2 \le nW_n. 
\end{align}
\end{lemma}

\begin{proof} The first idea of the proof is to rewrite any coupling in a disintegrated form, i.e., 
 $\pi(\dd x, \dd y)= \sum_i \delta_{X_i} (\dd x) q_i(y) \mu(\dd y)$. To that end, let $Q$ be the weak-$*$ compact subset of $(L^\infty(\mu))^n$ consisting of measurable vectors $q = (q_1, \ldots, q_n)$ satisfying
\[
q_i \ge 0, \qquad \sum_i q_i = 1 \ \ \mu\text{-a.e.}, \qquad \int q_i\, \dd \mu = \frac{1}{n},
\]
where we silenced the fact that each $q_i$ is a function of $y$.
Then
\[
W_n(X_1, \ldots, X_n) = \inf_{q \in Q} \sum_i \int d_{\mathbb{T}^d}(X_i, y)^2 q_i(y)\, \dd \mu(y).
\]
The differentiability statement will follow by a version of the envelope theorem, that we now develop.
The cost $d_{\mathbb{T}^d}(x,y)^2$ is uniformly Lipschitz in $x$. Thus,  evaluating the cost for a different point configuration at an optimizer $q^0$  gives 
\begin{align*}
W_n(X_1, \ldots , X_n) - W_n(X_1', \ldots, X_n') &\le \sum_i \int  \big ( d_{\mathbb{T}^d}(X_i, y)^2 - d_{\mathbb{T}^d}(X_i', y)^2\big)q_i^0(y)\, \dd \mu(y) \\
& \lesssim n^{-1}\sum_{i=1}^n d_{\mathbb{T}^d}(X_i, X_i'),
\end{align*}
 which shows that $W_{n,d}$ is Lipschitz.

We next consider perturbing a single $X_i$. We fix a reference set of sample points. As above, let $q^0$ be the optimizer for the initial set of points; it is further unique up to a $\mu$ null set. Denote by $q^t$ the optimizer for the same set of points with $X_i$ replaced by $X_i +tv $. By weak-$*$ compactness of $Q$, \cite[Theorem~5.20]{Vil09} and uniqueness of the transport map owing to McCann's theorem \cite{McC01}, $q^t \to q^0$ weak-$*$.

For $t \ne 0$, set
\[
g_t(y) := \frac{d_{\mathbb{T}^d}(X_i + tv, y)^2 - d_{\mathbb{T}^d}(X_i, y)^2}{t}.
\]
For $t > 0$, optimality at the two point clouds gives the elementary sandwich
\[
\int g_t(y) q_i^t(y)\, \dd \mu(y) \le \frac{W_n(X_i + tv) - W_n(X_i)}{t} \le \int g_t(y) q_i^0(y)\, \dd \mu(y).
\]
The same comparison, with the inequalities reversed, applies as $t \uparrow 0$. Away from the cut locus, $g_t(y) \to 2\mathfrak{d}_{X_i}(y) \cdot v$. The functions $g_t$ are uniformly bounded and the cut locus has $\mu$-measure zero, so dominated convergence gives this convergence in $L^1(\mu)$. Combining it with $q^t \to q^0$ weak-$*$ shows that both outer integrals have the same limit. Therefore
\[
\left.\frac{d}{dt}\right|_{t=0} W_n(X_1, \ldots, X_i + tv, \ldots, X_n) = 2\int_{A_i} \mathfrak{d}_{X_i}(y) \cdot v\, \dd \mu(y) = \frac{2}{n} d_i \cdot v.
\]
This proves \eqref{eq:derivative_identity} at every differentiability point, hence almost everywhere by Rademacher's theorem. Since $d_i$ is the mean of $\mathfrak{d}_{X_i}$ under the probability measure $n\mathbf{1}_{A_i}m$, the variance identity gives
\[
n\int_{A_i} |\mathfrak{d}_{X_i}(y)|^2\, \dd \mu(y) = |d_i|^2 + n\int_{A_i} |\mathfrak{d}_{X_i}(y) - d_i|^2\, \dd \mu(y).
\]
Summing in $i$ proves the identity and inequality in \eqref{eq:disp_bound}. Equality in the inequality would require the displacement to be constant almost everywhere on every transport cell.
\end{proof}

Let us also define the vector measure 
\begin{equation}
\label{eq: jVecMeas}
\mathfrak{j}:=\frac1n\sum_i d_i\,\delta_{X_i},
\end{equation}
where $\delta_{X_i}$ is
the Dirac mass at $X_i$.  This measure will be compared with $-\nabla u$ throughout Section
\ref{sec:kernel}. It is an $\R^d$-valued measure on $\T^d$, acting on a continuous vector field
$\varphi:\T^d\to\R^d$ by
\[
  \ip{\mathfrak{j}}{\varphi}=\frac1n\sum_i d_i\cdot\varphi(X_i).
\]

\subsection{Recovering one Hoeffding order from a covariance}
\label{sec:covariance}

The explicit formula for the Hoeffding projection in Appendix \ref{app:hoeffding} is difficult to estimate directly in this case. Instead, we test $W_{n,d}$ against the $U_a$ defined below, which belongs to the $r$-th Hoeffding subspace. By Hoeffding orthogonality, their covariance with $W_{n,d}$ depends only on $P_rW_{n,d}$. The response functional $\mathcal R_r(a)$ defined below collects these
covariances, and its dual norm then determines $\Var(P_rW_{n,d})$.

Let $a:(\T^d)^r\to\R$ be smooth, symmetric, and canonical, and first define
\begin{align*}
    U_a
    :=
    \sum_{\substack{i_1,\ldots,i_r\in[n]\\
    \text{pairwise distinct}}}
    a(X_{i_1},\ldots,X_{i_r}).
\end{align*}
Additionally, set
\begin{align*}
    g
    :=
    \nabla_1(-\Delta_1)^{-1}a, \qquad 
    G_i^{\#}(x)
    :=
    \sum_{\substack{i_2,\ldots,i_r\in[n]\setminus\{i\}\\
    \text{pairwise distinct}}}
    g(x,X_{i_2},\ldots,X_{i_r}),
\end{align*}
where $\nabla_1$ and $\Delta_1$ act only on the first variable. Define the response functional
\begin{align*}
    \mathcal R_r(a)
    :=
    \Cov(W_{n,d},U_a)
\end{align*}
and its dual norm by
\begin{align*}
    \norm{\mathcal R_r}_*
    :=
    \sup_{\substack{a\text{ symmetric and canonical}\\
    \norm{a}_2=1}}
    \abs{\mathcal R_r(a)}.
\end{align*}

\begin{lemma}\label{lem:order-r-response}
For every smooth, symmetric, and canonical $a$,
\begin{align}\label{eq:order-r-response}
    \mathcal R_r(a)
    =
    \frac{2}{n}\sum_{i=1}^n
    \E\left[
        d_i\cdot G_i^{\#}(X_i)
    \right].
\end{align}
Moreover,
\begin{align}\label{eq:order-r-duality}
    \Var(P_rW_{n,d})
    =
    \frac{\norm{\mathcal R_r}_*^2}{r!(n)_r}.
\end{align}
\end{lemma}

\begin{proof}
Since $a$ is canonical, it has zero mean in its first variable.
Thus, by the definition of $g$,
\[
    a
    =-\Delta_1(-\Delta_1)^{-1}a
    =-\operatorname{div}_1 g.
\]
For each ordered tuple $(i_1,\ldots,i_r)$ of distinct indices,
we integrate by parts in $X_{i_1}$, keeping all other sample
coordinates fixed. 
Importantly, periodicity eliminates the boundary term.

Since $\E U_a=0$, using \eqref{eq:derivative_identity} and then
grouping the ordered tuples according to their first index gives
\begin{align*}
    \Cov(W_{n,d},U_a)
    &=
    \sum_{\substack{i_1,\ldots,i_r\in[n]\\
                    \text{pairwise distinct}}}
    \E\left[
        W_{n,d}\,a(X_{i_1},\ldots,X_{i_r})
    \right]\\
    &=
    \sum_{\substack{i_1,\ldots,i_r\in[n]\\
                    \text{pairwise distinct}}}
    \E\left[
        \nabla_{X_{i_1}}W_{n,d}\cdot
        g(X_{i_1},\ldots,X_{i_r})
    \right]\\
    &=
    \frac{2}{n}
    \sum_{i=1}^n
    \sum_{\substack{i_2,\ldots,i_r\in[n]\setminus\{i\}\\
                    \text{pairwise distinct}}}
    \E\left[
        d_i\cdot g(X_i,X_{i_2},\ldots,X_{i_r})
    \right]\\
    &=
    \frac{2}{n}\sum_{i=1}^n
    \E\left[
        d_i\cdot G_i^{\#}(X_i)
    \right].
\end{align*}
The last equality
uses linearity of expectation and the definition of $G_i^{\#}$.
This proves \eqref{eq:order-r-response}. Let $h_r:=h_{r,W_{n,d}}$ be the order-$r$ Hoeffding kernel
defined in Appendix \ref{app:hoeffding}.
Hoeffding orthogonality gives
\begin{align*}
    \mathcal R_r(a)
    =
    (n)_r\langle h_r,a\rangle,\qquad 
    \Var(P_rW_{n,d})
    =
    \binom{n}{r}\norm{h_r}_2^2.
\end{align*}
Therefore,
\(   \norm{\mathcal R_r}_*
    =
    (n)_r\norm{h_r}_2,
\)
which proves \eqref{eq:order-r-duality}.
\end{proof}

We now give a simpler form for the second-order component of the Hoeffding expansion.
To that end, let $F=F(X_1,\ldots,X_n)\in L^2$ be real-valued, symmetric, and
jointly translation invariant. By Appendix \ref{app:hoeffding}, its second-order Hoeffding kernel has the
form $q(x-y)$:
\begin{align*}
    P_2F
    =
    \sum_{1\leq i<j\leq n} (\E[F\;|\; X_i,X_j] - \E[F])
    =
    \sum_{1\leq i<j\leq n}q(X_i-X_j).
\end{align*}

Note that $q$ is even and centered. Let $B\subset\Z^d\setminus\{0\}$ be symmetric under $k\mapsto-k$.
Recall $q_B$ is the Fourier projection of $q$ onto $B$, and set
\begin{align*}
    (P_2F)_B
    &:=
    \sum_{1\leq i<j\leq n}q_B(X_i-X_j),\\
    S_f
    &:=
    \sum_{i\neq j}f(X_i-X_j).
\end{align*}
The following lemma shows that estimating $(P_2F)_B$ reduces to bounding
$\Cov(F,S_f)$ uniformly over real, even functions $f$ with $\norm{f}_2=1$
and Fourier support in $B$. This allows us to estimate each frequency
band separately.
\begin{lemma}
\label{lem:pairwise-duality}
For every such $B$,
\begin{align}\label{eq:pairwise-duality}
    \Var\big((P_2F)_B\big)
    =
    \binom{n}{2}\norm{q_B}_2^2
    =
    \frac{1}{2n(n-1)}
    \sup_{\substack{\norm{f}_2=1\\
    f\text{ real and even}\\
    \supp\widehat f\subset B}}
    \abs{\Cov(F,S_f)}^2.
\end{align}
\end{lemma}

\begin{proof}
Since $q_B(x-y)$ is canonical in each variable, terms indexed by
different unordered pairs are orthogonal. Hence
\begin{align*}
    \Var\big((P_2F)_B\big)
    =
    \binom{n}{2}\norm{q_B}_2^2.
\end{align*}
Similarly, $f(x-y)$ is canonical in each variable. Therefore,
\begin{align*}
    \Cov(F,S_f)
    =
    \Cov(P_2F,S_f)
    =
    n(n-1)\langle q,f\rangle.
\end{align*}
Since $\supp\widehat f\subset B$, $\langle q,f\rangle = \langle q_B,f\rangle$.
Taking the supremum over all admissible $f$ with $\norm{f}_2=1$
proves \eqref{eq:pairwise-duality}.
\end{proof}

The two identities have different roles below:
\eqref{eq:order-r-duality} controls the higher-order terms, while
\eqref{eq:pairwise-duality} identifies the second-order projection.

\subsection{Frequency cutoffs and the product Poincar\'e inequality}
\label{ssec:cutoff}

For a function $F:(\T^d)^n\to\R$, define its product spectral projections by
\begin{align*}
    \widehat{\Pi_{\leq\Lambda}F}(k_1,\ldots,k_n)
    &:=
    \mathbf 1_{\{\lambda_{k_1}+\cdots+\lambda_{k_n}\leq\Lambda\}}
    \widehat F(k_1,\ldots,k_n),\\
    \Pi_{>\Lambda}F
    &:=
    F-\Pi_{\leq\Lambda}F.
\end{align*}
The following lemma bounds the variance carried by the high-frequency
spectral terms.

\begin{lemma}\label{lem:high-frequency}
Let $\Lambda>0$ and let
$F\in H^1((\T^d)^n)$ be centered. Then
\begin{align}
    \Lambda\norm{\Pi_{>\Lambda}F}_2^2
    \leq
    \sum_{i=1}^n\norm{\nabla_{X_i}F}_2^2.
\end{align}
For the quadratic transport cost,
\begin{align}
    \sum_{i=1}^n
    \E\abs{\nabla_{X_i}W_{n,d}}^2
    \leq
    \frac{4}{n}\E W_{n,d}
    \leq
    \begin{cases}
        C(\log n)n^{-2},&d=2,\\
        Cn^{-5/3},&d=3,\\
        Cn^{-3/2},&d=4.
    \end{cases}
\end{align}
\end{lemma}

\begin{proof}
By Parseval's identity,
\begin{align*}
    \Lambda\norm{\Pi_{>\Lambda}F}_2^2
    &=
    \Lambda
    \sum_{\lambda_{k_1}+\cdots+\lambda_{k_n}>\Lambda}
    \abs{\widehat F(k_1,\ldots,k_n)}^2\\
    &\leq
    \sum_{k_1,\ldots,k_n}
    \left(\lambda_{k_1}+\cdots+\lambda_{k_n}\right)
    \abs{\widehat F(k_1,\ldots,k_n)}^2\\
    &=
    \sum_{i=1}^n\norm{\nabla_{X_i}F}_2^2.
\end{align*}

For $W_{n,d}$, by \eqref{eq:derivative_identity} and \eqref{eq:disp_bound}, we have
\begin{align*}
    \sum_{i=1}^n\E\abs{\nabla_{X_i}W_{n,d}}^2
    =
    \frac{4}{n^2}\E\sum_{i=1}^n\abs{d_i}^2
    \leq
    \frac{4}{n}\E W_{n,d}.
\end{align*}
The result now follows from the fact \cite[equation~(23) and discussion following]{MBNW24} that
\begin{align*}
    \E W_{n,d}
    \leq
    \begin{cases}
        C(\log n)n^{-1},&d=2,\\
        Cn^{-2/3},&d=3,\\
        Cn^{-1/2},&d=4.
    \end{cases} &\qedhere
\end{align*}
\end{proof}

%------------------------------------------------
%                 Section Order-2 kernels
%------------------------------------------------
\section{An equivalence with order-2 kernels}
\label{sec:kernel}

The objective of this section is to replace $n(W_{n,d}-\E W_{n,d})$ by an asymptotically equivalent, explicit quadratic statistic. This replacement depends on a comparison between two quantities. On the one hand there is the
current (how the mass moves) actually produced by the optimal allocation; on the other, the current   predicted by the
linear model in which the sample is viewed as a small perturbation of the uniform density and
the transport is the gradient of a Poisson potential, recall the hand-wavy computation of Section~\ref{ssec:intuition}.
 
We organize this section as follows:  Sections \ref{ssec:pair23}--\ref{ssec:higher23} focus on the case $d\le3$ and
Sections \ref{ssec:stress4}--\ref{ssec:agg4} address the critical case $d=4$.

\subsection{The displacement input}
\label{ssec:GH}

Throughout this subsection $d\in\{2,3\}$ and $L=n^{1/d}$, so that $L^{-1}$ is the typical distance
between neighbouring sample points. 
 
 The proposition below says that, once both are
averaged at a spatial scale $s$, the two agree up to a quantitatively small error. Its proof is a variation of the estimates of Goldman and Huesmann \cite{GH22}. These are stated on a torus of side $L$
on which lives a sample of unit intensity, whereas we work on the unit torus with $n$ points, and the two
normalisations differ by explicit powers of $L$.

We first need to introduce an averaging kernel. It will be used to perform the regularization needed when establishing the various estimates and comparisons between actual transport and the linearized approximation of it via the Poisson equation. Let $B_1$ be the unit ball in $\mathbb{R}^d$.
We fix, once and for all, a nonnegative radial $\eta\in C^\infty_c(B_1)$,   with
\begin{equation}\label{eq:mollifier}
  \int_{\R^d}\eta=1,\qquad \inf_{\abs\xi\le2}\bigl|\widehat\eta_{\R^d}(\xi)\bigr|\ge\frac12,
  \qquad \widehat\eta_{\R^d}(\xi)=\int_{\R^d}\eta(z)e^{-2\pi i\xi\cdot z}\dd z .
\end{equation}
Such a function exists: start from any nonnegative radial smooth bump of unit mass and rescale it by
a small enough fixed factor, which spreads out its Fourier transform until the lower bound holds on
$\abs\xi\le2$.
 
The parameter $s$ below is the wavelength of the frequency band on which the above kernel will be tested
in Section \ref{ssec:pair23}. The averaging of the kernel will be performed at the proportional length
\begin{equation}\label{eq:ells}
  \ell_s:=\max\{L^{-1},\,c_{\rm GH}s\},\qquad L^{-1}\le s\le1,
\end{equation}
where $c_{\rm GH}\in(0,1)$ is a fixed constant, chosen small enough that the estimates of
\cite{GH22} apply at every physical radius $R$ with $1\le R\le c_{\rm GH}L$.  Observe that the hypothesis
$L\gg R$ of that paper is a smallness condition on the fixed ratio $R/L$
\cite[p.~1448, footnote~1]{GH22}. Thus $\ell_s$ is a fixed multiple of $s$, floored at the sample
spacing $L^{-1}$: below that length there is nothing left to average. 
 
For $x\in\T^d$ choose any lift $\tilde x\in\R^d$ and set
\begin{equation}\label{eq:periodised}
  \eta_{s}(x):=\sum_{z\in\Z^d}\ell_s^{-d}\,\eta\Bigl(\frac{\tilde x+z}{\ell_s}\Bigr);
\end{equation}
changing the lift only reindexes this finite sum, so $\eta_s$ is well defined on $\T^d$.
 
We now introduce four different fields that will be relevant for the proof. They correspond to smoothed quantities or  objects capturing the first order behavior of the transport.  Recall the displacement measure $\mathfrak{j}=\frac1n\sum_{i}d_i\delta_{X_i}$ of Section \ref{ssec:derivative}.
Define
\begin{align}
  \rho_{s}&:=\eta_{s}*\mu_n=\frac1n\sum_{i=1}^n\eta_{s}(\cdot-X_i)
    &&\text{(the sample density, smoothed at scale $s$),}\label{eq:rho}\\
  u_{s}&:\ \Delta u_{s}=\rho_{s}-1,\quad \textstyle\int_{\T^d}u_{s}\dd\mu=0
    &&\text{(the potential predicted by the linear model),}\label{eq:u}\\
  J_{s}&:=\eta_{s}*\mathfrak{j}=\frac1n\sum_{i=1}^n\eta_{s}(\cdot-X_i)\,d_i
    &&\text{(the true current, smoothed at scale $s$),}\label{eq:J}\\
  E_{s}&:=J_{s}+\rho_{s}\nabla u_{s}
    &&\text{(the discrepancy between the two).}\label{eq:E}
\end{align}

\begin{remark}\label{displacement_remark} A few comments are in order.
First, \eqref{eq:J} admits the equivalent formulation 
\[
J_s(x)=\int_{\T^d\times\T^d}\eta_s(x-z)\, \mathfrak{d}_z(y)\dd\pi_{n}(z,y),
\] for $\pi_n$ the
 optimal plan of Section \ref{ssec:derivative};  recall that the empirical measure is its first marginal. 
 
 Second, note that the  sign in \eqref{eq:E} is the
one appropriate to our orientation: our displacements point from uniform point $y$ to the atom $x$, so
the linear prediction for $J_s$ is $-\rho_s\nabla u_s$, and $E_s$ is small exactly when the
allocation is well described by the Poisson equation \eqref{eq:u} at scale $s$.
\end{remark}

Each of the fields \eqref{eq:rho}--\eqref{eq:E} depends on the sample as well as on the point of
the torus at which it is evaluated. We write $V=V(\omega,x)$ with $\omega=(X_1,\dots,X_n)$ the
sample and $x\in\T^d$, and measure such fields in the joint norm of
$L^p\bigl(\mu^{\otimes n}\otimes\mu\bigr)$:
\[
  \norm{V(\omega,\cdot)}_{L^p_x}:=\Bigl(\int_{\T^d}\abs{V(\omega,x)}^p\dd\mu(x)\Bigr)^{1/p},
  \qquad
  \norm{V}_{L^p_{\omega,x}}:=\bigl(\E\norm{V(\omega,\cdot)}^p_{L^p_x}\bigr)^{1/p}.
\]

\begin{proposition}[Displacement input]\label{prop:GH}
Let $d\in\{2,3\}$ and $L=n^{1/d}$. There is a constant $C<\infty$, depending only on $d$ and on the
fixed choices \eqref{eq:mollifier} and \eqref{eq:ells}, such that the following bounds hold for every
$s$ with $L^{-1}\le s\le1$. For $d=2$, one has
\begin{equation}\label{eq:GH2}
  \norm{E_{s}}_{L^2_{\omega,x}}\le C\,\frac{\log(2+Ls)}{L^2s},
  \qquad
  \norm{\rho_{s}-1}_{L^4_{\omega,x}}\le \frac{C}{Ls},
  \qquad
  \norm{\nabla u_{s}}_{L^4_{\omega,x}}\le \frac CL\sqrt{\log(e/s)} .
\end{equation}
For $d=3$, it holds that
\begin{equation}\label{eq:GH3}
  \norm{E_{s}}_{L^2_{\omega,x}}\le \frac{C}{L^2s},
  \qquad
  \norm{\rho_{s}-1}_{L^4_{\omega,x}}\le \frac{C}{(Ls)^{3/2}},
  \qquad
  \norm{\nabla u_{s}}_{L^4_{\omega,x}}\le \frac{C}{L^{3/2}s^{1/2}} .
\end{equation}
\end{proposition}
 
The first bound in each line is the one that matters: it says that the nonlinearity of the transport
problem is invisible at scale $s$ up to an error of size $L^{-2}s^{-1}$, with a logarithmic loss in
dimension two. The other two bounds are auxiliary, and control the density factor $\rho_s$ multiplying
the predicted gradient in \eqref{eq:E}; they are used only in the H\"older estimate of Section
\ref{ssec:pair23}.

\begin{proof}
The estimates of \cite{GH22} are stated on the torus $\T^d_L=(\R/L\Z)^d$ of side $L$, for a sample of
unit intensity, and at a physical averaging radius $R$. The proof consists in matching the two
normalisations.
 
\emph{Step 1: dilation.} Let $S_L(x)=Lx$ and dilate the optimal plan by
\begin{equation}\label{eq:dilate}
  \int F(x,y)\dd\widehat\pi_{n}(x,y):=n\int F(Lx,Ly)\dd\pi_{n}(x,y).
\end{equation}
The first marginal of $\widehat\pi_n$ is $\sum_i\delta_{LX_i}$, i.e., the counting measure of $n$
i.i.d.\ uniform points on the torus $\T^d_L$ of volume $L^d=n$ (one point per unit volume on
average, with typical spacing of order one). This is the measure denoted $\mu^{1,L}$ in
\cite[equation~(2.2)]{GH22}.
The second marginal is the Lebesgue measure. Set
\begin{equation}\label{eq:R}
  R:=L\ell_s=\max\{1,c_{\rm GH}Ls\},
\end{equation}
and let $\widehat\rho_R,\widehat u_R,\widehat J_R$ be the averaged density, Poisson potential and
displacement current built from $\widehat\pi_n$ at radius $R$ exactly as in \eqref{eq:rho}--\eqref{eq:J}.
Direct substitution gives, for every $x\in\T^d$,
\begin{equation}\label{eq:transfer}
  \widehat\rho_R(Lx)=\rho_{s}(x),\qquad
  \widehat J_R(Lx)=L\,J_{s}(x),\qquad
  \nabla\widehat u_R(Lx)=L\,\nabla u_{s}(x).
\end{equation}
Indeed, the factor $n=L^d$ in the dilated plan cancels
the factor $L^{-d}$ in the rescaled averaging kernel,
which gives the density identity. For the current,
there is an additional factor $L$ because each
displacement is multiplied by $L$. Finally, the
zero-mean solution of the rescaled Poisson equation is
$\widehat u_R(z)=L^2u_s(z/L)$, so the chain rule gives
$\nabla\widehat u_R(Lx)=L\nabla u_s(x)$.
 \medskip
 
\emph{Step 2: the estimates of \cite{GH22}.} We first record where our four fields occur in that
paper. With $R$ the averaging radius and $\eta_R$ the mollifier, their sample of unit intensity is
$\mu^{1,L}$, and
\[
  \mu^{1,L}_R=\eta_R*\mu^{1,L}=\widehat\rho_R,\qquad
  W^{1,L}_R=\mu^{1,L}_R-1=\widehat\rho_R-1,\qquad
  \nabla u^{1,L}_R=\nabla\widehat u_R,
\]
the last because $u^{1,L}$ is the mean-zero, periodic solution of
$\Delta u^{1,L}=\mu^{1,L}-1$. Their plan $\pi^{1,L}$ is oriented from the atoms (first coordinate
$x$) to Lebesgue (second coordinate $y$), and their displacement is $y-x$, target minus atom, so
that with our opposite convention, it holds that $\int\eta_R(x)(y-x)\dd\pi^{1,L}=-\widehat J_R(0)$.

The quantity estimated in \cite[Prop.~4.8]{GH22} is  exactly our
$\widehat E_R$. Indeed, the first marginal of $\pi^{1,L}$ is $\mu^{1,L}$, so
$\int\eta_R(x)\dd\pi^{1,L}(x,y)=\int\eta_R\dd\mu^{1,L}=\widehat\rho_R(0)$, whence
\begin{equation}\label{eq:key}
  \int\eta_R(x)\bigl(y-x-\nabla u^{1,L}_R(0)\bigr)\dd\pi^{1,L}
  =-\widehat J_R(0)-\widehat\rho_R(0)\,\nabla\widehat u_R(0)=-\widehat E_R(0).
\end{equation}
In particular, the density factor in $E_s$ is the mass
$\int\eta_R\dd\mu^{1,L}$ carried by the smoothing of the atomic marginal.

For $1\le R\le c_{\rm GH}L$ and every $z\in\T^d_L$, the cited estimates read
\begin{align}
  \E\bigl[\abs{\widehat J_R(z)+\widehat\rho_R(z)\nabla\widehat u_R(z)}^2\bigr]^{1/2}
    &\le\begin{cases} C\log(1+R)/R, & d=2,\\ C/R, & d=3,\end{cases}\label{eq:cite1}\\[2pt]
  \E\bigl[\abs{\widehat\rho_R(z)-1}^4\bigr]^{1/4}&\le CR^{-d/2},\label{eq:cite2}\\[2pt]
  \E\bigl[\abs{\nabla\widehat u_R(z)}^4\bigr]^{1/4}
    &\le\begin{cases} C\log^{1/2}(2L/R), & d=2,\\ CR^{-1/2}, & d=3,\end{cases}\label{eq:cite3}
\end{align}
the right-hand sides being independent of $z$. In \cite[Lemmas 2.3 and 3.1]{GH22} that uniformity
is part of the statement, which is made for every $x\in Q_L$. In \cite[Prop.~4.8]{GH22} the estimate
is stated at $z=0$ only, and holds at every $z$ because the mollified quantities
$\widehat\rho_R(\cdot),\nabla\widehat u_R(\cdot),\widehat J_R(\cdot)$ are stationary
\cite[Section~2.1]{GH22}. Thus, $\E\abs{\,\cdot\,(z)}^p$ does not depend on $z$.

Writing $\norm{F}_{L^p(\T^d_L)}:=\bigl(L^{-d}\int_{\T^d_L}\abs F^p\bigr)^{1/p}$ for the
\emph{normalised} spatial norm, Fubini turns each of \eqref{eq:cite1}--\eqref{eq:cite3} into the
same bound for the annealed spatial norm, with the same constant:
\[
  \E\norm{F}^p_{L^p(\T^d_L)}=\frac{1}{L^d}\int_{\T^d_L}\E\abs{F(z)}^p\dd z
  \;\le\;\sup_{z\in\T^d_L}\E\abs{F(z)}^p .
\]

\emph{Step 3: back to the unit torus.} Put $q:=Ls\ge1$. By \eqref{eq:R}, it holds that 
\begin{equation}\label{eq:Rq}
  c_{\rm GH}\,q\le R\le q,
  \qquad
  R^{-1}\le c_{\rm GH}^{-1}q^{-1},\qquad R^{-d/2}\le c_{\rm GH}^{-d/2}q^{-d/2} .
\end{equation}
In dimension two, $r\mapsto\log(1+r)/r$ is decreasing, so \eqref{eq:Rq} gives
$\log(1+R)/R\le C\log(2+q)/q$.  Since $R\ge c_{\rm GH}Ls$ we have $2L/R\le 2/(c_{\rm GH}s)$, so
$\sqrt{\log(2L/R)}\le C\sqrt{\log(e/s)}$. In dimension three, \eqref{eq:Rq} gives
$R^{-1/2}\le c_{\rm GH}^{-1/2}q^{-1/2}$. Combining these with \eqref{eq:transfer}, under which the
current and gradient norms gain a factor $L^{-1}$ while the density norm is unchanged, turns
\eqref{eq:cite1}--\eqref{eq:cite3} into \eqref{eq:GH2}--\eqref{eq:GH3}; for instance, in dimension
two,
\[
  \norm{E_{s}}_{L^2_{\omega,x}}
  =L^{-1}\Bigl(\E\bigl\|\widehat J_R+\widehat\rho_R\nabla\widehat u_R\bigr\|^2_{L^2(\T^d_L)}\Bigr)^{1/2}
  \le \frac{C}{L}\cdot\frac{\log(2+q)}{q}=C\,\frac{\log(2+Ls)}{L^2s}.
\]
Finally, the separated-scale condition $1\le R\le c_{\rm GH}L$ needed in Step 2 holds as soon as
$L\ge c_{\rm GH}^{-1}$, because $R=\max\{1,c_{\rm GH}Ls\}\le c_{\rm GH}L$ for $s\le1$, and $R\ge1$ by \eqref{eq:R}. The finitely many
smaller values of $L$ are absorbed into $C$, all six quantities being bounded there.
\end{proof}

In Section~\ref{ssec:pair23}, we shall express pairings
with the original current $\mathfrak{j}$ in terms of its smoothed
version $J_s=\eta_s*\mathfrak{j}$. This is possible because convolution
with $\eta_s$ has a uniformly bounded inverse on the
frequency band of the test field, as the following lemma shows. Let
\[
  A_s:=\bigl\{k\in\Z^d:\ s^{-1}\le\abs k < 2s^{-1}\bigr\}
\]
be the frequency band of wavelength $s$.
 
\begin{lemma}[The averaging is invertible on $A_s$]\label{lem:multiplier}
The Fourier multiplier of the periodised kernel \eqref{eq:periodised} is
$\widehat\eta_{s}(k)=\widehat\eta_{\R^d}(\ell_sk)$ for $k\in\Z^d$. Since $\ell_s\le s$ by
\eqref{eq:ells}, every $k\in A_s$ satisfies $\abs{\ell_sk}\le2$, whence
\begin{equation}\label{eq:multiplier}
  \bigl|\widehat\eta_{s}(k)\bigr|\ge\frac12,\qquad k\in A_s .
\end{equation}
Consequently, for every field $V$ with Fourier support in $A_s$ there is a unique field $\widetilde V$
with the same support such that $V=\eta_{s}*\widetilde V$, and
$\norm{\widetilde V}_{L^2_x}\le2\norm{V}_{L^2_x}$.
\end{lemma}
 
\begin{proof}
The identity for the multiplier is the Poisson summation formula applied to \eqref{eq:periodised};
the lower bound is \eqref{eq:mollifier}. Define $\widetilde V$ by
$\widehat{\widetilde V}(k)=\widehat V(k)/\widehat\eta_{s}(k)$ on $A_s$ and $0$ elsewhere; Parseval
and \eqref{eq:multiplier} give the norm bound.
\end{proof}

\subsection{Second-order Hoeffding components ($d=2,3$)}
\label{ssec:pair23}
We now prove that $P_2W_{n,d}$ is well approximated by
$U_n:=U_{n,\infty}$. By Lemma \ref{lem:pairwise-duality}, it is enough to control $\Cov(W_{n,d}-U_n,S_f)$ uniformly over functions $f$ with $\norm{f}_2=1$. We first do this when the Fourier support of $f$ is contained in one frequency band.

\begin{lemma}\label{lem:one-band-response}
Let $d\in\{2,3\}$ and $L=n^{1/d}$. Suppose that
$f:\T^d\to\R$ is real and even, with $\norm{f}_2=1$ and
\begin{align*}
    \supp\widehat f\subset A_s,
    \qquad
    L^{-1}\leq s\leq1.
\end{align*}
Then,
\begin{align}\label{eq:one-band-response}
    \abs{\Cov(W_{n,d}-U_n,S_f)}
    \lesssim
    \begin{cases}
        \displaystyle
        L^{-1}\left[
            \log(2+Ls)
            +
            \left(1+(Ls)^{-1/2}\right)\sqrt{\log(e/s)}
        \right]
        +n^{-1}s,
        & d=2,\\[3mm]
        \displaystyle
        L^{-1/2}
        +L^{-3/2}s^{-1}
        +n^{-1}s^{1/2},
        & d=3.
    \end{cases}
\end{align}
\end{lemma}
\begin{proof}
    Take
\begin{align*}
    Q_f:=\frac12\sum_{i\neq j}(-\Delta)^{-1}f(X_i-X_j),
    \qquad
    g_f(x):=\sum_{i=1}^n\nabla(-\Delta)^{-1}f(x-X_i).
\end{align*}
Since $f$ is even, $(-\Delta)^{-1} f$ is even, and therefore
\begin{align*}
    \nabla_{X_i}Q_f
    &=
    g_f(X_i),\\
    -\sum_{i=1}^n\Delta_{X_i}Q_f
    &=
    S_f.
\end{align*}
Moreover, $(-\Delta)^{-1} f$ is centered, so the cross terms corresponding to
different sample points vanish. Hence
\begin{align}\label{eq:gf_bound}
    \E\norm{g_f}_{L_x^2}^2
    =
    n\norm{\nabla(-\Delta)^{-1}f}_2^2
    =
    n\sum_{k\in A_s}
    \frac{\abs{\widehat f(k)}^2}{\lambda_k}
    \lesssim
    ns^2.
\end{align}
By Lemma \ref{lem:multiplier}, there exists a vector field
$\widetilde g_f$ with Fourier support in $A_s$ such that
\begin{align*}
    g_f
    &=
    \eta_s*\widetilde g_f,\\
    \norm{\widetilde g_f}_{L^2_{\omega,x}}
    &\leq
    2\norm{g_f}_{L^2_{\omega,x}}
    \lesssim
    \sqrt n\,s.
\end{align*}
By the response identity \eqref{eq:order-r-response},
\begin{align}
\label{eq:pair-response-current}
    \Cov(W_{n,d},S_f)
    =
    \frac2n \E\sum_{i=1}^n d_i\cdot g_f(X_i)
    =
    2\E\int_{\T^d}g_f\cdot\dd \mathfrak{j},
\end{align}
where $\mathfrak{j}$ was the vector measure defined in equation~\eqref{eq: jVecMeas}. Since $\eta_s$ is even and $J_s=\eta_s*\mathfrak{j}$, we obtain
\begin{align}\label{eq:smoothed-pair-response}
    \Cov(W_{n,d},S_f)
    &=
    2\E\int_{\T^d}
    J_s\cdot\widetilde g_f\dd\mu\notag\\
    &=
    2\E\int_{\T^d}
    E_s\cdot\widetilde g_f\dd\mu
    -
    2\E\int_{\T^d}
    (\rho_s-1)\nabla u_s\cdot\widetilde g_f\dd\mu\notag\\
    &\qquad
    -
    2\E\int_{\T^d}
    \nabla u_s\cdot\widetilde g_f\dd\mu.
\end{align}
Define 
\begin{equation}
\label{eq: SkDef}
  S_k = \sum_{i=1}^{n} e_k(X_i).
\end{equation} 
Then,
\begin{align*}
    -2\E\int_{\T^d}
    \nabla u_s\cdot\widetilde g_f\dd\mu
    &=
    -2\sum_{k\neq0}
    \E\left[
        \widehat{\nabla u_s}(k)\cdot
        \overline{\widehat{\widetilde g_f}(k)}
    \right]\\
    &=
    -2\sum_{k\neq0}
    \E\left[
        -\frac{2\pi ik}{\lambda_k}
        \widehat{\eta_s}(k)\frac{S_{-k}}{n}
        \cdot
        \overline{
            \frac{2\pi ik}{\lambda_k}
            \frac{\widehat f(k)}{\widehat{\eta_s}(k)}
            S_{-k}
        }
    \right]\\
    &=
    2\sum_{k\neq0}
    \frac{\widehat f(k)}{\lambda_k}
    \frac{\E\abs{S_k}^2}{n}\\
    &=
    2\sum_{k\neq0}
    \frac{\widehat f(k)}{\lambda_k}.
\end{align*}
Furthermore,
\begin{align*}
    \Cov(U_n,S_f)
    &=
    \frac{4}{n^2}
    \sum_{1\leq i<j\leq n}
    \E\left[
        G(X_i-X_j)f(X_i-X_j)
    \right]\\
    &=
    \frac{4}{n^2}\binom{n}{2}\langle G,f\rangle\\
    &=
    2\frac{n-1}{n}
    \sum_{k\neq0}\frac{\widehat f(k)}{\lambda_k}.
\end{align*}
Together with \eqref{eq:smoothed-pair-response}, we have
\begin{align}\label{eq:2_cov_error}
    \Cov(W_{n,d}-U_{n},S_f) &= 2\E\int_{\T^d}
    E_s\cdot\widetilde g_f\dd\mu
    -
    2\E\int_{\T^d}
    (\rho_s-1)\nabla u_s\cdot\widetilde g_f\dd\mu\notag\\
    &\qquad
    +
    \frac{2}{n}
    \sum_{k\neq0}\frac{\widehat f(k)}{\lambda_k}.
\end{align}
It remains to bound the three terms in
\eqref{eq:2_cov_error}. By Cauchy--Schwarz, \eqref{eq:gf_bound}, and Proposition \ref{prop:GH},
\begin{align*}
    \left|
        \E\int_{\T^d}
        E_s\cdot\widetilde g_f\dd\mu
    \right|
    &\leq
    \norm{E_s}_{L^2_{\omega,x}}
    \norm{\widetilde g_f}_{L^2_{\omega,x}}\\
    &\lesssim
    \begin{cases}
        L^{-1}\log(2+Ls),
        &d=2,\\
        L^{-1/2},
        &d=3.
    \end{cases}
\end{align*}

By H\"older's inequality,
\begin{align*}
    \left|
        \E\int_{\T^d}
        (\rho_s-1)\nabla u_s\cdot\widetilde g_f\dd\mu
    \right|
    &\leq
    \norm{\rho_s-1}_{L^4_{\omega,x}}
    \norm{\nabla u_s}_{L^4_{\omega,x}}
    \norm{\widetilde g_f}_{L^2_{\omega,x}}\\
    &\lesssim
    \begin{cases}
        \displaystyle
        L^{-1}
        \left(1+(Ls)^{-1/2}\right)
        \sqrt{\log(e/s)},
        &d=2,\\[2mm]
        L^{-3/2}s^{-1},
        &d=3.
    \end{cases}
\end{align*}

Finally, since $\supp\widehat f\subset A_s$ and
$\norm{f}_2=1$, Cauchy--Schwarz gives
\begin{align*}
    \frac1n
    \left|
        \sum_{k\neq0}
        \frac{\widehat f(k)}{\lambda_k}
    \right|
    &\leq
    \frac1n
    \left(
        \sum_{k\in A_s}\lambda_k^{-2}
    \right)^{1/2}
    \norm{f}_2\\
    &\lesssim
    n^{-1}s^{2-d/2}\\
    &=
    \begin{cases}
        n^{-1}s,
        &d=2,\\
        n^{-1}s^{1/2},
        &d=3.
    \end{cases}
\end{align*}
Combining these three estimates proves
\eqref{eq:one-band-response}.

\end{proof}

Now we state the main lemma in this subsection. Let $\mathcal F$ be the set of real, even, and centered functions
$f\in L^2(\T^d,\mu)$ with $\norm{f}_2=1$.

\begin{lemma}
\label{lem:uniform-pairwise-response}
For $d\in\{2,3\}$, as $n\to \infty$
\begin{align}\label{eq:uniform-pairwise-response}
    \sup_{f\in\mathcal F}
    \abs{\Cov(W_{n,d}-U_n,S_f)}
    \longrightarrow0.
\end{align}
Consequently, Lemma \ref{lem:pairwise-duality} gives
\begin{align*}
    \E\left[
        \abs{nP_2W_{n,d}-nU_n}^2
    \right]
    &=
    n^2\Var(P_2W_{n,d}-U_n)\\
    &=
    \frac{n}{2(n-1)}
    \sup_{f\in\mathcal F}
    \abs{\Cov(W_{n,d}-U_n,S_f)}^2
    \longrightarrow0.
\end{align*}
\end{lemma}
\begin{proof}
Fix the cutoff $\Lambda=n^{2\alpha}$, where
\begin{align*}
    0<\alpha<\frac12,
    &\qquad d=2,\\
    \frac16<\alpha<\frac13,
    &\qquad d=3.
\end{align*}
Using the standard Fourier cutoff, write
\begin{align*}
    \widehat{f_{\leq\Lambda}}(k)
    :=
    \mathbf 1_{\{\lambda_k\leq\Lambda\}}\widehat f(k),
    \qquad
    f_{>\Lambda}:=f-f_{\leq\Lambda}.
\end{align*}
Then,
\begin{align}\label{eq:low-high-response}
    \Cov(W_{n,d}-U_n,S_f)
    &=
    \Cov(W_{n,d}-U_n,S_{f_{\leq\Lambda}})
    +
    \Cov(W_{n,d}-U_n,S_{f_{>\Lambda}}).
\end{align}

We first control the low-frequency term. Let
\begin{align*}
    \mathcal B_j
    :=
    \left\{
        k\in\Z^d:
        2^j\leq\abs{k}<2^{j+1},
        \quad
        \lambda_k\leq\Lambda
    \right\},
    \qquad
    0\leq j\leq J,
\end{align*}
where $J$ is the largest index for which $\mathcal B_J$ is nonempty, and
let $f_j$ be the Fourier projection of $f$ onto $\mathcal B_j$. Then
\begin{align*}
    f_{\leq\Lambda}
    =
    \sum_{j=0}^Jf_j,
    \qquad
    \sum_{j=0}^J\norm{f_j}_2^2
    =
    \norm{f_{\leq\Lambda}}_2^2
    \leq1.
\end{align*}
Moreover,
\begin{align*}
    2^J\asymp\sqrt{\Lambda}=n^\alpha,
    \qquad
    J\lesssim\log n,
    \qquad
    \mathcal B_j\subseteq A_{2^{-j}}.
\end{align*}
Since
\begin{align*}
    2^{-J}
    \asymp
    \Lambda^{-1/2}
    =
    n^{-\alpha}
    \geq
    L^{-1},
\end{align*}
Lemma \ref{lem:one-band-response} applies on every $\mathcal B_j$.

By homogeneity of Lemma \ref{lem:one-band-response} and Cauchy--Schwarz in $j$, for
$d=2$ we obtain
\begin{align*}
    \abs{
        \Cov(W_{n,2}-U_n,S_{f_{\leq\Lambda}})
    }^2
    &\lesssim
    L^{-2}
    \sum_{j=0}^J
    \left[
        \log^2(2+L2^{-j})
        +
        \left(1+L^{-1}2^j\right)(1+j)
    \right]
    +
    n^{-2}\sum_{j=0}^J4^{-j}\\
    &\lesssim
    n^{-1}(\log n)^3
    \longrightarrow0.
\end{align*}
Here we used $L=n^{1/2}$, $J\lesssim\log n$, and
\begin{align*}
    L^{-1}2^J
    \lesssim
    L^{-1}\sqrt{\Lambda}
    =
    n^{-1/2+\alpha}
    \lesssim1.
\end{align*}
For $d=3$, the same argument gives
\begin{align*}
    \abs{
        \Cov(W_{n,3}-U_n,S_{f_{\leq\Lambda}})
    }^2
    &\lesssim
    \sum_{j=0}^J
    \left(
        L^{-1}
        +
        L^{-3}4^j
        +
        n^{-2}2^{-j}
    \right)\\
    &\lesssim
    L^{-1}(J+1)
    +
    L^{-3}4^J
    +
    n^{-2}\\
    &\lesssim
    n^{-1/3}\log n
    +
    n^{-1}\Lambda
    +
    n^{-2}
    \longrightarrow0.
\end{align*}

We next control the high-frequency term. Since
$e_k(X_i-X_j)$ has product eigenvalue $2\lambda_k$, the statistic
$S_{f_{>\Lambda}}$ has product spectrum above $2\Lambda$. Therefore, by Lemma \ref{lem:high-frequency}
\begin{align}\label{eq:high-W-response}
    \abs{\Cov(W_{n,d},S_{f_{>\Lambda}})}^2
    &\lesssim
    \begin{cases}
        (\log n)\Lambda^{-1},
        &d=2,\\
        n^{1/3}\Lambda^{-1},
        &d=3.
    \end{cases}
\end{align}

For the Green statistic,
\begin{align*}
    \Cov(U_n,S_{f_{>\Lambda}})
    =
    2\frac{n-1}{n}
    \sum_{\lambda_k>\Lambda}
    \frac{\widehat f(k)}{\lambda_k}.
\end{align*}
Hence, by Cauchy--Schwarz and Appendix \ref{app:lattice},
\begin{align}\label{eq:high-U-response}
    \abs{\Cov(U_n,S_{f_{>\Lambda}})}^2
    &\lesssim
    \left(
        \sum_{\lambda_k>\Lambda}
        \frac{1}{\lambda_k^2}
    \right)
    \norm{f_{>\Lambda}}_2^2\notag\\
    &\lesssim
    \begin{cases}
        \Lambda^{-1},
        &d=2,\\
        \Lambda^{-1/2},
        &d=3.
    \end{cases}
\end{align}
Combining \eqref{eq:high-W-response} and
\eqref{eq:high-U-response}, we obtain
\begin{align*}
    \abs{
        \Cov(W_{n,d}-U_n,S_{f_{>\Lambda}})
    }^2
    \lesssim
    \begin{cases}
        (\log n)\Lambda^{-1}+\Lambda^{-1},
        &d=2,\\
        n^{1/3}\Lambda^{-1}+\Lambda^{-1/2},
        &d=3,
    \end{cases}
    \longrightarrow0.
\end{align*}

Together with the low-frequency estimates and
\eqref{eq:low-high-response}, this proves the lemma.
\end{proof}

\subsection{Higher-order Hoeffding components ($d=2,3$)}
\label{ssec:higher23}
We now prove that $\sum_{r\geq 3} P_rW_{n,d}$ is negligible.
\begin{lemma}\label{lem:higher-orders}
For $d\in\{2,3\}$,
\begin{align}\label{eq:higher-orders}
    \E\left[
        \abs{
            n\sum_{r=3}^nP_rW_{n,d}
        }^2
    \right]
    \longrightarrow 0, \qquad \text{ as } n \to \infty.
\end{align}
\end{lemma}

\begin{proof}
Fix $\Lambda=n^{2\alpha}$, where
\begin{align*}
    0<\alpha<\frac12,
    &\qquad d=2,\\
    \frac16<\alpha<\frac29,
    &\qquad d=3.
\end{align*}
By orthogonality of the spectral projections,
\begin{align}\label{eq:higher-order-spectral-split}
    \E\left[
        \left|
            n\sum_{r=3}^nP_rW_{n,d}
        \right|^2
    \right]
    &=
    n^2\E\left[
        \left|
            \Pi_{\leq\Lambda}
            \sum_{r=3}^nP_rW_{n,d}
        \right|^2
    \right]
    +
    n^2\E\left[
        \left|
            \Pi_{>\Lambda}
            \sum_{r=3}^nP_rW_{n,d}
        \right|^2
    \right].
\end{align}
By Lemma \ref{lem:high-frequency},
\begin{align*}
    n^2\E\left[
        \left|
            \Pi_{>\Lambda}
            \sum_{r=3}^nP_rW_{n,d}
        \right|^2
    \right]
    \lesssim
    \begin{cases}
        \Lambda^{-1}\log n,
        &d=2,\\
        \Lambda^{-1}n^{1/3},
        &d=3.
    \end{cases}
    \longrightarrow0
\end{align*}
Therefore, it remains to prove
\begin{align}\label{eq:higher-order-low-goal}
    n^2\E\left[
        \left|
            \Pi_{\leq\Lambda}
            \sum_{r=3}^nP_rW_{n,d}
        \right|^2
    \right]
    \longrightarrow0.
\end{align}
To estimate the low-frequency part of $P_rW_{n,d}$ through the response
functional in Lemma \ref{lem:order-r-response}, we will test against symmetric, canonical functions of unit
$L^2$ norm whose product spectrum lies below $\Lambda$. We denote this
class by
\begin{align*}
    \mathcal C_{r,\Lambda}
    :=
    \left\{
        a:(\T^d)^r\to\R:
        \begin{array}{l}
            a\text{ is symmetric and canonical},\\
            \norm{a}_2=1,\quad
            \Pi_{\leq\Lambda}a=a
        \end{array}
    \right\}.
\end{align*}
By Lemma \ref{lem:order-r-response},
\begin{align*}
    n^2\E\left[
        \left|
            \Pi_{\leq\Lambda}
            \sum_{r=3}^nP_rW_{n,d}
        \right|^2
    \right]
    =
    n^2\sum_{r=3}^n
    \frac{1}{r!(n)_r}
    \sup_{a\in\mathcal C_{r,\Lambda}}
    \abs{\mathcal R_r(a)}^2.
\end{align*}
Since $a$ is canonical, the smallest eigenvalue is at least $4\pi^2 r$. Therefore, taking $R:=\left\lfloor\frac{\Lambda}{4\pi^2}\right\rfloor$, we have
\begin{align}\label{eq:higher-order-low-response}
    n^2\E\left[
        \left|
            \Pi_{\leq\Lambda}
            \sum_{r=3}^nP_rW_{n,d}
        \right|^2
    \right]
    =
    n^2\sum_{r=3}^R
    \frac{1}{r!(n)_r}
    \sup_{a\in\mathcal C_{r,\Lambda}}
    \abs{\mathcal R_r(a)}^2.
\end{align}
Now fix $3\leq r\leq R$ and $a\in\mathcal C_{r,\Lambda}$. Define
\(
    g:=\nabla_1(-\Delta_1)^{-1}a.
\) as previously.
By \eqref{eq:order-r-response},
\begin{align*}
    \mathcal R_r(a)
    =
    \frac{2}{n}\E\sum_{i=1}^n
    d_i\cdot G_i^\sharp(X_i).
\end{align*}
To replace the label-dependent field $G_i^\sharp$ by a common field, define
\begin{align*}
    G^\sharp(x)
    :=
    \sum_{\substack{i_2,\ldots,i_r\in[n]\\
    \text{pairwise distinct}}}
    g(x,X_{i_2},\ldots,X_{i_r}).
\end{align*}
Then
\begin{align}
\label{eq:higher-order-response-split}
    \mathcal R_r(a)
    =
    \frac{2}{n}\E\sum_{i=1}^n
    d_i\cdot G^\sharp(X_i)+
    \frac{2}{n}\E\sum_{i=1}^n
    d_i\cdot
    \left(G_i^\sharp-G^\sharp\right)(X_i).
\end{align}

\paragraph{The common-field term.}
We first control the first term in
\eqref{eq:higher-order-response-split}, following the argument of
Lemma \ref{lem:uniform-pairwise-response}.
Decompose $a$ in its first Fourier variable using the disjoint bands
$\mathcal B_j$ introduced in the proof of
\eqref{eq:uniform-pairwise-response}, with the present cutoff $\Lambda$.
Let $s_j=2^{-j}$ and let $J$ be the largest index for which
$\mathcal B_J\neq\varnothing$. Then $J\lesssim\log n$.

For $0\leq j\leq J$, define
\begin{align*}
    \widehat{a_j}(k_1,\ldots,k_r)
    &:=
    \mathbf 1_{\{k_1\in\mathcal B_j\}}
    \widehat a(k_1,\ldots,k_r),\\
    g_j
    &:=
    \nabla_1(-\Delta_1)^{-1}a_j,\\
    (G^\sharp)_j(x)
    &:=
    \sum_{\substack{i_2,\ldots,i_r\in[n]\\
                    \text{pairwise distinct}}}
    g_j(x,X_{i_2},\ldots,X_{i_r}).
\end{align*}
Since $a$ is canonical and has product spectrum at most $\Lambda$,
these bands cover its first-coordinate Fourier support.
Their disjointness gives
\begin{align*}
    a&=\sum_{j=0}^J a_j,
    &g&=\sum_{j=0}^J g_j,
    &G^\sharp&=\sum_{j=0}^J(G^\sharp)_j,
    &\sum_{j=0}^J\norm{a_j}_2^2&=\norm a_2^2=1.
\end{align*}

We first bound the second moment for $(G^\sharp)_j$. Expanding the square gives
\begin{align*}
    \E\int_{\T^d}
    \abs{(G^\sharp)_j(x)}^2\dd\mu(x)
    =
    \sum_{I,I'}
    \E\int_{\T^d}
    g_j(x,X_I)\cdot g_j(x,X_{I'})\dd\mu(x),
\end{align*}
where $I$ and $I'$ range over ordered tuples of $r-1$ distinct labels and
$X_I=(X_{i_2},\ldots,X_{i_r})$. If $I$ and $I'$ contain different sets of
labels, by canonicity, the integral is $0$. Thus a term survives only when $I'$ is
a permutation of $I$. Since there are $(n)_{r-1}$ choices of $I$ and
$(r-1)!$ such permutations, symmetry in the last $r-1$ variables gives
\begin{align}\label{eq:common-field-moment}
    \E\int_{\T^d}
    \abs{(G^\sharp)_j(x)}^2\dd\mu(x)
    =
    (r-1)!(n)_{r-1}\norm{g_j}_2^2.
\end{align}
Moreover,
\begin{align*}
    \norm{g_j}_2^2
    =
    \sum_{k_1,\ldots,k_r}
    \frac{1}{\lambda_{k_1}}
    \abs{\widehat{a_j}(k_1,\ldots,k_r)}^2
    \lesssim
    s_j^2\norm{a_j}_2^2,
\end{align*}
because $k_1\in A_{s_j}$. Consequently,
\begin{align}\label{eq:common-field-band-bound}
    \E\int_{\T^d}
    \abs{(G^\sharp)_j(x)}^2\dd\mu(x)
    \lesssim
    s_j^2(r-1)!(n)_{r-1}\norm{a_j}_2^2.
\end{align}
By Lemma \ref{lem:multiplier}, there exists a random vector field
$(\widetilde G^\sharp)_j$ such that
\begin{align*}
    (G^\sharp)_j
    =
    \eta_{s_j}*(\widetilde G^\sharp)_j
\end{align*}
and
\begin{align}\label{eq:common-field-smoothed-bound}
    \E\int_{\T^d}
    \abs{(\widetilde G^\sharp)_j(x)}^2\dd\mu(x)
    \lesssim
    s_j^2(r-1)!(n)_{r-1}\norm{a_j}_2^2.
\end{align}
By the definitions of $\rho_s,u_s,J_s,E_s$ above Remark \ref{displacement_remark} in
Section \ref{ssec:GH}, we have
\begin{align}\label{eq:common-field-decomposition}
    \frac{2}{n}\E\sum_{i=1}^n
    d_i\cdot(G^\sharp)_j(X_i)
    &=
    2\E\int_{\T^d}
    J_{s_j}\cdot(\widetilde G^\sharp)_j\dd\mu\notag\\
    &=
    2\E\int_{\T^d}
    E_{s_j}\cdot(\widetilde G^\sharp)_j\dd\mu\notag\\
    &\quad-
    2\E\int_{\T^d}
    \nabla u_{s_j}\cdot(\widetilde G^\sharp)_j\dd\mu\notag\\
    &\quad-
    2\E\int_{\T^d}
    (\rho_{s_j}-1)\nabla u_{s_j}
    \cdot(\widetilde G^\sharp)_j\dd\mu.
\end{align}

Recall the proof for Lemma \ref{lem:uniform-pairwise-response}, the term
\begin{align*}
    \E\int_{\T^d}
    \nabla u_{s_j}\cdot(\widetilde G^\sharp)_j\dd\mu
\end{align*}
is the leading contribution, and it is cancelled by
$\Cov(U_n,S_f)$. The key simplification for higher orders is that this term
vanishes by itself. More precisely, for every $r\geq3$,
\begin{align*}
    \E\int_{\T^d}
    \nabla u_{s_j}\cdot(\widetilde G^\sharp)_j\dd\mu
    =
    0.
\end{align*}
The argument is simple from Hoeffding's perspective. For each fixed $x\in\T^d$, we have
\begin{align*}
    \nabla u_{s_j}(x)
    =
    \frac1n\sum_{\ell=1}^n
    K_{s_j}(x-X_\ell),
\end{align*}
where $K_{s_j}$ is a deterministic centered vector field. Hence
$\nabla u_{s_j}(x)$ belongs to the first Hoeffding subspace
$\mathcal H_1$.

On the other hand, $(\widetilde G^\sharp)_j$ has the form
\begin{align*}
    (\widetilde G^\sharp)_j(x)
    =
    \sum_{\substack{i_2,\ldots,i_r\in[n]\\
    \text{pairwise distinct}}}
    \widetilde g_j(x,X_{i_2},\ldots,X_{i_r}).
\end{align*}
The kernel $\widetilde g_j = (\eta_{s_j})^{-1} * (\nabla_1(-\Delta_1)^{-1}a_j)$ is canonical in each of its last $r-1$
variables because $a_j$ is canonical. Therefore, for each fixed $x$,
$(\widetilde G^\sharp)_j(x)$ belongs to the Hoeffding subspace
$\mathcal H_{r-1}$. Since $r\geq3$, Hoeffding orthogonality from
Appendix \ref{app:hoeffding} gives
\begin{align*}
    \E\left[
        \nabla u_{s_j}(x)\cdot
        (\widetilde G^\sharp)_j(x)
    \right]
    =
    0.
\end{align*}
Integrating over $x$ proves
\begin{align*}
    \E\int_{\T^d}
    \nabla u_{s_j}\cdot
    (\widetilde G^\sharp)_j\dd\mu
    =
    0.
\end{align*}
We now bound the remaining two terms in
\eqref{eq:common-field-decomposition} using Proposition \ref{prop:GH}.
By Cauchy--Schwarz and \eqref{eq:common-field-smoothed-bound},
\begin{align*}
   \left|
        \E\int_{\T^d}
        E_{s_j}\cdot(\widetilde G^\sharp)_j\dd\mu
    \right|
    &\leq
    \left(
        \E\int_{\T^d}
        \abs{E_{s_j}}^2\dd\mu
    \right)^{1/2}
    \left(
        \E\int_{\T^d}
        \abs{(\widetilde G^\sharp)_j}^2\dd\mu
    \right)^{1/2}\\
    &\lesssim
    \sqrt{(r-1)!(n)_{r-1}}\norm{a_j}_2
    \begin{cases}
        L^{-2}\log(2+Ls_j),
        &d=2,\\
        L^{-2},
        &d=3.
    \end{cases}
\end{align*}
Similarly, by H\"older's inequality,
\begin{align*}
    &\left|
        \E\int_{\T^d}
        (\rho_{s_j}-1)\nabla u_{s_j}
        \cdot(\widetilde G^\sharp)_j\dd\mu
    \right|\\
    &\qquad\leq
    \left(
        \E\int_{\T^d}
        \abs{\rho_{s_j}-1}^4\dd\mu
    \right)^{1/4}
    \left(
        \E\int_{\T^d}
        \abs{\nabla u_{s_j}}^4\dd\mu
    \right)^{1/4}
    \left(
        \E\int_{\T^d}
        \abs{(\widetilde G^\sharp)_j}^2\dd\mu
    \right)^{1/2}\\
    &\qquad\lesssim
    \sqrt{(r-1)!(n)_{r-1}}\norm{a_j}_2
    \begin{cases}
        L^{-2}\sqrt{\log(e/s_j)},
        &d=2,\\
        L^{-3}s_j^{-1},
        &d=3.
    \end{cases}
\end{align*}
Since the middle term in
\eqref{eq:common-field-decomposition} vanishes, we obtain
\begin{align*}
    \left|
        \frac{2}{n}\E\sum_{i=1}^n
        d_i\cdot(G^\sharp)_j(X_i)
    \right|
    \lesssim
    \sqrt{(r-1)!(n)_{r-1}}\norm{a_j}_2
    \begin{cases}
        n^{-1}\log n,
        &d=2,\\
        n^{-2/3},
        &d=3.
    \end{cases}
\end{align*}
Here we used $L=n^{1/d}$ and $s_j^{-1}\lesssim\sqrt{\Lambda}$. In
dimension three, the assumption $\Lambda\leq n^{4/9}$ gives
\begin{align*}
    L^{-3}s_j^{-1}
    \lesssim
    n^{-1}\sqrt{\Lambda}
    \lesssim
    n^{-7/9}
    \lesssim
    n^{-2/3}.
\end{align*}

Finally, since $G^\sharp=\sum_{j=0}^J(G^\sharp)_j$ and the functions
$a_j$ are orthogonal,
\begin{align*}
    \sum_{j=0}^J\norm{a_j}_2
    &\leq
    \sqrt{J+1}
    \left(
        \sum_{j=0}^J\norm{a_j}_2^2
    \right)^{1/2}\\
    &=
    \sqrt{J+1}
    \lesssim
    \sqrt{\log n}.
\end{align*}
Summing over the dyadic bands gives
\begin{align*}
    \left|
        \frac{2}{n}\E\sum_{i=1}^n
        d_i\cdot G^\sharp(X_i)
    \right|^2
    \lesssim
    (r-1)!(n)_{r-1}
    \begin{cases}
        n^{-2}(\log n)^3,
        &d=2,\\
        n^{-4/3}\log n,
        &d=3.
    \end{cases}
\end{align*}

\paragraph{The excluded-label term.}
We now turn to the second term in
\eqref{eq:higher-order-response-split}. The difference
$G^\sharp-G_i^\sharp$ consists of the summands in which one of the labels
$i_2,\ldots,i_r$ equals $i$. Since these labels are pairwise distinct, the
label $i$ occurs in exactly one position.

For $b\in\{2,\ldots,r\}$, define the diagonal kernel
$\operatorname{Diag}_b g:(\T^d)^{r-1}\to\R^d$ by
\begin{align*}
    (\operatorname{Diag}_b g)
    (x_1,\ldots,x_{b-1},x_{b+1},\ldots,x_r)
    :=
    g(x_1,\ldots,x_{b-1},x_1,x_{b+1},\ldots,x_r).
\end{align*}
We first prove the diagonal estimate
\begin{align}\label{eq:diagonal-trace}
    \norm{\operatorname{Diag}_b g}_2
    \lesssim
    \norm{a}_2
    \begin{cases}
        \sqrt{\log(2+\Lambda)},
        &d=2,\\
        \Lambda^{1/4},
        &d=3.
    \end{cases}
\end{align}
Recall that
\begin{align*}
    \widehat g(k_1,\ldots,k_r)
    =
    \frac{2\pi i k_1}{\lambda_{k_1}}
    \widehat a(k_1,\ldots,k_r).
\end{align*}
Identifying $x_b$ with $x_1$ adds the corresponding Fourier frequencies.
Writing
\begin{align*}
    q:=k_1+k_b,
    \qquad
    \mathbf k'
    :=
    (k_2,\ldots,k_{b-1},k_{b+1},\ldots,k_r),
\end{align*}
we obtain
\begin{align*}
    \widehat{\operatorname{Diag}_b g}(q,\mathbf k')
    =
    \sum_{k_1+k_b=q}
    \frac{2\pi i k_1}{\lambda_{k_1}}
    \widehat a(k_1,\ldots,k_r).
\end{align*}
Since $a$ has product spectrum below $\Lambda$, every nonzero term in this
sum satisfies $\abs{k_1}\lesssim\sqrt{\Lambda}$. Therefore, by
Cauchy--Schwarz,
\begin{align*}
    \abs{
        \widehat{\operatorname{Diag}_b g}(q,\mathbf k')
    }^2
    &\lesssim
    \left(
        \sum_{0<\abs{k_1}\lesssim\sqrt{\Lambda}}
        \frac{1}{\abs{k_1}^2}
    \right)
    \sum_{k_1+k_b=q}
    \abs{\widehat a(k_1,\ldots,k_r)}^2.
\end{align*}
Summing over $q$ and $\mathbf k'$, each Fourier coefficient of $a$ is
counted once. Hence, by Parseval and Appendix \ref{app:lattice},
\begin{align*}
    \norm{\operatorname{Diag}_b g}_2^2
    &\lesssim
    \left(
        \sum_{0<\abs{k_1}\lesssim\sqrt{\Lambda}}
        \frac{1}{\abs{k_1}^2}
    \right)\norm{a}_2^2\\
    &\lesssim
    \norm{a}_2^2
    \begin{cases}
        \log(2+\Lambda),&d=2,\\
        \Lambda^{1/2},&d=3.
    \end{cases}
\end{align*}
Taking square roots proves \eqref{eq:diagonal-trace}.

For $b\in\{2,\ldots,r\}$ and $i\in[n]$, define
\begin{align*}
    Q_{b,i}(x)
    :=
    \sum_{\substack{
        i_2,\ldots,i_{b-1},i_{b+1},\ldots,i_r
        \in[n]\setminus\{i\}\\
        \text{pairwise distinct}
    }}
    (\operatorname{Diag}_b g)
    \left(
        x,X_{i_2},\ldots,X_{i_{b-1}},
        X_{i_{b+1}},\ldots,X_{i_r}
    \right).
\end{align*}
Then
\begin{align*}
    \left(G^\sharp-G_i^\sharp\right)(X_i)
    =
    \sum_{b=2}^rQ_{b,i}(X_i).
\end{align*}

Note that
\begin{align*}
    \E\left[
        \abs{Q_{b,i}(X_i)}^2
        \,\middle|\,X_i
    \right]
    &=
    (r-2)!(n-1)_{r-2}\\
    &\quad\times
    \int_{(\T^d)^{r-2}}
    \abs{
        (\operatorname{Diag}_b g)(X_i,\mathbf y)
    }^2
    \dd\mu^{\otimes(r-2)}(\mathbf y).
\end{align*}
Averaging over $X_i$ and summing over $i$ gives
\begin{align*}
    \sum_{i=1}^n
    \E\left[
        \abs{Q_{b,i}(X_i)}^2
    \right]
    =
    n(r-2)!(n-1)_{r-2}
    \norm{\operatorname{Diag}_b g}_2^2.
\end{align*}
By Cauchy--Schwarz and
$\sum_{i=1}^n\E\abs{d_i}^2\leq n\E W_{n,d}$,
\begin{align*}
    \left|
        \frac{2}{n}
        \E\sum_{i=1}^n
        d_i\cdot Q_{b,i}(X_i)
    \right|
    &\leq
    \frac{2}{n}
    \left(
        \E\sum_{i=1}^n\abs{d_i}^2
    \right)^{1/2}
    \left(
        \E\sum_{i=1}^n
        \abs{Q_{b,i}(X_i)}^2
    \right)^{1/2}\\
    &\lesssim
    \sqrt{(r-2)!(n-1)_{r-2}}\norm{a}_2
    \begin{cases}
        n^{-1/2}
        \sqrt{\log n\,\log(2+\Lambda)},
        &d=2,\\
        n^{-1/3}\Lambda^{1/4},
        &d=3.
    \end{cases}
\end{align*}
Summing over $b=2,\ldots,r$ and using $\norm{a}_2=1$ gives
\begin{align*}
    &\left|
        \frac{2}{n}\E\sum_{i=1}^n
        d_i\cdot
        \left(G_i^\sharp-G^\sharp\right)(X_i)
    \right|^2\\
    &\qquad\lesssim
    r^2(r-2)!(n-1)_{r-2}
    \begin{cases}
        n^{-1}\log n\,\log(2+\Lambda),
        &d=2,\\
        n^{-2/3}\Lambda^{1/2},
        &d=3
    \end{cases}\\
    &\qquad\lesssim
    r^2(r-2)!(n-1)_{r-2}
    \begin{cases}
        n^{-1}(\log n)^2,
        &d=2,\\
        n^{-4/9},
        &d=3.
    \end{cases}
\end{align*}

Combining the two estimates in
\eqref{eq:higher-order-response-split}, we obtain
\begin{align}\label{eq:higher-order-response-bound}
    \abs{\mathcal R_r(a)}^2
    \lesssim
    (r-2)!(n-1)_{r-2}
    \begin{cases}
        (r-1)n^{-1}(\log n)^3
        +r^2n^{-1}(\log n)^2,
        &d=2,\\
        (r-1)n^{-1/3}\log n
        +r^2n^{-4/9},
        &d=3.
    \end{cases}
\end{align}

Combining \eqref{eq:higher-order-low-response} with
\eqref{eq:higher-order-response-bound}, and using
\begin{align*}
    \frac{n^2(r-2)!(n-1)_{r-2}}{r!(n)_r}
    =
    \frac{n}{r(r-1)(n-r+1)}
    \lesssim \frac{1}{r(r-1)},
\end{align*}
since $r \leq R =o(n)$, we obtain the following bounds.
For $d=2$, we have
\begin{align*}
    &n^2\E\left[
        \left|
            \Pi_{\leq\Lambda}
            \sum_{r=3}^nP_rW_{n,2}
        \right|^2
    \right]\\
    &\qquad\lesssim
    n^{-1}(\log n)^3
    \sum_{r=3}^R\frac1r
    +
    n^{-1}(\log n)^2
    \sum_{r=3}^R\frac{r}{r-1}\\
    &\qquad\lesssim
    n^{-1}(\log n)^4
    +
    n^{-1}R(\log n)^2\\
    &\qquad\lesssim
    n^{-1}(\log n)^4
    +
    n^{-1+2\alpha}(\log n)^2
    \longrightarrow0,
\end{align*}
where we used $R\lesssim n^{2\alpha}$ and $\alpha<1/2$.

For $d=3$, we similarly have
\begin{align*}
    &n^2\E\left[
        \left|
            \Pi_{\leq\Lambda}
            \sum_{r=3}^nP_rW_{n,3}
        \right|^2
    \right]\\
    &\qquad\lesssim
    n^{-1/3}\log n
    \sum_{r=3}^R\frac1r
    +
    n^{-4/9}
    \sum_{r=3}^R\frac{r}{r-1}\\
    &\qquad\lesssim
    n^{-1/3}(\log n)^2
    +
    Rn^{-4/9}\\
    &\qquad\lesssim
    n^{-1/3}(\log n)^2
    +
    n^{2\alpha-4/9}
    \longrightarrow0,
\end{align*}
where we used $R\lesssim n^{2\alpha}$ and $\alpha<2/9$.

This proves \eqref{eq:higher-order-low-goal}. Combining it with the
high-frequency estimate in \eqref{eq:higher-order-spectral-split} proves the lemma.

\end{proof}

\subsection{The critical case: the error current is a divergence ($d=4$)}
\label{ssec:stress4}

In this section, the main objective is to establish the estimate in Lemma~\ref{lem:stress-square-function}, which controls the difference between the actual transport current and the solution to the linear approximation. 

The first objective is to improve previous arguments to make them hold at smaller scales. 
Indeed, the cutoff argument used in dimensions two and three becomes critical in
dimension four. 
Recall that we take $\Lambda=n^{2\alpha}$. In dimension two,
we may choose
\begin{align*}
    0<\alpha<\frac12,
\end{align*}
while in dimension three we may choose
\begin{align*}
    \frac16<\alpha<\frac29.
\end{align*}
Thus the admissible interval has width $1/2$ in dimension two and only
$1/18$ in dimension three. In dimension four, we work at the cutoff
\begin{align*}
    \Lambda=L^2=n^{1/2}.
\end{align*}
At this cutoff, Lemma \ref{lem:high-frequency} bounds the
high-frequency contribution to the variance by $O(n^{-2})$,
which is negligible compared with the target scale
$n^{-2}\log n$. The remaining low-frequency terms require
a sharper argument than the estimates used above.
We first express the error current as a divergence, and then
use this representation to control all frequency scales together.

Let $\pi_n$ be an optimal coupling between $\mu_n$ and $\mu$. Recall from
the notation section that $\df_x(y)$ denotes the source-minus-target
displacement from $y$ to $x$. For $\pi_n$-almost every $(x,y)$, let
\begin{align*}
    \gamma_{x,y}(t)
    :=
    y+t\,\df_x(y)
    \pmod{\Z^4},
    \qquad 0\leq t\leq1,
\end{align*}
be the minimizing segment from $y$ to $x$.

Recall that the current $\mathfrak{j} = \frac{1}{n}\sum_{i}d_i \delta_{X_i}$ records the displacement at the sample point:
\begin{align*}
    \langle \mathfrak{j},\varphi\rangle
    :=
    \int_{\T^4\times\T^4}
    \df_x(y)\cdot\varphi(x)\dd\pi_n(x,y).
\end{align*}
We also define the current $\mathfrak{j}^\flat$ by
\begin{align*}
    \langle \mathfrak{j}^\flat,\varphi\rangle
    :=
    \int_{\T^4\times\T^4}
    \df_x(y)\cdot\varphi(y)\dd\pi_n(x,y).
\end{align*}
Thus $\mathfrak{j}$ and $\mathfrak{j}^\flat$ record the same displacement at the two endpoints of
the transport segment.

Define the matrix-valued measure $S_1$ through its action on smooth
matrix-valued test functions $\Phi:\T^4\to\R^{4\times4}$ by
\begin{align*}
    \langle S_1,\Phi\rangle
    :=
    \int_0^1
    \int_{\T^4\times\T^4}
    \df_x(y)^\top
    \Phi\bigl(\gamma_{x,y}(t)\bigr)
    \df_x(y)
    \dd\pi_n(x,y)\dd t.
\end{align*}
Equivalently, for every Borel set $A\subseteq\T^4$, $S_1$ is defined by
\begin{align*}
    S_1(A)
    :=
    \int_0^1
    \int_{\T^4\times\T^4}
    \mathbf 1_{\{\gamma_{x,y}(t)\in A\}}
    \df_x(y)\df_x(y)^\top
    \dd\pi_n(x,y)\dd t.
\end{align*}

For a matrix-valued measure $S$, define its divergence by
\begin{align*}
    \langle\operatorname{div}S,\varphi\rangle
    :=
    -\langle S,\nabla\varphi\rangle
\end{align*}
for every smooth vector field $\varphi:\T^4\to\R^4$, where
\begin{align*}
    (\nabla\varphi)_{ij}
    :=
    \partial_j\varphi_i,
    \qquad i,j\in[4].
\end{align*}
When $S$ is smooth, this is equivalent to
\begin{align*}
    (\operatorname{div}S)_i
    =
    \sum_{j=1}^4\partial_jS_{ij},
    \qquad i\in[4].
\end{align*}
For every such $\varphi$, the fundamental theorem of calculus gives
\begin{align*}
    \langle \mathfrak{j}-\mathfrak{j}^\flat,\varphi\rangle
    &=
    \int_{\T^4\times\T^4}
    \df_x(y)\cdot
    \bigl(\varphi(x)-\varphi(y)\bigr)
    \dd\pi_n(x,y)\\
    &=
    \int_0^1
    \int_{\T^4\times\T^4}
    \df_x(y)^\top
    \nabla\varphi\bigl(\gamma_{x,y}(t)\bigr)
    \df_x(y)
    \dd\pi_n(x,y)\dd t\\
    &=
    \langle S_1,\nabla\varphi\rangle\\
    &=
    \langle-\operatorname{div}S_1,\varphi\rangle.
\end{align*}
Therefore,
\begin{align}\label{eq:first-stress-identity}
    \mathfrak{j}-\mathfrak{j}^\flat=-\operatorname{div}S_1
\end{align}
as vector-valued distributions on $\T^4$. We next compare $\mathfrak{j}^\flat$ with $-\nabla u$, where $u$ is the mean-zero
solution of \( \Delta u=\mu_n-\mu.\)
Define the matrix-valued measure $S_0$ by
\begin{align*}
    \langle S_0,\Phi\rangle
    &:=
    \int_0^1(1-t)
    \int_{\T^4\times\T^4}
    \df_x(y)^\top
    \Phi\bigl(\gamma_{x,y}(t)\bigr)
    \df_x(y)
    \dd\pi_n(x,y)\dd t.
\end{align*}
Equivalently, for every Borel set $A\subseteq\T^4$, $S_0$ is defined by
\begin{align*}
    S_0(A)
    :=
    \int_0^1(1-t)
    \int_{\T^4\times\T^4}
    \mathbf 1_{\{\gamma_{x,y}(t)\in A\}}
    \df_x(y)\df_x(y)^\top
    \dd\pi_n(x,y)\dd t.
\end{align*}
For every smooth function $\psi:\T^4\to\R$, second-order Taylor expansion
along $\gamma_{x,y}$ gives
\[
    \psi(x)-\psi(y)
    =
    \df_x(y)\cdot\nabla\psi(y)
    +
    \int_0^1(1-t)
    \df_x(y)^\top
    \nabla^2\psi\bigl(\gamma_{x,y}(t)\bigr)
    \df_x(y)\dd t.
\]
Therefore,
\begin{align*}
    \langle\mu_n-\mu,\psi\rangle
    &=
    \int_{\T^4\times\T^4}
    \bigl(\psi(x)-\psi(y)\bigr)
    \dd\pi_n(x,y)\\
    &=
    \langle \mathfrak{j}^\flat,\nabla\psi\rangle
    +
    \langle S_0,\nabla^2\psi\rangle\\
    &=
    \left\langle
        -\operatorname{div}\mathfrak{j}^\flat
        +
        \operatorname{div}\operatorname{div}S_0,
        \psi
    \right\rangle.
\end{align*}
Hence,
\begin{align}\label{eq:second-stress-identity}
    \mu_n-\mu
    =
    -\operatorname{div}\mathfrak{j}^\flat
    +
    \operatorname{div}\operatorname{div}S_0.
\end{align}
We now combine \eqref{eq:first-stress-identity} and \eqref{eq:second-stress-identity}. For a vector-valued distribution $V$, define
\begin{align*}
    \mathsf QV
    :=
    \nabla\Delta^{-1}\operatorname{div}V,
\end{align*}
where the zero Fourier mode is set equal to zero. Equivalently,
\begin{align*}
    \widehat{(\mathsf QV)}_i(k)
    =
    \sum_{\ell=1}^4
    \frac{k_ik_\ell}{\abs{k}^2}
    \widehat V_\ell(k),
    \qquad k\neq0.
\end{align*}
Note that, for every function $f$ and every vector field $V$ on $\T^4$,
\begin{align*}
    \mathsf Q(\nabla f)=\nabla f,
    \qquad
    \mathsf QV=0
    \quad\text{if }\operatorname{div}V=0.
\end{align*}

Let $T:\T^4\to\T^4$ be the optimal transport map from $\mu$ to $\mu_n$.
By the characterization of the optimal map
\cite[Theorem 9]{McC01}, there exists a
Lipschitz function $\phi:\T^4\to\R$ such that
\begin{align*}
    \nabla\phi(y)
    =
    -\df_{T(y)}(y)
\end{align*}
for $\mu$-almost every $y$. Consequently,
\(
    \mathfrak{j}^\flat
    =
    -\nabla\phi\,\mu,
\)
and thus 
\(
    \mathsf Q\mathfrak{j}^\flat=\mathfrak{j}^\flat.
\)

Applying $\nabla\Delta^{-1}$ to
\eqref{eq:second-stress-identity} gives
\begin{align*}
    \nabla u
    =
    -\mathsf Q\mathfrak{j}^\flat
    +
    \mathsf Q\operatorname{div}S_0.
\end{align*}
Therefore,
\begin{align}\label{eq:target-current-error}
    \mathfrak{j}^\flat+\nabla u
    =
    \mathsf Q\operatorname{div}S_0.
\end{align}

For a matrix-valued distribution $S$, define $\mathcal KS$ by
\begin{align*}
    \widehat{(\mathcal KS)}_{ij}(k)
    &:=
    \sum_{\ell=1}^4
    \frac{k_ik_\ell}{\abs{k}^2}
    \widehat S_{\ell j}(k),
    \qquad k\neq0,\\
    \widehat{\mathcal KS}(0)&:=0.
\end{align*}
Then
\begin{align*}
    \operatorname{div}(\mathcal KS)
    =
    \mathsf Q\operatorname{div}S.
\end{align*}
Combining \eqref{eq:first-stress-identity} and
\eqref{eq:target-current-error}, we obtain
\begin{align}\label{eq:stress-identity}
    \mathfrak{j}+\nabla u
    =
    \operatorname{div}\Sigma,
    \qquad
    \Sigma:=-S_1+\mathcal KS_0.
\end{align}
This agrees with the intuition from Section \ref{ssec:intuition} that
$\mathfrak{j}\approx-\nabla u$.  The error is the divergence of
$\Sigma$. The measures $S_0,S_1$ are carried by the transport segments, but
$\mathcal K$ is nonlocal. We control $\Sigma$ after averaging at the scale
$L^{-1}=n^{-1/4}$. Let $q_a$ be the heat kernel on $\T^4$ with
\begin{equation}
\label{eq: HeatKern}
    \widehat q_a(k):=e^{-4\pi^2a^2\abs{k}^2}.
\end{equation}
Thus $q_a*\Sigma$ is the matrix field obtained by averaging $\Sigma$ over
distances of order $a$, defined componentwise by distributional duality:
\begin{align*}
    (q_a*\Sigma)_{ij}(x)
    :=
    \langle\Sigma_{ij},q_a(x-\cdot)\rangle.
\end{align*}
Its squared $L^2$ norm is
\begin{align*}
    \norm{q_a*\Sigma}_{L_x^2}^2
    :=
    \int_{\T^4}
    \abs{(q_a*\Sigma)(x)}^2
    \dd\mu(x),
\end{align*}
where $\abs{A}^2:=\sum_{i,j=1}^4\abs{A_{ij}}^2$.

With these elements at hand, we can now develop the various lemmas necessary. 
First, we establish the gradient estimate at radii of scale $L$ in Lemma~\ref{lem:terminal-gradient}. This new result can then be combined with the work of Goldman and Huesmann to prove Proposition~\ref{prop:radius-support}.

\begin{lemma}[Terminal-scale gradient]
\label{lem:terminal-gradient}
Let $d\geq3$, $L=n^{1/d}$, and let $Y_1,\ldots,Y_n$ be independent
with Haar law $m_L=L^{-d}\mathcal L_L$ on $\T_L^d$, where
$\mathcal L_L$ denotes Lebesgue measure. For the periodisation
$\eta_\rho^{(L)}$ of $\rho^{-d}\eta(\cdot/\rho)$, set
\[
    v_{\rho,w}^Y
    :=\nabla\Delta^{-1}\left[
        \eta_\rho^{(L)}*
        \left(\sum_{i=1}^n\delta_{Y_i}-\mathcal L_L\right)
    \right](w),
\]
with the zero Fourier mode removed. For every $\alpha\in(0,1]$ and
$2\leq p<\infty$,
\begin{align}\label{eq:terminal-gradient}
    \sup_{w\in\T_L^d}
    \bigl(\E|v_{\rho,w}^Y|^p\bigr)^{1/p}
    \leq C_{\alpha,p,d}\rho^{1-d/2},
    \qquad \alpha L\leq\rho\leq L,
\end{align}
where the constant also depends on the fixed mollifier $\eta$.
\end{lemma}

\begin{proof}[Proof of Lemma~\ref{lem:terminal-gradient}]
Put $K_{\rho,L}:=\nabla\Delta^{-1}\eta_\rho^{(L)}$, with zero Fourier
mode removed. Then
\begin{align}\label{eq:terminal-gradient-sum}
    v_{\rho,w}^Y=\sum_{i=1}^nK_{\rho,L}(w-Y_i),
\end{align}
a sum of independent centered vectors. Scaling and the rapid Fourier
decay of $\eta$ give, uniformly for $\rho/L\in[\alpha,1]$,
\begin{align}\label{eq:terminal-kernel-bounds}
    K_{\rho,L}(Lz)&=L^{1-d}K_{\rho/L,1}(z),\\
    \norm{K_{\rho,L}}_\infty&\leq C_{\alpha,d}L^{1-d}.
\end{align}
Rosenthal's inequality applied componentwise to
\eqref{eq:terminal-gradient-sum} therefore yields
\begin{align*}
    \bigl(\E|v_{\rho,w}^Y|^p\bigr)^{1/p}
    &\leq C_{p,d}\bigl(\sqrt n+n^{1/p}\bigr)
                    \norm{K_{\rho,L}}_\infty\\
    &\leq C_{\alpha,p,d}L^{1-d/2}
     \leq C_{\alpha,p,d}\rho^{1-d/2},
\end{align*}
since $n=L^d$, $p\geq2$, and $\rho\leq L$. This proves
\eqref{eq:terminal-gradient}.
\end{proof}

We keep the physical coordinates of \eqref{eq:dilate}: $\widehat\pi_n$ has first marginal
$\widehat\mu_n:=\sum_{i=1}^n\delta_{LX_i}$ and second marginal Lebesgue measure on $\T^4_L$, both of
mass $n=L^4$. For a physical radius $\rho$, let $\eta^{(L)}_\rho$ be the periodisation on $\T^4_L$ of
$\rho^{-4}\eta(\cdot/\rho)$, with $\eta$ the mollifier fixed in \eqref{eq:mollifier}, let
$\widehat u_\rho$ be the mean-zero solution of
\begin{align*}
    \Delta\widehat u_\rho
    =
    \eta^{(L)}_\rho*(\widehat\mu_n-1),
\end{align*}
and set $v_{\rho,w}:=\nabla\widehat u_\rho(w)$, the averaged shift predicted by the linear model at
$w$ and scale $\rho$.

In the application above, take $Y_i=LX_i$. These variables are
independent with common law $m_L$, and $v_{\rho,w}^Y$ then coincides
with the previously defined field $v_{\rho,w}$.

\begin{proposition}[Regularity radius and support estimate]
\label{prop:radius-support}
There are a constant $c>0$, constants $C,C_p<\infty$, and jointly measurable random radii
$r_*(w)\geq1$, all uniform in $L$ and $w\in\T^4_L$, such that
\begin{align}\label{eq:radius-moments}
    \E\left[e^{c\,r_*(w)^2}\right]\leq C,
    \qquad
    \norm{v_{\rho,w}}_{L^p(\Omega)}\leq C_p\,\rho^{-1}
    \quad(1\leq\rho\leq L),
\end{align}
and such that, on the event $r_*(w)<L$, whenever $r_*(w)\leq\rho\leq L$,
\begin{align}\label{eq:support-estimate}
    \sup\Big\{
        \abs{x-(y-v_{\rho,w})}:(x,y)\in\operatorname{Spt}\widehat\pi_n,\
        x\in B_\rho(w)
    \Big\}
    \leq
    C\rho^{2/3}r_*(w)^{1/3},
\end{align}
the differences being taken in coherent lifts: choose a lift of $x$ in $B_\rho(w)$
and lift $y$ so that $x-y$ is the minimizing displacement of the transport segment.
Changing the lift translates both endpoints by the same vector in $L\Z^4$.
\end{proposition}
 
Estimate \eqref{eq:support-estimate} is \cite[Proposition 4.7]{GH22} with their rate function
$\beta\equiv1$, which is the value in dimension four, and the exponent $2/3$ is their
$\rho(r_*^2/\rho^2)^{1/(d+2)}$ at $d=4$; the radius and its stretched-exponential moment are
\cite[Theorem 4.5]{GH22}, and the gradient moments at scales separated from $L$ are
\cite[Lemma 3.1]{GH22}. 
 
The gradient moments in \eqref{eq:radius-moments} are proved in \cite{GH22} for $\rho$ separated
from the system size. The remaining scales, where $\rho$ is comparable to $L$, are covered
by Lemma~\ref{lem:terminal-gradient}.

 \begin{lemma}\label{lem:segment-length}
For $(x,y)\in\operatorname{Spt}\widehat\pi_n$ let $\widehat\gamma_{x,y}$ be the transport segment of
the dilated plan and $\ell(x,y)$ its length. For $w\in\T^4_L$ and $R\geq1$, let
\begin{align*}
    D_{w,R}
    :=
    \sup\big\{
        \ell(x,y):(x,y)\in\operatorname{Spt}\widehat\pi_n,\
        \widehat\gamma_{x,y}\cap B_R(w)\neq\varnothing
    \big\}.
\end{align*}
 Then, for every
$p<\infty$,
\begin{align}\label{eq:segment-moments}
    \sup_{L\geq1}\
    \sup_{w\in\T^4_L}\
    \norm{D_{w,R}}_{L^p(\Omega)}
    \leq
    C_p(1+R).
\end{align}
\end{lemma}

The proof of Lemma \ref{lem:segment-length} uses the regularity radius and the support estimate of
\cite{GH22}, which are available in every dimension.

\begin{proof}[Proof of Lemma~\ref{lem:segment-length}]
Fix $p\in(0,\infty)$, $L\geq1$, $w\in\T_L^4$, and $R\geq1$, and
write $D:=D_{w,R}$. Every minimizing segment on $\T_L^4$ has length
at most $L$, so $D\leq L$. Let $C_0$ be the constant in \eqref{eq:support-estimate}. We first choose $\varepsilon$ small enough that the good events below
force segment lengths below $r$ in Step 1 and below $L/12$ in
Step 2. To ensure both conclusions, it suffices to set
$\varepsilon\in(0,1)$ sufficiently small that
\[
    C_0 4^{2/3}\varepsilon^{1/3}+\varepsilon<\frac1{12}.
\]

\emph{Step 1: segment lengths between $r$ and $2r$.}
Let $r\geq R$ satisfy $12r\leq L$. On the event $\{D<2r\}$,
every transport segment meeting $B_R(w)$ has length less than $2r$.
Choose a point where the segment meets $B_R(w)$ and lift it
to the Euclidean ball $B_R(\widetilde w)$, where $\widetilde w$
is a fixed lift of $w$. Lift the entire transport segment
continuously through this point. Its atomic endpoint
$\widetilde x$ then satisfies
\[
    |\widetilde x-\widetilde w|
    \leq R+\ell(x,y)<R+2r\leq3r.
\]
\[
    x\in B_{R+2r}(w)\subseteq B_{3r}(w)\subseteq B_{4r}(w).
\]
Suppose also that
\[
    r_*(w)<\varepsilon r,
    \qquad |v_{4r,w}|<\varepsilon r.
\]
Then $r_*(w)<L$ and $r_*(w)\leq4r\leq L$, so
\eqref{eq:support-estimate} applies at radius $4r$. For every such
segment it gives
\begin{align*}
    \ell(x,y)
    =|x-y|
    \leq |x-(y-v_{4r,w})|+|v_{4r,w}|
    \leq
    \bigl(C_0 4^{2/3}\varepsilon^{1/3}+\varepsilon\bigr)r
    <r.
\end{align*}
The bound is uniform over all segments meeting $B_R(w)$, so taking
the supremum yields $D<r$. Consequently,
\[
    \{r\leq D<2r\}
    \subseteq
    \{r_*(w)\geq\varepsilon r\}
    \cup
    \{|v_{4r,w}|\geq\varepsilon r\}.
\]
By \eqref{eq:radius-moments} and Markov's inequality, for every
$q\geq2$,
\[
    \mathbb P\{r\leq D<2r\}
    \leq C e^{-c\varepsilon^2r^2}
       +\left(\frac{C_q(4r)^{-1}}{\varepsilon r}\right)^q.
\]
Since $r\geq1$ and $q$ can be chosen arbitrarily large, for every
$M>0$ this implies
\begin{equation}\label{eq:crossing}
    \mathbb P\{r\leq D<2r\}
    \leq C_Mr^{-M},
    \qquad R\leq r,\quad 12r\leq L.
\end{equation}

\emph{Step 2: segment lengths comparable to $L$.}
For the fixed point $w$, almost surely every atom has a
representative in $w+[-L/2,L/2)^4$ lying in $B_L(w)$.
Use the corresponding coherent lifts of the transport segments, as described in the previous step.
On the event
\[
    r_*(w)<\varepsilon L,
    \qquad |v_{L,w}|<\varepsilon L,
\]
the support estimate at radius $L$ of Proposition~\ref{prop:radius-support} gives, for every transport segment,
\[
    \ell(x,y)
    \leq C_0 L^{2/3}r_*(w)^{1/3}+|v_{L,w}|
    \leq (C_0\varepsilon^{1/3}+\varepsilon)L
    <\frac{L}{12}.
\]
Taking the supremum and using \eqref{eq:radius-moments} again gives,
for every $M>0$,
\[
    \mathbb P\{D\geq L/12\}
    \leq C e^{-c\varepsilon^2L^2}+C_ML^{-M}
    \leq C_M L^{-M}.
\]

\emph{Step 3: summing the length scales.}
Choose $M>p$ and set $r_j:=2^jR$ for $j\geq0$.
If $D\geq R$ and $D<L/12$, then $D$ belongs to some interval
$[r_j,2r_j)$ with $12r_j\leq L$. Using $D\leq L$,
\eqref{eq:crossing}, and the estimate from Step 2, we obtain
\begin{align*}
    \E[D^p]
    &\leq R^p
       +\sum_{\substack{j\geq0\\12r_j\leq L}}
         (2r_j)^p\mathbb P\{r_j\leq D<2r_j\}
       +L^p\mathbb P\{D\geq L/12\}\\
    &\leq R^p
       +C_{p,M}\sum_{j\geq0}(2^jR)^{p-M}
       +C_M L^{p-M}\\
    &\leq R^p+C_{p,M}R^{p-M}+C_M L^{p-M}\\
    &\leq C_pR^p,
\end{align*}
where the last inequality uses $R,L\geq1$ and $M>p$.
Taking $p$-th roots and then the suprema over $L$ and $w$ proves
\eqref{eq:segment-moments}.
\end{proof}

\begin{lemma}
\label{lem:stress-L2}
For every fixed $a>0$, we have
\begin{align}\label{eq:stress-L2}
    \E\left[
        \norm{q_{a/L}*\Sigma}_{L_x^2}^2
    \right]
    \leq C_aL^{-4}.
\end{align}
\end{lemma}

The proof of Lemma~\ref{lem:stress-L2} is given in
Appendix \ref{app:stress}. The proof first uses the $L^2$ boundedness of $\mathcal K$ to reduce
to the measures $S_0,S_1$, which are carried by the transport segments; their local masses are
controlled by Lemma \ref{lem:segment-length}. We work in the physical coordinates of \eqref{eq:dilate}: with $S_L(x)=Lx$ and
$n=L^4$, the dilated plan $\widehat\pi_n$ has marginals $\sum_i\delta_{LX_i}$ and Lebesgue measure on
$\T^4_L$, both of mass $n$, and unit intensity.

\begin{lemma}
\label{lem:stress-square-function}
Fix $\vartheta>0$ and let  $ s_\ell:=2^{-\ell}$, $ 0\leq\ell\leq J,$
where $J$ is the largest integer such that $s_J\geq2L^{-1}$. Then
\begin{align}\label{eq:stress-square-function}
    \E\sum_{\ell=0}^J
    s_\ell^2
    \norm{
        q_{\vartheta s_\ell}*(\mathfrak{j}+\nabla u)
    }_{L_x^2}^2
    \leq
    C_\vartheta L^{-4}.
\end{align}
\end{lemma}

\begin{proof}
By \eqref{eq:stress-identity},
\begin{align*}
    q_{\vartheta s_\ell}*(\mathfrak{j}+\nabla u)
    =
    q_{\vartheta s_\ell}*\operatorname{div}\Sigma.
\end{align*}
For every $k\neq0$, by Cauchy–Schwarz,
\begin{align*}
    \abs{
        \widehat{\operatorname{div}\Sigma}(k)
    }^2
    &=
    4\pi^2
    \sum_{i=1}^4
    \abs{
        \sum_{j=1}^4
        k_j\widehat\Sigma_{ij}(k)
    }^2\\
    &\leq
    4\pi^2\abs{k}^2
    \abs{\widehat\Sigma(k)}^2.
\end{align*}
Therefore, Parseval's identity gives
\begin{align*}
    &\sum_{\ell=0}^J
    s_\ell^2
    \norm{
        q_{\vartheta s_\ell}*(\mathfrak{j}+\nabla u)
    }_{L_x^2}^2\\
    &\quad\leq
    4\pi^2
    \sum_{k\neq0}
    \left(
        \sum_{\ell=0}^J
        s_\ell^2\abs{k}^2
        e^{-8\pi^2\vartheta^2
        (s_\ell^2-L^{-2})\abs{k}^2}
    \right)
    e^{-8\pi^2\vartheta^2L^{-2}\abs{k}^2}
    \abs{\widehat\Sigma(k)}^2.
\end{align*}
Since $s_\ell\geq2L^{-1}$, we have
\(
    s_\ell^2-L^{-2}
    \geq
    3s_\ell^2/4.
\)
Since the scales $s_\ell$ are dyadic,
\begin{align*}
    \sup_{k\neq0}
    \sum_{\ell=0}^J
    s_\ell^2\abs{k}^2
    e^{-8\pi^2\vartheta^2
    (s_\ell^2-L^{-2})\abs{k}^2}
    \leq C_\vartheta.
\end{align*}
Consequently,
\begin{align*}
    \sum_{\ell=0}^J
    s_\ell^2
    \norm{
        q_{\vartheta s_\ell}*(\mathfrak{j}+\nabla u)
    }_{L_x^2}^2
    \leq
    C_\vartheta
    \norm{q_{\vartheta/L}*\Sigma}_{L_x^2}^2.
\end{align*}
Taking expectations and applying Lemma \ref{lem:stress-L2} with
$a=\vartheta$ proves \eqref{eq:stress-square-function}.
\end{proof}

\subsection{Summing all frequencies and all orders at once (\texorpdfstring{$d=4$}{d=4})}
\label{ssec:agg4}

Throughout this subsection, $d=4$, $L=n^{1/4}$, and
$\Lambda=L^2=n^{1/2}$. We prove that the squared error in the Green
approximation, multiplied by $n^2$, is bounded uniformly in $n$.
The two summations in the proof use different orthogonality properties:
the square-function estimate controls the frequency scales together,
and Hoeffding orthogonality controls the orders together.

The pair-frequency cutoff corresponding to $\Lambda$ is
\begin{align}\label{eq:agg4-cutoffs}
    M:=\frac{L}{\sqrt{8}\,\pi},
\end{align}
because the pair mode $e_k(x-y)$ has product eigenvalue
$2\lambda_k=8\pi^2|k|^2$. Recall the disjoint bands
\[
    \mathcal B_j
    :=\{k\in\Z^4:2^j\leq|k|<2^{j+1},\ \lambda_k\leq\Lambda\},
    \qquad s_j:=2^{-j},\qquad j\geq0.
\]
There are $O(1+\log n)$ nonempty bands. Every such band satisfies
$2^j\leq L/(2\pi)$, so $s_j\geq2\pi/L>2/L$ and the
square-function estimate \eqref{eq:stress-square-function} applies.
Empty bands are omitted from all sums.

Fix $\vartheta>0$ and define
\begin{align}\label{eq:agg4-error-field}
    \mathcal E_j
    :=q_{\vartheta s_j}*(\mathfrak{j}+\nabla u)
    =q_{\vartheta s_j}*\operatorname{div}\Sigma.
\end{align}
Here $\mathfrak{j}$ inside the convolution denotes the displacement current, and
the second equality is \eqref{eq:stress-identity}. The estimate that
will control both the pair response and the higher-order common responses is
\[
    \sum_j s_j^2\norm{\mathcal E_j}_{L^2_{\omega,x}}^2
    \leq C_\vartheta L^{-4}
    =\frac{C_\vartheta}{n}.
\]
This is a bound on the entire sum over scales. On $\mathcal B_j$, the
inverse of the heat multiplier in \eqref{eq: HeatKern} has $L^2$
operator norm at most $e^{16\pi^2\vartheta^2}$, uniformly in $j$.
Constants below may depend on the fixed value of $\vartheta$.

\begin{proposition}[Aggregate response at the critical cutoff]
\label{prop:aggregate4}
There is a constant $C<\infty$, independent of $n$, such that
\begin{align}\label{eq:agg4-conclusions}
    n^2\norm{P_2W_{n,4}-U_{n,L}}_2^2&\leq C,
    &n^2\sum_{r=3}^n\norm{P_rW_{n,4}}_2^2&\leq C.
\end{align}
Consequently,
\begin{align}\label{eq:agg4-approximation}
    \norm{n(W_{n,4}-\E W_{n,4})-nU_{n,L}}_2^2
    \leq C=o(\log n).
\end{align}
\end{proposition}

\begin{proof}
We prove the bounds for sufficiently large $n$; the finitely many
remaining values are absorbed into $C$.

We first specify the pairings used below. For random vector fields
$A,B\in L^2(\Omega\times\T^4)$, write
\[
    \langle A,B\rangle_x
    :=\int_{\T^4}A(x)\cdot B(x)\dd\mu(x),
    \qquad
    \norm{A}_{L^2_{\omega,x}}^2
    :=\E\int_{\T^4}|A(x)|^2\dd\mu(x).
\]
Thus $\E\langle A,B\rangle_x$ means that we first take the spatial
inner product at a fixed sample, and then average over
$X_1,\ldots,X_n$. For a smooth vector field $h$, the current acts by
\[
    \langle \mathfrak{j},h\rangle_x
    :=\frac1n\sum_{i=1}^n d_i\cdot h(X_i).
\]
We also use $\langle\nabla u,h\rangle_x$ for the distributional
action of $\nabla u$ on $h$. All such test fields below are
trigonometric polynomials. Since the heat kernel is even, moving its
convolution from a test field to a distribution preserves the pairing.

\medskip
\noindent\emph{Step 1: the second-order component.}
Fix a band with $\mathcal B_j\cap\{0<|k|\leq M\}\neq\varnothing$.
Let $f$ be real and even, with $\norm{f}_2=1$ and Fourier support in
this intersection. Recall
\[
    S_f:=\sum_{i\neq\ell}f(X_i-X_\ell),
    \qquad
    g_f(x):=\sum_{i=1}^n\nabla(-\Delta)^{-1}f(x-X_i).
\]
Define $g_f^{\mathrm h}$ by inverting the heat multiplier on the
Fourier support of $g_f$, so that
$g_f=q_{\vartheta s_j}*g_f^{\mathrm h}$. By \eqref{eq:gf_bound}
and the uniform inverse multiplier bound,
\begin{align}\label{eq:agg4-pair-field}
    \norm{g_f^{\mathrm h}}_{L^2_{\omega,x}}^2
    \leq Cns_j^2.
\end{align}
The response identity \eqref{eq:pair-response-current} and the
self-adjointness of heat convolution give
\begin{align*}
    \Cov(W_{n,4},S_f)
    &=2\E\langle \mathfrak{j},g_f\rangle_x\\
    &=2\E\langle \mathfrak{j}+\nabla u,g_f\rangle_x
      -2\E\langle\nabla u,g_f\rangle_x\\
    &=2\E\langle\mathcal E_j,g_f^{\mathrm h}\rangle_x
      -2\E\langle\nabla u,g_f\rangle_x.
\end{align*}
The Fourier calculation used in \eqref{eq:2_cov_error} gives
\(
    -2\E\langle\nabla u,g_f\rangle_x
    =2\langle G_M,f\rangle.
\)
On the other hand, the normalization
$U_{n,M}=n^{-2}\sum_{i\neq\ell}G_M(X_i-X_\ell)$ gives
\[
    \Cov(U_{n,M},S_f)
    =2\frac{n-1}{n}\langle G_M,f\rangle.
\]
Here $\langle G_M,f\rangle=\int_{\T^4}G_M(x)f(x)\dd\mu(x)$
is a deterministic scalar inner product. Subtracting the last two
identities explains the correction of size $n^{-1}$ in
\begin{align}\label{eq:agg4-pair-response}
    \Cov(W_{n,4}-U_{n,M},S_f)
    &=2\E\langle\mathcal E_j,g_f^{\mathrm h}\rangle_x
      +\frac2n\langle G_M,f\rangle.
\end{align}
In dimension four, each dyadic band satisfies
\[
    |\langle G_M,f\rangle|^2
    \leq\sum_{\substack{k\in\mathcal B_j\\|k|\leq M}}
                 \lambda_k^{-2}
    \leq C.
\]
Indeed, the number of lattice points in the band is at most $C2^{4j}$,
whereas $\lambda_k^{-2}\leq C2^{-4j}$ there. Cauchy--Schwarz and
\eqref{eq:agg4-pair-field} therefore yield, uniformly over these $f$,
\begin{align}\label{eq:agg4-pair-bound}
    |\Cov(W_{n,4}-U_{n,M},S_f)|
    \leq C\sqrt n\,s_j\norm{\mathcal E_j}_{L^2_{\omega,x}}
           +\frac Cn.
\end{align}

Let $D:=\Pi_{\leq\Lambda}P_2W_{n,4}-U_{n,M}$, and let $D_j$
be its pair-frequency projection onto $\mathcal B_j$. Pairwise duality
\eqref{eq:pairwise-duality} gives
\[
    n^2\norm{D_j}_2^2
    =\frac{n}{2(n-1)}
      \sup_f|\Cov(W_{n,4}-U_{n,M},S_f)|^2,
\]
where the supremum is over the unit tests in the fixed pair band.
The right-hand side of \eqref{eq:agg4-pair-bound} is independent of
$f$. Squaring that bound with $(a+b)^2\leq2a^2+2b^2$ therefore gives
\[
    \sup_f|\Cov(W_{n,4}-U_{n,M},S_f)|^2
    \leq Cns_j^2\norm{\mathcal E_j}_{L^2_{\omega,x}}^2
           +\frac C{n^2}.
\]
Since $n/[2(n-1)]\leq1$ for $n\geq2$, the same upper bound applies
to $n^2\norm{D_j}_2^2$.

Let $N$ be the number of nonempty pair bands. Their disjointness gives
$\norm{D}_2^2=\sum_j\norm{D_j}_2^2$, and
$N\leq C(1+\log n)$. Summing the preceding inequality and then using
\eqref{eq:stress-square-function}, we obtain
\begin{align}\label{eq:agg4-pair-low}
    n^2\norm{D}_2^2
    &=\sum_j n^2\norm{D_j}_2^2\notag\\
    &\leq Cn\sum_j s_j^2\norm{\mathcal E_j}_{L^2_{\omega,x}}^2
              +\frac{CN}{n^2}\notag\\
    &\leq CnL^{-4}+C\frac{1+\log n}{n^2}
    \leq C.
\end{align}
The factor $1+\log n$ arises only from summing the remainder
$C/n^2$ over the bands. The main term is bounded by applying the
square-function estimate to the whole weighted sum. Replacing that
estimate by a separate bound $C/n$ for each summand would introduce
an additional logarithmic factor.

To replace the Green cutoff $M$ by $L$, observe that
\begin{align}\label{eq:agg4-cutoff-shell}
    n^2\norm{U_{n,L}-U_{n,M}}_2^2
    =2\frac{n-1}{n}\sum_{M<|k|\leq L}\lambda_k^{-2}
    \leq C.
\end{align}
Indeed, $L/M=\sqrt8\,\pi$ is fixed, so this shell meets only a
bounded number of dyadic annuli, each contributing at most a constant.

\medskip
\noindent\emph{Step 2: the common responses of orders $r\geq3$.}
Use the admissible kernels $\mathcal C_{r,\Lambda}$ and the response
splitting in \eqref{eq:higher-order-response-split}. At the present
cutoff, the largest possible order is
\[
    R:=\left\lfloor\frac{\Lambda}{4\pi^2}\right\rfloor
    =O(L^2)=o(n).
\]
Fix $3\leq r\leq R$ and $a\in\mathcal C_{r,\Lambda}$. Recall
that $a$ is symmetric, canonical, and has $\norm{a}_2=1$.
Project its first Fourier variable onto $\mathcal B_j$, obtaining
$a_j$, and set
\[
    g_j:=\nabla_1(-\Delta_1)^{-1}a_j,
    \qquad
    (G^\sharp)_j(x)
    :=\sum_{\substack{i_2,\ldots,i_r\in[n]\\
                      \text{pairwise distinct}}}
          g_j(x,X_{i_2},\ldots,X_{i_r}).
\]
The bands cover the first-variable Fourier support of $a$, and hence
\[
    G^\sharp=\sum_j(G^\sharp)_j,
    \qquad \sum_j\norm{a_j}_2^2=1.
\]
For this fixed $r$ and $a$, abbreviate the heat-deconvolved field by
\[
    H_j:=(G^\sharp)^{\mathrm h}_j,
    \qquad (G^\sharp)_j=q_{\vartheta s_j}*H_j.
\]
The kernel $a_j$ is still symmetric in its last $r-1$ variables.
Thus the second-moment calculation in
\eqref{eq:common-field-moment}--\eqref{eq:common-field-band-bound},
together with the inverse heat multiplier bound, gives
\begin{align}\label{eq:agg4-common-moment}
    \norm{H_j}_{L^2_{\omega,x}}^2
    \leq Cs_j^2(r-1)!(n)_{r-1}\norm{a_j}_2^2.
\end{align}

We next explain which Hoeffding projection contributes to the
response. All Hoeffding projections in this step act on
$X_1,\ldots,X_n$, with $x$ held fixed, and act componentwise on vector
fields. The field $H_j$ has the form
\[
    H_j(x)
    =\sum_{\substack{i_2,\ldots,i_r\in[n]\\
                     \text{pairwise distinct}}}
       h_j(x,X_{i_2},\ldots,X_{i_r}),
\]
where $h_j$ is obtained from $a_j$ by applying
$\nabla_1(-\Delta_1)^{-1}$ and the inverse heat multiplier in the
first variable only. These operations preserve canonicity in each of
the last $r-1$ variables: integrating $h_j$ with respect to any one
of those variables gives zero. Each summand therefore belongs to the
Hoeffding subspace indexed by its $r-1$ sample labels. Consequently,
\[
    P_{r-1}H_j(x)=H_j(x).
\]

For comparison, define the smoothed Poisson field
$V_j:=q_{\vartheta s_j}*\nabla u$. Since
$\Delta u=\mu_n-\mu$, we have
\[
    V_j(x)=\frac1n\sum_{\ell=1}^n K_j(x-X_\ell),
    \qquad
    K_j:=\nabla\Delta^{-1}(q_{\vartheta s_j}-1),
    \qquad \int_{\T^4}K_j\dd\mu=0.
\]
For each fixed $x$, this is a sum of centered functions of individual
sample variables. Thus $P_1V_j(x)=V_j(x)$. Because $r-1\geq2$,
Hoeffding orthogonality gives
\[
    \E[V_j(x)\cdot H_j(x)]=0,
    \qquad
    \E\langle V_j,H_j\rangle_x=0.
\]
For a direct verification, each product of a summand of $V_j$ and a
summand of $H_j$ contains a sample label among $i_2,\ldots,i_r$
different from $\ell$. Integrating over that sample variable
annihilates the product by canonicity.

We can now rewrite the first term of
\eqref{eq:higher-order-response-split} explicitly:
\begin{align*}
    \mathcal R_r^{\rm com}(a)
    &=\frac2n\E\sum_{i=1}^n d_i\cdot G^\sharp(X_i)\\
    &=2\sum_j\E\langle \mathfrak{j},(G^\sharp)_j\rangle_x\\
    &=2\sum_j\E\langle q_{\vartheta s_j}*\mathfrak{j},H_j\rangle_x\\
    &=2\sum_j\E\langle\mathcal E_j-V_j,H_j\rangle_x\\
    &=2\sum_j\E\langle\mathcal E_j,H_j\rangle_x\\
    &=2\sum_j\E\langle P_{r-1}\mathcal E_j,H_j\rangle_x.
\end{align*}
The third equality moves the heat convolution to the current. The
fifth uses the vanishing expectation just proved. The last uses
$P_{r-1}H_j=H_j$ and self-adjointness of the Hoeffding projection.

Cauchy--Schwarz first in $(\omega,x)$ and then over $j$ gives
\begin{align*}
    |\mathcal R_r^{\rm com}(a)|
    &\leq2\sum_j
       \norm{P_{r-1}\mathcal E_j}_{L^2_{\omega,x}}
       \norm{H_j}_{L^2_{\omega,x}}\\
    &\leq C\sqrt{(r-1)!(n)_{r-1}}
       \sum_j s_j\norm{P_{r-1}\mathcal E_j}_{L^2_{\omega,x}}
                    \norm{a_j}_2\\
    &\leq C\sqrt{(r-1)!(n)_{r-1}}
       \left(\sum_j s_j^2
           \norm{P_{r-1}\mathcal E_j}_{L^2_{\omega,x}}^2\right)^{1/2}
       \left(\sum_j\norm{a_j}_2^2\right)^{1/2}.
\end{align*}
The last factor is one. Squaring and taking the supremum over $a$
therefore yields
\begin{align}\label{eq:agg4-common-dual}
    \sup_{a\in\mathcal C_{r,\Lambda}}
        |\mathcal R_r^{\rm com}(a)|^2
    \leq C(r-1)!(n)_{r-1}
         \sum_j s_j^2\norm{P_{r-1}\mathcal E_j}_{L^2_{\omega,x}}^2.
\end{align}
In particular, this summation introduces no factor equal to the
number of frequency bands.

It remains to sum over $r$. Since $R=o(n)$,
\[
    \frac{n^2(r-1)!(n)_{r-1}}{r!(n)_r}
    =\frac{n^2}{r(n-r+1)}\leq Cn,
    \qquad 3\leq r\leq R.
\]
For every fixed $j$, the different Hoeffding projections are
orthogonal, so Bessel's inequality gives
\[
    \sum_{r=3}^R
       \norm{P_{r-1}\mathcal E_j}_{L^2_{\omega,x}}^2
    \leq\norm{\mathcal E_j}_{L^2_{\omega,x}}^2.
\]
Combining these facts with \eqref{eq:stress-square-function}, we obtain
\begin{align}\label{eq:agg4-common-sum}
    n^2\sum_{r=3}^R\frac{1}{r!(n)_r}
         \sup_{a\in\mathcal C_{r,\Lambda}}
             |\mathcal R_r^{\rm com}(a)|^2
    &\leq Cn\sum_j s_j^2\sum_{r=3}^R
                \norm{P_{r-1}\mathcal E_j}_{L^2_{\omega,x}}^2
         \notag\\
    &\leq Cn\sum_j s_j^2
                \norm{\mathcal E_j}_{L^2_{\omega,x}}^2
    \leq CnL^{-4}=C.
\end{align}
Thus the sum over orders introduces no factor $R$.

\medskip
\noindent\emph{Step 3: the excluded-label responses.}
We now control the second term of
\eqref{eq:higher-order-response-split}. The Fourier trace argument
proving \eqref{eq:diagonal-trace}, applied in dimension four, gives
\begin{align}\label{eq:agg4-trace}
    \norm{\operatorname{Diag}_b g}_2^2
    \leq C\sum_{0<|k_1|\leq L}\frac{1}{|k_1|^2}
    \leq CL^2,
    \qquad a\in\mathcal C_{r,\Lambda}.
\end{align}
For the fields $Q_{b,i}$ defined in that argument, the conditional
second-moment calculation following \eqref{eq:diagonal-trace} yields
\begin{align}\label{eq:agg4-exclusion-moment}
    \sum_{i=1}^n\E\left|\sum_{b=2}^rQ_{b,i}(X_i)\right|^2
    \leq C(r-1)^2n(r-2)!(n-1)_{r-2}L^2.
\end{align}
For each $i$, let $P_m^{(-i)}$ denote the order-$m$ Hoeffding
projection in the remaining variables $(X_\ell)_{\ell\neq i}$,
with $X_i$ held fixed. Conditional on $X_i$, the field
$Q_{b,i}(X_i)$ is a sum of canonical kernels in $r-2$ distinct
remaining sample variables. It therefore belongs entirely to order
$r-2$ in those variables. Conditional orthogonality allows us to
write the excluded-label response as
\[
    \mathcal R_r^{\rm exc}(a)
    =-\frac2n\sum_{i=1}^n
       \E\left[P_{r-2}^{(-i)}d_i\cdot
                       \sum_{b=2}^rQ_{b,i}(X_i)\right].
\]
Cauchy--Schwarz and \eqref{eq:agg4-exclusion-moment} imply
\begin{align*}
    \sup_{a\in\mathcal C_{r,\Lambda}}
       |\mathcal R_r^{\rm exc}(a)|^2
    &\leq \frac{CL^2(r-1)^2(r-2)!(n-1)_{r-2}}{n}
       \sum_{i=1}^n\E|P_{r-2}^{(-i)}d_i|^2.
\end{align*}
The factorial ratio after multiplication by $n^2/[r!(n)_r]$ is
\[
    \frac{n(r-1)^2(r-2)!(n-1)_{r-2}}{r!(n)_r}
    =\frac{r-1}{r(n-r+1)}\leq\frac Cn.
\]
Summing over $r$ and applying Bessel's inequality with $X_i$ held
fixed therefore gives
\begin{align}\label{eq:agg4-exclusion-sum}
    n^2\sum_{r=3}^R\frac{1}{r!(n)_r}
         \sup_{a\in\mathcal C_{r,\Lambda}}
             |\mathcal R_r^{\rm exc}(a)|^2
    &\leq \frac{CL^2}{n}\sum_{i=1}^n\sum_{r=3}^R
                    \E|P_{r-2}^{(-i)}d_i|^2\notag\\
    &\leq \frac{CL^2}{n}\sum_{i=1}^n\E|d_i|^2\notag\\
    &\leq CL^2\E W_{n,4}\leq C.
\end{align}
The third inequality uses \eqref{eq:disp_bound}, and the last uses
$\E W_{n,4}=O(L^{-2})$. Holding $X_i$ fixed is essential here:
evaluating a random field at $X_i$ does not in general preserve its
Hoeffding order in all $n$ sample variables.

Since $\mathcal R_r=\mathcal R_r^{\rm com}+\mathcal R_r^{\rm exc}$,
we have
\[
    \sup_a|\mathcal R_r(a)|^2
    \leq2\sup_a|\mathcal R_r^{\rm com}(a)|^2
        +2\sup_a|\mathcal R_r^{\rm exc}(a)|^2.
\]
Applying \eqref{eq:higher-order-low-response},
\eqref{eq:agg4-common-sum}, and \eqref{eq:agg4-exclusion-sum} gives
\begin{align}\label{eq:agg4-higher-low}
    n^2\sum_{r=3}^n
       \norm{\Pi_{\leq\Lambda}P_rW_{n,4}}_2^2\leq C.
\end{align}

\medskip
\noindent\emph{Step 4: the remaining product frequencies.}
Lemma~\ref{lem:high-frequency}, at the same critical cutoff, gives
\begin{align}\label{eq:agg4-high}
    n^2\norm{\Pi_{>\Lambda}(W_{n,4}-\E W_{n,4})}_2^2
    \leq\frac{4n\E W_{n,4}}{\Lambda}\leq C.
\end{align}
The last inequality follows from $n=L^4$, $\Lambda=L^2$, and
$\E W_{n,4}=O(L^{-2})$. Product spectral projections commute with
Hoeffding projections, and the Hoeffding orders are orthogonal.
Consequently, \eqref{eq:agg4-high} bounds both
$n^2\norm{\Pi_{>\Lambda}P_2W_{n,4}}_2^2$ and
$n^2\sum_{r=3}^n\norm{\Pi_{>\Lambda}P_rW_{n,4}}_2^2$.

For the second-order component, use the decomposition
\[
    P_2W_{n,4}-U_{n,L}
    =D+\Pi_{>\Lambda}P_2W_{n,4}+(U_{n,M}-U_{n,L}).
\]
The bounds \eqref{eq:agg4-pair-low}, \eqref{eq:agg4-cutoff-shell},
and \eqref{eq:agg4-high}, together with the squared triangle
inequality, prove the first assertion of \eqref{eq:agg4-conclusions}.
The second follows by adding \eqref{eq:agg4-higher-low} to the
high-frequency bound for orders $r\geq3$.

Finally, $P_1W_{n,4}=0$ by \eqref{eq:P1=0}, and $U_{n,L}$ belongs
to the second Hoeffding subspace. Hence orthogonality gives
\begin{align*}
    &\norm{n(W_{n,4}-\E W_{n,4})-nU_{n,L}}_2^2\\
    &\qquad=n^2\norm{P_2W_{n,4}-U_{n,L}}_2^2
             +n^2\sum_{r=3}^n\norm{P_rW_{n,4}}_2^2
    \leq C,
\end{align*}
which proves \eqref{eq:agg4-approximation}.
\end{proof}

\begin{remark}
The logarithm in the leading variance comes from the Green statistic:
\[
    n^2\Var(U_{n,L})
    =2\frac{n-1}{n}\sum_{0<|k|\leq L}\lambda_k^{-2}
    \asymp\log L\asymp\log n.
\]
The proposition bounds the squared approximation error on this scale
by $O(1)=o(\log n)$. In Step~1, the only logarithmic factor in the
error bound is the contribution $O((1+\log n)/n^2)$ of the remainder
summed over the frequency bands. The square-function estimate controls
the main error over all scales, while Bessel's inequality controls
the higher-order responses over all Hoeffding orders.
\end{remark}

%\subsection{The two approximation theorems}
%\label{ssec:approx}

%------------------------------------------------
%                 Section CLTs
%------------------------------------------------
\section{Central limit theorems}
\label{sec:clt}

In this Section, we prove the limit laws of the two leading order statistics, which are the last ingredients required to prove our main theorems.

\subsection{Degenerate $U$-statistics with a fixed square-integrable kernel ($d=2,3$)}
\label{ssec:clt23}

\begin{proposition}\label{prop:green-fixed-limit}
For $d\in\{2,3\}$, with the normalization in \eqref{eq:Ustat},
\begin{align}\label{eq:green-fixed-convergence}
    nU_n\weak L_d,
\end{align}
where $L_d$ is the series in \eqref{eq:limitlaw}, convergent in $L^2$ and
almost surely. Moreover, for every $n\geq1$,
\begin{align}\label{eq:green-fixed-variance}
    \E U_n&=0,
    &\Var(nU_n)
    &=2\frac{n-1}{n}\sum_{k\ne0}\lambda_k^{-2}.
\end{align}
In particular,
\begin{align}\label{eq:green-fixed-variance-limit}
    n^2\Var(U_n)\longrightarrow\Var(L_d)
    =2\sum_{k\ne0}\lambda_k^{-2},
\end{align}
and, when $d=2$, this limit equals $\beta(2)/(12\pi^2)$.
\end{proposition}

\begin{proof}
By \eqref{eq:degenerate}, $G\in L^2(\mu)$. Since $X_i-X_j$ has law
$\mu$ whenever $i\ne j$, the statistic $U_n$ is well defined in $L^2$.
%and does not depend on the choice of an almost-everywhere representative of $G$. 
Canonicity gives $\E U_n=0$. For $n\geq2$, the first equality
in \eqref{eq:pairwise-duality}, applied to $nU_n$ with pair kernel
$2G/n$, gives \eqref{eq:green-fixed-variance} by
\eqref{eq:degenerate}; the case $n=1$ is immediate. Applying the same
identity with pair kernel $2(G-G_K)/n$ and using Parseval's identity
also gives
\begin{align}\label{eq:green-fixed-tail}
    \E\abs{nU_n-nU_{n,K}}^2
    =2\frac{n-1}{n}\sum_{\abs{k}>K}\lambda_k^{-2}
    \leq2\sum_{\abs{k}>K}\lambda_k^{-2}
    \longrightarrow0
\end{align}
as $K\to\infty$, uniformly in $n$.

We next prove convergence for a fixed $K<\infty$, using the empirical
Fourier sums $S_k$ from Section~\ref{ssec:pair23} and defined in \eqref{eq: SkDef}. Choose one
representative $k$ from each class $[k]\in\mathcal P_d$ with
$\abs{k}\leq K$. Since $S_{-k}=\overline{S_k}$ by the property of the complex exponential, the identity
\begin{align*}
    \sum_{i\ne j}e_k(X_i-X_j)=\abs{S_k}^2-n
\end{align*}
yields the finite Fourier expansion
\begin{align}\label{eq:green-fixed-fourier}
    nU_{n,K}
    =\sum_{\substack{[k]\in\mathcal P_d\\\abs{k}\leq K}}
        \frac{2}{\lambda_{[k]}}
        \left(\frac{\abs{S_k}^2}{n}-1\right).
\end{align}

Orthogonality of torus characters gives, for these representatives,
\begin{align*}
    \E e_k(X_1)&=0,
    &\E\bigl[e_k(X_1)\overline{e_\ell(X_1)}\bigr]
        &=\mathbf 1_{\{k=\ell\}},
    &\E\bigl[e_k(X_1)e_\ell(X_1)\bigr]&=0.
\end{align*}
Consequently, the real and imaginary parts of the vector
$(e_k(X_1))_{[k]\in\mathcal P_d,\,\abs{k}\leq K}$ have covariance matrix
$\tfrac12 I$. The multivariate central limit theorem thus implies
\begin{align*}
    \left(\frac{S_k}{\sqrt n}\right)_{[k]\in\mathcal P_d,\,\abs{k}\leq K}
    \weak
    (Z_{[k]})_{[k]\in\mathcal P_d,\,\abs{k}\leq K},
\end{align*}
where all the real and imaginary parts of the $Z_{[k]}$ are independent
$\mathcal N(0,1/2)$ variables. In particular, the variables
$\abs{Z_{[k]}}^2$ are independent and have the $\mathrm{Exp}(1)$ law.
By the continuous mapping theorem and
\eqref{eq:green-fixed-fourier},
\begin{align}\label{eq:green-fixed-truncated-limit}
    nU_{n,K}\weak L_{d,K}
    :=\sum_{\substack{[k]\in\mathcal P_d\\\abs{k}\leq K}}
        \frac{2}{\lambda_{[k]}}(\zeta_{[k]}-1).
\end{align}

The summands in \eqref{eq:limitlaw} are independent and centered, and
\begin{align*}
    \sum_{[k]\in\mathcal P_d}
        \Var\left(\frac{2}{\lambda_{[k]}}(\zeta_{[k]}-1)\right)
    =4\sum_{[k]\in\mathcal P_d}\lambda_{[k]}^{-2}
    =2\sum_{k\ne0}\lambda_k^{-2}<\infty.
\end{align*}
The series therefore converges in $L^2$ and, by the convergence theorem
for independent centered square-integrable series, almost surely
(for example, when the classes are enumerated in nondecreasing order
of $\abs{k}$). Its tails satisfy
\begin{align}\label{eq:green-limit-tail}
    \E\abs{L_d-L_{d,K}}^2
    =2\sum_{\abs{k}>K}\lambda_k^{-2}\longrightarrow0.
\end{align}
Combining \eqref{eq:green-fixed-tail},
\eqref{eq:green-fixed-truncated-limit}, and
\eqref{eq:green-limit-tail} proves
\eqref{eq:green-fixed-convergence} by taking the limit in $n$ first and then $K$.
The same variance computation gives
\eqref{eq:green-fixed-variance-limit} directly, without requiring
moment convergence to follow from weak convergence.

Finally, in dimension two, Jacobi's two-square identity, as recorded
in Appendix~\ref{app:lattice}, gives
\begin{align*}
    \sum_{(a,b)\in\Z^2\setminus\{(0,0)\}}
        \frac{1}{(a^2+b^2)^2}
    =4\zeta(2)\beta(2).
\end{align*}
Since $\lambda_k=4\pi^2\abs{k}^2$ and $\zeta(2)=\pi^2/6$,
\begin{align*}
    2\sum_{k\in\Z^2\setminus\{0\}}\lambda_k^{-2}
    =\frac{4\zeta(2)\beta(2)}{8\pi^4}
    =\frac{\beta(2)}{12\pi^2}.
\end{align*}
\end{proof}

\begin{remark}\label{rem:green-fixed-dimension}
The only dimension-dependent input in this argument is the
square-integrability of $G$, equivalently
$\sum_{k\ne0}\lambda_k^{-2}<\infty$. This condition holds for $d\leq3$
and fails at $d=4$, where the uniform Fourier-tail estimate
\eqref{eq:green-fixed-tail} is no longer available and a growing cutoff
is needed.
\end{remark}

\subsection{The critical kernel: fourth moment and de Jong's theorem (\texorpdfstring{$d=4$}{d=4})}
\label{ssec:clt4}

Canonicity and the variance identity
\eqref{eq:pairwise-duality}, applied with pair kernel $2G_K/n$, give
\begin{align}\label{eq:green-critical-variance-cutoff}
    \E U_{n,K}&=0,
    &\Var(nU_{n,K})
    &=2\frac{n-1}{n}\sum_{0<\abs{k}\leq K}\lambda_k^{-2}.
\end{align}
At the cutoff $K=L$, the lattice asymptotic
\eqref{eq:lattice-critical-four} therefore yields
\begin{align}\label{eq:green-critical-variance}
    \Var(nU_{n,L})=\frac{\log n}{16\pi^2}+O(1).
\end{align}
To identify the limiting law, we first estimate the fourth moment of
a canonical pair sum.

\begin{lemma}[Fourth moment of a pair sum]
\label{lem:pair-fourth-moment}
Let $g\in L^4(\T^4,\mu)$ be real, even, centered, and nonzero. Set
\[
    A:=\norm{g}_2^2,\qquad
    B:=\norm{g}_4^4,\qquad
    C_1:=\norm{g*g}_2^2.
\]
For the pair sum $S_g:= 2 \sum_{i<j} g(X_i-X_j)$  and
every $n\geq2$,
\begin{align}\label{eq:pair-fourth-moment}
    \abs{\frac{\E S_g^4}{(\E S_g^2)^2}-3}
    \lesssim
    \frac1{n^2}
    +\frac{B}{n^2A^2}
    +\frac{\sqrt{BC_1}}{nA^2}
    +\frac{C_1}{A^2},
\end{align}
with an absolute implicit constant.
\end{lemma}

\begin{proof}
Write $N=\binom n2$ and
\[
    D:=\int_{\T^4}g(x)^2(g*g)(x)\dd\mu(x).
\]
Orthogonality of distinct unordered pairs gives
$\E(S_g/2)^2=NA$. Expand the fourth power of
$S_g/2$. Each product factor will be seen as graph edge where the vertices are the indices of the X's.  Each product gives rise to four
unordered edges, counted with multiplicity. 

If a vertex has degree one, conditioning on all other sample points shows that the corresponding product has expectation zero. Such products therefore contribute nothing to \(\E[(S_g/2)^4]\).
 The remaining possibilities are a single edge repeated four
times, two distinct edges each repeated twice, a triangle with one
edge repeated twice, and a cycle on four vertices.

Their expectations are, respectively, $B$, $A^2$, $D$, and $C_1$.
For the second case, the value is $A^2$ even when the two edges share
a vertex, since the conditional second moment of either edge is
$A$. In the last two cases, integrating over a vertex incident to
two different edges produces the convolution $g*g$. Counting the
orders of the four edges consequently gives
\begin{align}\label{eq:pair-fourth-moment-exact}
    \E(S_g/2)^4
    =NB+3N(N-1)A^2
        +36\binom n3D+72\binom n4C_1.
\end{align}
Indeed, the two doubled edges have $6\binom N2=3N(N-1)$
orderings. Each set of three vertices has three choices for the
doubled triangle edge and $4!/2!=12$ orderings. Each set of four
vertices has three cycles and $4!=24$ orderings.

Subtracting $3N^2A^2$ from \eqref{eq:pair-fourth-moment-exact} and
dividing by $N^2A^2$ gives
\[
    \frac{\E S_g^4}{(\E S_g^2)^2}-3
    =\frac{B-3A^2}{NA^2}
      +\frac{36\binom n3D+72\binom n4C_1}{N^2A^2}.
\]
Since $\abs{D}\leq\sqrt{BC_1}$ by Cauchy--Schwarz, the binomial
coefficients give \eqref{eq:pair-fourth-moment}.
\end{proof}

\begin{proposition}\label{prop:green-critical-limit}
In dimension four,
\begin{align}\label{eq:green-critical-convergence}
    \frac{nU_{n,L}}{\sqrt{\log n}}
    \weak\mathcal N\left(0,\frac{1}{16\pi^2}\right).
\end{align}
\end{proposition}

\begin{proof}
For $g=G_K$, the quantities in Lemma~\ref{lem:pair-fourth-moment}
satisfy, by Parseval's identity and
\eqref{eq:lattice-critical-four},
\begin{align}\label{eq:green-critical-kernel-bounds}
    A&=\sum_{0<\abs{k}\leq K}\lambda_k^{-2}
       =\frac{\log K}{8\pi^2}+O(1),\notag\\
    C_1&=\sum_{0<\abs{k}\leq K}\lambda_k^{-4}\lesssim1.
\end{align}
The inverse Hausdorff--Young inequality and
Lemma~\ref{lem:weighted-lattice-sums} also give
\begin{align}\label{eq:green-critical-fourth-norm}
    B=\norm{G_K}_4^4
    \leq\left(\sum_{0<\abs{k}\leq K}
                    \lambda_k^{-4/3}\right)^3
    \lesssim K^4.
\end{align}
Thus, at $K=L$, we have $A\asymp\log n$, $B\lesssim n$, and
$C_1\lesssim1$. Since $nU_{n,L}=S_{G_L}/n$, the fourth-moment ratio
is unchanged by this deterministic factor. Applying
\eqref{eq:pair-fourth-moment} yields
\begin{align}\label{eq:green-critical-fourth-moment}
    \abs{
        \frac{\E[(nU_{n,L})^4]}{\Var(nU_{n,L})^2}-3
    }
    \lesssim\frac1{(\log n)^2}\longrightarrow0.
\end{align}

It remains to check that no sample point carries a nonvanishing
fraction of the variance. Write
\[
    H_{ij,n}:=\frac2nG_L(X_i-X_j),\qquad i\ne j,
\]
so that $nU_{n,L}=\sum_{i<j}H_{ij,n}$. These kernels have finite
fourth moments and satisfy
\[
    \E[H_{ij,n}\mid X_i]
    =\E[H_{ij,n}\mid X_j]=0.
\]
All pair variances are equal. Hence
\begin{align}\label{eq:green-critical-influence}
    \max_{1\leq i\leq n}
    \frac{\sum_{j\ne i}\E H_{ij,n}^2}
         {\Var(nU_{n,L})}
    =\frac{n-1}{\binom n2}=\frac2n\longrightarrow0.
\end{align}
For a triangular array of canonical pair kernels, de Jong's theorem
\cite{dJ87} states that \eqref{eq:green-critical-fourth-moment} and
\eqref{eq:green-critical-influence} imply
\[
    \frac{nU_{n,L}}{\sqrt{\Var(nU_{n,L})}}
    \weak\mathcal N(0,1).
\]
Together with \eqref{eq:green-critical-variance}, this proves
\eqref{eq:green-critical-convergence}.
\end{proof}

\subsection{Proofs of the main theorems}
\label{ssec:proofs}

\begin{proof}[Proof of Theorem~\ref{thm:main}]
Let $d\in\{2,3\}$. By \eqref{eq:P1=0}, the Hoeffding decomposition,
and the fact that $U_n$ belongs to the second Hoeffding subspace,
orthogonality gives
\begin{align*}
    &\norm{n(W_{n,d}-\E W_{n,d})-nU_n}_2^2\\
    &\qquad=
        n^2\norm{P_2W_{n,d}-U_n}_2^2
        +n^2\sum_{r=3}^n\norm{P_rW_{n,d}}_2^2
        \longrightarrow0.
\end{align*}
Here the first term tends to zero by
Lemma~\ref{lem:uniform-pairwise-response} and
\eqref{eq:uniform-pairwise-response}, while the second tends to zero
by Lemma~\ref{lem:higher-orders} and \eqref{eq:higher-orders}.
Together with \eqref{eq:green-fixed-convergence}, this proves
\[
    n(W_{n,d}-\E W_{n,d})\weak L_d
\]
by Slutsky's theorem.

To transfer the variance, use $\E U_n=0$ and the reverse triangle
inequality in $L^2$:
\begin{align*}
    \abs{n\sqrt{\Var(W_{n,d})}-\sqrt{\Var(nU_n)}}
    \leq\norm{n(W_{n,d}-\E W_{n,d})-nU_n}_2
    \longrightarrow0.
\end{align*}
Equation~\eqref{eq:green-fixed-variance} therefore yields
\[
    n^2\Var(W_{n,d})\longrightarrow2\sum_{k\ne0}\lambda_k^{-2}.
\]
The value $\beta(2)/(12\pi^2)$ in dimension two follows from
Proposition~\ref{prop:green-fixed-limit}.
\end{proof}

\begin{proof}[Proof of Theorem~\ref{thm:main2}]
For $d=4$, Proposition~\ref{prop:aggregate4}, specifically
\eqref{eq:agg4-approximation}, gives
\begin{align*}
    \norm{
        \frac{n(W_{n,4}-\E W_{n,4})}{\sqrt{\log n}}
        -\frac{nU_{n,L}}{\sqrt{\log n}}
    }_2^2
    =O\bigl((\log n)^{-1}\bigr)\longrightarrow0.
\end{align*}
Combining this estimate with \eqref{eq:green-critical-convergence}
from Proposition~\ref{prop:green-critical-limit} and applying
Slutsky's theorem proves
\[
    \frac{n}{\sqrt{\log n}}(W_{n,4}-\E W_{n,4})
    \weak\mathcal N\left(0,\frac{1}{16\pi^2}\right).
\]

Since $U_{n,L}$ is centered, the reverse triangle inequality gives
\begin{align*}
    &\abs{
        \frac{n\sqrt{\Var(W_{n,4})}}{\sqrt{\log n}}
        -\sqrt{\frac{\Var(nU_{n,L})}{\log n}}
    }\\
    &\qquad\leq
    \norm{
        \frac{n(W_{n,4}-\E W_{n,4})-nU_{n,L}}{\sqrt{\log n}}
    }_2\longrightarrow0.
\end{align*}
The variance asymptotic \eqref{eq:green-critical-variance} now yields
\[
    \frac{n^2}{\log n}\Var(W_{n,4})
    \longrightarrow\frac{1}{16\pi^2},
\]
which completes the proof.
\end{proof}

%------------------------------------------------
%                 Appendix
%------------------------------------------------
\newpage
\appendix

\section{Hoeffding subspaces and decomposition}
\label{app:hoeffding}

Let $X_1,\ldots,X_n$ be independent with common law $\mu$. For
$S\subseteq[n]$, write $X_S=(X_i)_{i\in S}$.

For nonempty $S\subseteq[n]$, define $\mathcal H_S$ to be the space of random
variables $H=h(X_S)\in L^2$ satisfying
\begin{align*}
    \E\left[H\mid X_{S\setminus\{i\}}\right]
    =
    0
    \qquad
    \text{for every }i\in S.
\end{align*}
Let $\mathcal H_\varnothing$ be the space of constant random variables. If
$S\neq T$, then for every $H_S\in\mathcal H_S$ and $H_T\in\mathcal H_T$,
\begin{align*}
    \E[H_SH_T]=0.
\end{align*}

For $0\leq r\leq n$, define the Hoeffding subspace of order $r$ by
\begin{align*}
    \mathcal H_r
    :=
    \bigoplus_{\substack{S\subseteq[n]\\|S|=r}}
    \mathcal H_S.
\end{align*}
Then
\begin{align*}
    L^2\bigl((\T^d)^n,\mu^{\otimes n}\bigr)
    =
    \bigoplus_{S\subseteq[n]}\mathcal H_S
    =
    \bigoplus_{r=0}^n\mathcal H_r.
\end{align*}

For $F=F(X_1,\ldots,X_n)\in L^2$ and $S\subseteq[n]$, define
\begin{align*}
    P_SF
    :=
    \sum_{T\subseteq S}
    (-1)^{\abs{S}-\abs{T}}
    \E[F\mid X_T],
\end{align*}
where $\E[F\mid X_\varnothing]:=\E F$.
Then $P_SF\in\mathcal H_S$ and
\begin{align*}
    F
    =
    \sum_{S\subseteq[n]}P_SF.
\end{align*}
The projection of $F$ onto $\mathcal H_r$ is
\begin{align*}
    P_rF
    :=
    \sum_{\substack{S\subseteq[n]\\|S|=r}}P_SF,
\end{align*}
and
\begin{align*}
    F-\E F
    =
    \sum_{r=1}^nP_rF.
\end{align*}
Moreover,
\begin{align*}
    \E\left[(P_rF)(P_sF)\right]
    &=
    0,
    \qquad r\neq s,\\
    \Var(F)
    &=
    \sum_{r=1}^n\Var(P_rF).
\end{align*}

If $F$ is symmetric, then for every $S\subseteq[n]$ with $\abs{S}=r$,
the random variable $P_SF$ is given by the same symmetric canonical kernel.
More precisely, there exists a symmetric canonical function
$h_{r,F}:(\T^d)^r\to\R$ such that, whenever
$S=\{i_1<\cdots<i_r\}$,
\begin{align*}
    P_SF
    =
    h_{r,F}(X_{i_1},\ldots,X_{i_r}).
\end{align*}
Consequently,
\begin{align*}
    P_rF
    =
    \sum_{\substack{S\subseteq[n]\\\abs{S}=r}}
    h_{r,F}(X_S),
\end{align*}
for $S=\{i_1<\cdots<i_r\}$. By orthogonality,
\begin{align*}
    \E\left[\abs{P_rF}^2\right]
    =
    \binom{n}{r}\norm{h_{r,F}}_2^2.
\end{align*}

Let $a:(\T^d)^r\to\R$ be symmetric and canonical, and define
\begin{align*}
    U_a
    &:=
    \sum_{\substack{i_1,\ldots,i_r\in[n]\\
    \text{pairwise distinct}}}
    a(X_{i_1},\ldots,X_{i_r})\\
    &=
    r!\sum_{\substack{I\subseteq[n]\\|I|=r}}
    a(X_I).
\end{align*}
Then $U_a\in\mathcal H_r$, and hence $\Cov(F,U_a) = \Cov(P_rF,U_a)$.
For symmetric canonical kernels $a$ and $b$,
\begin{align*}
    \E[U_aU_b]
    =
    r!(n)_r\langle a,b\rangle.
\end{align*}
Moreover,
\begin{align*}
    \Cov(F,U_a)
    =
    (n)_r\langle h_{r,F},a\rangle.
\end{align*}

The product Fourier projections preserve the Hoeffding subspaces:
\begin{align*}
    \Pi_{\leq\Lambda}\mathcal H_r
    &\subseteq\mathcal H_r,
    &
    \Pi_{>\Lambda}\mathcal H_r
    &\subseteq\mathcal H_r,
\end{align*}
and
\begin{align*}
    \Pi_{\leq\Lambda}P_rF
    =
    P_r\Pi_{\leq\Lambda}F,
    \qquad
    \Pi_{>\Lambda}P_rF
    =
    P_r\Pi_{>\Lambda}F.
\end{align*}
The same statements hold componentwise for vector-valued random variables and
for random fields with the spatial variable held fixed.

\section{Lattice sums and the constants}
\label{app:lattice}
We first estimate the number of lattice points in a ball. For $R\geq1$, set
\begin{align*}
    N_d(R):=\#\{k\in\Z^d:0<\abs{k}\leq R\}.
\end{align*}

\begin{lemma}\label{lem:lattice-count}
For every $d\geq1$,
\begin{align*}
    N_d(R)=\abs{B_1^d}R^d+O_d(R^{d-1}),
\end{align*}
where $\abs{B_1^d}$ denotes the volume of the Euclidean unit ball.
\end{lemma}

\begin{proof}
The cubes $Q_k:=k+[-1/2,1/2)^d$, $k\in\Z^d$, partition $\R^d$
and each has volume one. Every point in $Q_k$ lies within distance
$c_d:=\sqrt d/2$ of $k$. Hence, for $R>c_d$,
\begin{align*}
    B_{R-c_d}^d
    \subseteq
    \bigcup_{\abs{k}\leq R}Q_k
    \subseteq
    B_{R+c_d}^d.
\end{align*}
Taking volumes gives
\begin{align*}
    \abs{B_1^d}(R-c_d)^d
    \leq N_d(R)+1
    \leq \abs{B_1^d}(R+c_d)^d.
\end{align*}
Expanding the two bounds proves the result.
\end{proof}

\begin{lemma}\label{lem:weighted-lattice-sums}
For $p>0$ and $R\geq1$,
\begin{align*}
    \sum_{0<\abs{k}\leq R}\frac{1}{\abs{k}^p}
    &\lesssim_{d,p}
    \begin{cases}
        R^{d-p},&p<d,\\
        1,&p>d.
    \end{cases}
\end{align*}
In the critical case,
\begin{align*}
    \sum_{0<\abs{k}\leq R}\frac{1}{\abs{k}^d}
    =
    d\abs{B_1^d}\log R+O_d(1).
\end{align*}
Moreover, for $p>d$,
\begin{align*}
    \sum_{\abs{k}>R}\frac{1}{\abs{k}^p}
    \lesssim_{d,p}R^{d-p}.
\end{align*}
\end{lemma}

\begin{proof}
For every $k\in\Z^d$ with $0<\abs{k}\leq R$,
\begin{align*}
    \frac{1}{\abs{k}^p}
    =
    \frac{1}{R^p}
    +
    p\int_{\abs{k}}^R t^{-p-1}\dd t.
\end{align*}
Summing over $k$ and interchanging the finite sum and integral gives
\begin{align}\label{eq:sum_by_parts}
    \sum_{0<\abs{k}\leq R}\frac{1}{\abs{k}^p}
    &=
    \frac{N_d(R)}{R^p}
    +
    p\int_1^R
    \left(
        \sum_{0<\abs{k}\leq R}
        \mathbf 1_{\{\abs{k}\leq t\}}
    \right)t^{-p-1}\dd t \notag\\
    &=
    \frac{N_d(R)}{R^p}
    +
    p\int_1^R\frac{N_d(t)}{t^{p+1}}\dd t.
\end{align}
For $p\neq d$, the bounds follow from $N_d(t)\lesssim_d t^d$.
For $p=d$, Lemma \ref{lem:lattice-count} gives
\begin{align*}
    \sum_{0<\abs{k}\leq R}\frac{1}{\abs{k}^d}
    &=
    \frac{N_d(R)}{R^d}
    +
    d\abs{B_1^d}\int_1^R\frac{\dd t}{t}
    +
    O_d\left(\int_1^R\frac{\dd t}{t^2}\right)\\
    &=
    d\abs{B_1^d}\log R+O_d(1).
\end{align*}
Finally, if $p>d$, using $\abs{k}^{-p}=p\int_{\abs{k}}^\infty t^{-p-1}\dd t$, we obtain
\begin{align*}
    \sum_{\abs{k}>R}\frac{1}{\abs{k}^p}
    &=
    p\int_R^\infty
    \#\{k\in\Z^d:R<\abs{k}\leq t\}\,t^{-p-1}\dd t\\
    &=
    p\int_R^\infty
    \bigl(N_d(t)-N_d(R)\bigr)t^{-p-1}\dd t\\
    &=
    p\int_R^\infty\frac{N_d(t)}{t^{p+1}}\dd t
    -
    \frac{N_d(R)}{R^p}.
\end{align*}
Together with Lemma \ref{lem:lattice-count}, we have the final bound.
\end{proof}

In particular, since $\lambda_k=4\pi^2\abs{k}^2$ and
$\abs{B_1^4}=\pi^2/2$, in dimension four we have
\begin{align}\label{eq:lattice-critical-four}
    \sum_{0<\abs{k}\leq R}\frac{1}{\lambda_k^2}
    =
    \frac{\log R}{8\pi^2}+O(1).
\end{align}

\begin{lemma}\label{lem:square-lattice-sum}
For every $s>1$,
\begin{align*}
    \sum_{k\in\Z^2\setminus\{0\}}\frac{1}{\abs{k}^{2s}}
    =
    4\zeta(s)\beta(s),
\end{align*}
where
\begin{align*}
    \zeta(s):=\sum_{m=1}^{\infty}\frac{1}{m^s},
    \qquad
    \beta(s):=\sum_{j=0}^{\infty}\frac{(-1)^j}{(2j+1)^s}.
\end{align*}
In particular,
\begin{align*}
    \sum_{k\in\Z^2\setminus\{0\}}\frac{1}{\lambda_k^2}
    =
    \frac{\beta(2)}{24\pi^2}.
\end{align*}
\end{lemma}

\begin{proof}
Let $r_2(m)$ count the integer pairs $(k_1,k_2)$ satisfying
$k_1^2+k_2^2=m$. By Jacobi's two-square theorem \cite[Theorem~14.3, p.~428]{Nathanson00},
\begin{align*}
    r_2(m)=4\sum_{d\mid m}\chi_4(d),
    \qquad
    \chi_4(d):=
    \begin{cases}
        0,&d\text{ even},\\
        (-1)^{(d-1)/2},&d\text{ odd}.
    \end{cases}
\end{align*}
Grouping lattice points by their squared length and writing $m=d\ell$,
we obtain
\begin{align*}
    \sum_{k\in\Z^2\setminus\{0\}}\frac{1}{\abs{k}^{2s}}
    &=
    \sum_{m=1}^{\infty}\frac{r_2(m)}{m^s}\\
    &=
    4\sum_{d,\ell\geq1}\frac{\chi_4(d)}{d^s\ell^s}\\
    &=
    4\zeta(s)\beta(s).
\end{align*}
The second identity follows from $\lambda_k=4\pi^2\abs{k}^2$
and $\zeta(2)=\pi^2/6$.
\end{proof}

\begin{lemma}\label{lem:lattice-band-sums}
As $R\to\infty$, in dimension three,
\begin{align*}
    \sum_{R<\abs{k}\leq 2R}\frac{1}{\lambda_k^2}
    \asymp R^{-1}.
\end{align*}
In dimension four, for every fixed $\varepsilon\in(0,1)$,
\begin{align*}
    \sum_{\varepsilon R<\abs{k}\leq R}
    \frac{1}{\lambda_k^2}
    =
    \frac{\log(1/\varepsilon)}{8\pi^2}
    +O_\varepsilon(R^{-1}).
\end{align*}
\end{lemma}

\begin{proof}
By \eqref{eq:sum_by_parts}, for $1\leq A<B$,
\begin{align*}
    \sum_{A<\abs{k}\leq B}\frac{1}{\abs{k}^4}
    =
    \frac{N_d(B)}{B^4}
    -
    \frac{N_d(A)}{A^4}
    +
    4\int_A^B\frac{N_d(t)}{t^5}\dd t.
\end{align*}
By Lemma \ref{lem:lattice-count},
$N_d(t)=\abs{B_1^d}t^d+O_d(t^{d-1})$. Thus, for $d=3,4$,
\begin{align*}
    \sum_{A<\abs{k}\leq B}\frac{1}{\abs{k}^4}
    =
    d\abs{B_1^d}\int_A^B t^{d-5}\dd t
    +O_d(A^{d-5}).
\end{align*}
For $d=3$, taking $A=R$ and $B=2R$ gives
\begin{align*}
    \sum_{R<\abs{k}\leq 2R}\frac{1}{\abs{k}^4}
    =
    \frac{2\pi}{R}+O(R^{-2}).
\end{align*}
For $d=4$, taking $A=\varepsilon R$ and $B=R$ gives
\begin{align*}
    \sum_{\varepsilon R<\abs{k}\leq R}\frac{1}{\abs{k}^4}
    =
    2\pi^2\log(1/\varepsilon)+O_\varepsilon(R^{-1}).
\end{align*}
Dividing by $16\pi^4$ proves both claims.
\end{proof}

\section{Path stresses and segment-length moments ($d=4$)}
\label{app:stress}

\begin{proof}[Proof of Lemma \ref{lem:stress-L2}]
Fix $a>0$ and first assume $L\geq\max\{1,a\}$, so the periodic Gaussian bound in Step 4 is
uniform in $L$. The finitely many remaining sample sizes can be absorbed into $C_a$: the
deterministic displacement bound gives finite total variation for $S_0,S_1$, and heat smoothing
together with the $L^2$ boundedness of $\mathcal K$ gives a finite bound for each of these sizes.
\emph{Step 1: physical coordinates.} Let
\begin{align*}
    \widehat S_i:=nL^2\ (S_L)_\#S_i,
    \qquad i\in\{0,1\},
    \qquad
    \widehat\Sigma:=-\widehat S_1+\mathcal K\widehat S_0,
\end{align*}
where the factor $n$ comes from the mass of $\widehat\pi_n$ and the factor $L^2$ from the squared
 length that we owe to $\df_x(y)\df_x(y)^\top$. The Fourier multiplier in the definition of $\mathcal K$ is $k_ik_\ell/\abs{k}^2$ and  homogeneous of degree zero (invariant when replacing $k$ by $Lk$), so it
commutes with the dilation and $\widehat\Sigma$ is thus the dilation of $\Sigma$. Let $q_a^{(L)}$
be the heat kernel on $\T^4_L$, that is, the density with respect to Lebesgue measure with Fourier
coefficients $e^{-4\pi^2a^2\abs k^2/L^2}$, so that
\begin{align*}
    q_a^{(L)}(Lx)=L^{-4}q_{a/L}(x),
    \qquad x\in\T^4 .
\end{align*}
Combining the two displays above,
\begin{align}\label{eq:stress-scaling}
    \bigl(q_a^{(L)}*\widehat\Sigma\bigr)(Lx)
    =
    nL^{-2}\bigl(q_{a/L}*\Sigma\bigr)(x)
    =
    L^{2}\bigl(q_{a/L}*\Sigma\bigr)(x),
\end{align}
and therefore, by the change of variables $w=Lx$,
\begin{align}\label{eq:stress-reduction}
    \E\left[
        \norm{q_{a/L}*\Sigma}_{L_x^2}^2
    \right]
    =
    L^{-4}\,
    \E\left[
        \frac{1}{L^4}
        \int_{\T^4_L}
        \bigl|
            \bigl(q_a^{(L)}*\widehat\Sigma\bigr)(w)
        \bigr|^2
        \dd w
    \right].
\end{align}
It therefore suffices to bound the normalized spatial average on the right by a constant $C_a$ to prove \eqref{eq:stress-L2}.
 
\emph{Step 2: removing $\mathcal K$ and the spatial average.} Write
$\norm{F}_{L^2(\T^4_L)}^2:=L^{-4}\int_{\T^4_L}\abs F^2$. Since $\mathcal K$ is a Fourier multiplier
and $k_ik_\ell/\lvert k \rvert^2$ is bounded, it is bounded on $L^2(\T^4_L)$ and commutes with convolution by $q_a^{(L)}$.
Hence
\begin{align*}
    \norm{q_a^{(L)}*\widehat\Sigma}_{L^2(\T^4_L)}
    \leq
    \norm{q_a^{(L)}*\widehat S_1}_{L^2(\T^4_L)}
    +
    C\norm{q_a^{(L)}*\widehat S_0}_{L^2(\T^4_L)} .
\end{align*}
The sample is i.i.d.\ uniform and the periodic optimal plan is equivariant, implying that  the fields
$w\mapsto(q_a^{(L)}*\widehat S_i)(w)$ are stationary. By Fubini,
\begin{align}\label{eq:stress-stationarity}
    \E\norm{q_a^{(L)}*\widehat S_i}_{L^2(\T^4_L)}^2
    =
    \frac{1}{L^4}\int_{\T^4_L}
    \E\bigl|\bigl(q_a^{(L)}*\widehat S_i\bigr)(w)\bigr|^2\dd w
    =
    \E\bigl|\bigl(q_a^{(L)}*\widehat S_i\bigr)(0)\bigr|^2 ,
\end{align}
so it is enough to bound $\E\bigl|(q_a^{(L)}*\widehat S_i)(w)\bigr|^2$ uniformly in $w$ and $L$.
 
\emph{Step 3: the mass of the stress in a ball.} Both $\widehat S_0$ and $\widehat S_1$ are dominated
in total variation by the positive measure
\begin{align*}
    \widehat\nu(A)
    :=
    \int_0^1\int
    \1_{\{\widehat\gamma_{x,y}(t)\in A\}}
    \abs{x-y}^2
    \dd\widehat\pi_n(x,y)\dd t,
\end{align*}
because $\abs{\df\df^\top}=\abs{\df}^2$ and the weight $1-t$ in $S_0$ is at most one. Fix $w$ and
$m\geq0$. Every segment contributing to $\widehat\nu(B_{m+1}(w))$ meets $B_{m+1}(w)$, hence has
length at most $D_{w,m+1}$, and its endpoint $y$ on the Lebesgue side lies in
$B_{m+1+D_{w,m+1}}(w)$. Since the second marginal of $\widehat\pi_n$ is the Lebesgue measure, the mass of
such pairs is at most the volume of that ball, so
\begin{align}\label{eq:stress-ball-mass}
    \abs{\widehat S_i}\bigl(B_{m+1}(w)\bigr)
    \leq
    \widehat\nu\bigl(B_{m+1}(w)\bigr)
    \lesssim
    D_{w,m+2}^2
    \bigl(m+2+D_{w,m+2}\bigr)^4 .
\end{align}
By $D^2(m+2+D)^4\lesssim(m+2)^4D^2+D^6$, H\"older's inequality and \eqref{eq:segment-moments},
\begin{align}\label{eq:stress-ball-moment}
    \norm{
        \abs{\widehat S_i}\bigl(B_{m+1}(w)\bigr)
    }_{L^2(\Omega)}
    \lesssim
    (m+2)^4\norm{D_{w,m+2}}_{L^4}^2
    +
    \norm{D_{w,m+2}}_{L^{12}}^6
    \lesssim
    (1+m)^6 .
\end{align}
 
\emph{Step 4: summation over shells.} Partition $\T^4_L$ into the unit shells
$B_{m+1}(w)\setminus B_m(w)$, $m\geq0$. On the $m$-th shell,
\begin{align*}
    q_a^{(L)}
    \lesssim
    a^{-4}e^{-m^2/(Ca^2)} ,
\end{align*}
so that
\begin{align*}
    \bigl|\bigl(q_a^{(L)}*\widehat S_i\bigr)(w)\bigr|
    \leq
    \sum_{m\geq0}
    \left(
        \sup_{B_{m+1}(w)\setminus B_m(w)}q_a^{(L)}
    \right)
    \abs{\widehat S_i}\bigl(B_{m+1}(w)\bigr)
    \lesssim
    a^{-4}
    \sum_{m\geq0}
    e^{-m^2/(Ca^2)}
    \abs{\widehat S_i}\bigl(B_{m+1}(w)\bigr).
\end{align*}
Minkowski's inequality in $L^2(\Omega)$ and \eqref{eq:stress-ball-moment} give
\begin{align*}
    \E\bigl|\bigl(q_a^{(L)}*\widehat S_i\bigr)(w)\bigr|^2
    \lesssim
    a^{-8}
    \left(
        \sum_{m\geq0}e^{-m^2/(Ca^2)}(1+m)^6
    \right)^2
    \leq
    C_a ,
\end{align*}
a Gaussian summed against a polynomial. With \eqref{eq:stress-stationarity} and Step 2, this bounds
the right-hand side of \eqref{eq:stress-reduction} by $C_aL^{-4}$, which is \eqref{eq:stress-L2}.
\end{proof}

\section{A remark on superconcentration}
\label{app:superconcentration}

Define the Dirichlet energy of $W_{n,d}$ by
\begin{align*}
    \mathcal E_{n,d}
    :=
    \sum_{i=1}^n
    \E\abs{\nabla_{X_i}W_{n,d}}^2.
\end{align*}
The product Poincar\'e inequality gives
\begin{align*}
    \Var(W_{n,d})
    \leq
    \frac{1}{4\pi^2}\mathcal E_{n,d}.
\end{align*}
Following \cite{Cha14}, we say that $W_{n,d}$ is superconcentrated if
\begin{align*}
    \frac{\Var(W_{n,d})}{\mathcal E_{n,d}}
    \longrightarrow0.
\end{align*}

Equivalently, it is superconcentrated if Poincar\'e inequality gives asymptotically non-sharp bound.

\begin{proposition}\label{prop:superconcentration}
For $d\in\{2,3,4\}$, the Dirichlet energy satisfies
\begin{align*}
    \mathcal E_{n,d}
    \asymp
    \begin{cases}
        n^{-2}\log n,&d=2,\\
        n^{-5/3},&d=3,\\
        n^{-3/2},&d=4.
    \end{cases}
\end{align*}
Consequently,
\begin{align*}
    \frac{\Var(W_{n,d})}{\mathcal E_{n,d}}
    \asymp
    \begin{cases}
        (\log n)^{-1},&d=2,\\
        n^{-1/3},&d=3,\\
        n^{-1/2}\log n,&d=4,
    \end{cases}
\end{align*}
and hence $W_{n,d}$ is superconcentrated.
\end{proposition}

\begin{proof}
We first prove the upper bounds. By Lemma \ref{lem:high-frequency},
\begin{align*}
    \mathcal E_{n,d}
    \leq
    \frac4n\E W_{n,d}
    \lesssim
    \begin{cases}
        n^{-2}\log n,&d=2,\\
        n^{-5/3},&d=3,\\
        n^{-3/2},&d=4.
    \end{cases}
\end{align*}
It remains to prove the matching lower bounds.

\medskip
\paragraph{Lower bound for $d=2$.}
By orthogonality of the Hoeffding and Fourier projections,
\begin{align*}
    \mathcal E_{n,2}
    \geq
    \sum_{i=1}^n
    \E\abs{\nabla_{X_i}(P_2W_{n,2})_K}^2.
\end{align*}
Take $K=n^{1/8}$. The estimates in the proof of
Lemma \ref{lem:uniform-pairwise-response}, with $\alpha=1/4$, give
\begin{align*}
    n^2\E\abs{P_2W_{n,2}-U_n}^2
    \lesssim n^{-1/2}\log n.
\end{align*}
Since each pair mode with $\abs{k}\leq K$ has product eigenvalue
$2\lambda_k\leq8\pi^2K^2$, it follows that
\begin{align*}
    n^2\sum_{i=1}^n
    \E\abs{\nabla_{X_i}(P_2W_{n,2}-U_n)_K}^2
    &\leq
    8\pi^2K^2n^2\E\abs{P_2W_{n,2}-U_n}^2\\
    &\lesssim n^{-1/4}\log n
    \longrightarrow0.
\end{align*}
On the other hand, the Fourier representation of $U_{n,K}$ gives
\begin{align*}
    n^2\sum_{i=1}^n
    \E\abs{\nabla_{X_i}U_{n,K}}^2
    =
    4\frac{n-1}{n}
    \sum_{0<\abs{k}\leq K}\frac1{\lambda_k}
    \asymp\log K
    \asymp\log n.
\end{align*}
Therefore, by the triangle inequality,
\begin{align*}
    (n^2\mathcal E_{n,2})^{1/2}
    &\geq
    \left(
        n^2\sum_{i=1}^n
        \E\abs{\nabla_{X_i}(P_2W_{n,2})_K}^2
    \right)^{1/2}\\
    &\geq
    \left(
        n^2\sum_{i=1}^n
        \E\abs{\nabla_{X_i}U_{n,K}}^2
    \right)^{1/2}\\
    &\quad-
    \left(
        n^2\sum_{i=1}^n
        \E\abs{\nabla_{X_i}(P_2W_{n,2}-U_n)_K}^2
    \right)^{1/2}\\
    &\geq c\sqrt{\log n}-o(1).
\end{align*}
Hence $\mathcal E_{n,2}\gtrsim n^{-2}\log n$.

\paragraph{Lower bound for $d=3$.}
Let $L=n^{1/3}$ and fix $\varepsilon\in(0,1/2)$, to be chosen
sufficiently small. Set
\begin{align*}
    K:=\varepsilon L,
    \qquad
    B:=\{k\in\Z^3:K\leq\abs{k}<2K\}.
\end{align*}
For sufficiently large $n$, we have $L^{-1}\leq K^{-1}\leq1$.
Applying Lemma \ref{lem:one-band-response} with $s=K^{-1}$ and
then Lemma \ref{lem:pairwise-duality}, we obtain
\begin{align*}
    n^2\E\abs{(P_2W_{n,3}-U_n)_B}^2
    &\lesssim
    L^{-1}+L^{-3}K^2+n^{-2}K^{-1}\\
    &=
    K^{-1}\left(\varepsilon+\varepsilon^3+n^{-2}\right).
\end{align*}
On the other hand, by Lemma \ref{lem:lattice-band-sums}
\begin{align*}
    n^2\E\abs{(U_n)_B}^2
    =
    2\frac{n-1}{n}
    \sum_{k\in B}\lambda_k^{-2}
    \asymp K^{-1},
\end{align*}
where the comparison constants are independent of $\varepsilon$.
Choose $\varepsilon$ sufficiently small. Then, for sufficiently
large $n$, the triangle inequality gives
\begin{align*}
    \left(n^2\E\abs{(P_2W_{n,3})_B}^2\right)^{1/2}
    &\geq
    \left(n^2\E\abs{(U_n)_B}^2\right)^{1/2}\\
    &\quad-
    \left(n^2\E\abs{(P_2W_{n,3}-U_n)_B}^2\right)^{1/2}\\
    &\geq cK^{-1/2}.
\end{align*}
Every pair mode in $B$ has product eigenvalue
$2\lambda_k\geq8\pi^2K^2$. Therefore,
\begin{align*}
    n^2\mathcal E_{n,3}
    &\geq
    n^2\sum_{i=1}^n
    \E\abs{\nabla_{X_i}(P_2W_{n,3})_B}^2\\
    &\geq
    8\pi^2K^2n^2\E\abs{(P_2W_{n,3})_B}^2\\
    &\gtrsim K
    =\varepsilon L.
\end{align*}
Since $\varepsilon$ is fixed, this proves
$\mathcal E_{n,3}\gtrsim n^{-5/3}$.

\paragraph{Lower bound for $d=4$.}
Let $L=n^{1/4}$. By Proposition \ref{prop:aggregate4}, there is a
constant $C_0$ such that
\begin{align*}
    n^2\E\abs{P_2W_{n,4}-U_{n,L}}^2\leq C_0.
\end{align*}
Fix $\varepsilon\in(0,1)$, to be chosen sufficiently small, and set
\begin{align*}
    B:=\{k\in\Z^4:\varepsilon L<\abs{k}\leq L\}.
\end{align*}
By orthogonality,
\begin{align*}
    n^2\E\abs{(P_2W_{n,4}-U_{n,L})_B}^2\leq C_0.
\end{align*}
Moreover, the variance formula for the Green statistic and
Lemma \ref{lem:lattice-band-sums} give, for fixed $\varepsilon$,
\begin{align*}
    n^2\E\abs{(U_{n,L})_B}^2
    &=
    2\frac{n-1}{n}
    \sum_{\varepsilon L<\abs{k}\leq L}\lambda_k^{-2}\\
    &=
    \frac{1}{4\pi^2}\log\frac1\varepsilon+o(1).
\end{align*}
Choose $\varepsilon$ so that
\begin{align*}
    \left(\frac{1}{4\pi^2}\log\frac1\varepsilon\right)^{1/2}
    >\sqrt{C_0}+2.
\end{align*}
Then, for sufficiently large $n$, the triangle inequality gives
\begin{align*}
    \left(n^2\E\abs{(P_2W_{n,4})_B}^2\right)^{1/2}
    &\geq
    \left(n^2\E\abs{(U_{n,L})_B}^2\right)^{1/2}\\
    &\quad-
    \left(n^2\E\abs{(P_2W_{n,4}-U_{n,L})_B}^2\right)^{1/2}\\
    &\geq1.
\end{align*}
Every pair mode in $B$ has product eigenvalue
$2\lambda_k\geq8\pi^2\varepsilon^2L^2$. Therefore,
\begin{align*}
    n^2\mathcal E_{n,4}
    &\geq
    n^2\sum_{i=1}^n
    \E\abs{\nabla_{X_i}(P_2W_{n,4})_B}^2\\
    &\geq
    8\pi^2\varepsilon^2L^2
    n^2\E\abs{(P_2W_{n,4})_B}^2\\
    &\geq8\pi^2\varepsilon^2L^2.
\end{align*}
Since $\varepsilon$ is fixed, this proves
$\mathcal E_{n,4}\gtrsim n^{-3/2}$.

Combining these lower bounds with the upper bounds proves the
claimed orders of $\mathcal E_{n,d}$. The variance asymptotics in
Theorems \ref{thm:main} and \ref{thm:main2} then give the stated
ratios and superconcentration.
\end{proof}

\nocite{Hoe48,Ser80}
\bibliographystyle{alphaurl}
\bibliography{references}

@article{AKT84,
  author = {Ajtai, M. and Koml{\'o}s, J. and Tusn{\'a}dy, G.},
  title = {On optimal matchings},
  journal = {Combinatorica},
  year = {1984},
  volume = {4},
  number = {4},
  pages = {259--264},
  doi = {10.1007/BF02579135},
  url = {https://doi.org/10.1007/BF02579135}
}

@article{AST19,
  author = {Ambrosio, Luigi and Stra, Federico and Trevisan, Dario},
  title = {A {PDE} approach to a 2-dimensional matching problem},
  journal = {Probability Theory and Related Fields},
  year = {2019},
  volume = {173},
  number = {1--2},
  pages = {433--477},
  doi = {10.1007/s00440-018-0837-x},
  url = {https://doi.org/10.1007/s00440-018-0837-x}
}

@article{BJGR19,
  author = {Bernton, Espen and Jacob, Pierre E. and Gerber, Mathieu and Robert, Christian P.},
  title = {On parameter estimation with the {Wasserstein} distance},
  journal = {Information and Inference: A Journal of the IMA},
  year = {2019},
  volume = {8},
  number = {4},
  pages = {657--676},
  doi = {10.1093/imaiai/iaz003},
  url = {https://doi.org/10.1093/imaiai/iaz003}
}

@article{BL21,
  author = {Bobkov, Sergey G. and Ledoux, Michel},
  title = {A simple {Fourier} analytic proof of the {AKT} optimal matching theorem},
  journal = {The Annals of Applied Probability},
  year = {2021},
  volume = {31},
  number = {6},
  pages = {2567--2584},
  doi = {10.1214/20-AAP1656},
  url = {https://doi.org/10.1214/20-AAP1656}
}

@article{CLPS14,
  author = {Caracciolo, Sergio and Lucibello, Carlo and Parisi, Giorgio and Sicuro, Gabriele},
  title = {Scaling hypothesis for the {Euclidean} bipartite matching problem},
  journal = {Physical Review E},
  year = {2014},
  volume = {90},
  number = {1},
  pages = {012118},
  doi = {10.1103/PhysRevE.90.012118},
  url = {https://doi.org/10.1103/PhysRevE.90.012118}
}

@book{Cha14,
  author = {Chatterjee, Sourav},
  title = {Superconcentration and Related Topics},
  series = {Springer Monographs in Mathematics},
  publisher = {Springer},
  address = {Cham},
  year = {2014},
  doi = {10.1007/978-3-319-03886-5},
  url = {https://doi.org/10.1007/978-3-319-03886-5}
}

@article{dBGU05,
  author = {del Barrio, Eustasio and Gin{\'e}, Evarist and Utzet, Frederic},
  title = {Asymptotics for {$L_2$} functionals of the empirical quantile process, with applications to tests of fit based on weighted {Wasserstein} distances},
  journal = {Bernoulli},
  year = {2005},
  volume = {11},
  number = {1},
  pages = {131--189},
  doi = {10.3150/bj/1110228245},
  url = {https://doi.org/10.3150/bj/1110228245}
}

@article{dBGL24,
  author = {del Barrio, Eustasio and Gonz{\'a}lez-Sanz, Alberto and Loubes, Jean-Michel},
  title = {Central limit theorems for general transportation costs},
  journal = {Annales de l'Institut Henri Poincar{\'e}, Probabilit{\'e}s et Statistiques},
  year = {2024},
  volume = {60},
  number = {2},
  pages = {847--873},
  doi = {10.1214/22-AIHP1356},
  url = {https://doi.org/10.1214/22-AIHP1356}
}

@misc{dBG26,
  author = {del Barrio, Eustasio and Gonz{\'a}lez-Sanz, Alberto and Loubes, Jean-Michel and Rodr{\'i}guez-V{\'i}tores, David},
  title = {Distributional limit theory for optimal transport},
  year = {2025},
  eprint = {2505.19104},
  archivePrefix = {arXiv},
  primaryClass = {math.ST},
  note = {arXiv:2505.19104v2, revised 13 July 2026},
  doi = {10.48550/arXiv.2505.19104},
  url = {https://arxiv.org/abs/2505.19104v2}
}

@article{dBL19,
  author = {del Barrio, Eustasio and Loubes, Jean-Michel},
  title = {Central limit theorems for empirical transportation cost in general dimension},
  journal = {The Annals of Probability},
  year = {2019},
  volume = {47},
  number = {2},
  pages = {926--951},
  doi = {10.1214/18-AOP1275},
  url = {https://doi.org/10.1214/18-AOP1275}
}

@article{dJ87,
  author = {de Jong, Peter},
  title = {A central limit theorem for generalized quadratic forms},
  journal = {Probability Theory and Related Fields},
  year = {1987},
  volume = {75},
  number = {2},
  pages = {261--277},
  doi = {10.1007/BF00354037},
  url = {https://doi.org/10.1007/BF00354037}
}

@article{GH22,
  author = {Goldman, Michael and Huesmann, Martin},
  title = {A fluctuation result for the displacement in the optimal matching problem},
  journal = {The Annals of Probability},
  year = {2022},
  volume = {50},
  number = {4},
  pages = {1446--1477},
  doi = {10.1214/21-AOP1562},
  url = {https://doi.org/10.1214/21-AOP1562}
}

@article{GT21,
  author = {Goldman, Michael and Trevisan, Dario},
  title = {Convergence of asymptotic costs for random {Euclidean} matching problems},
  journal = {Probability and Mathematical Physics},
  year = {2021},
  volume = {2},
  number = {2},
  pages = {341--362},
  doi = {10.2140/pmp.2021.2.341}
}

@article{GGCN23,
  author = {Gonz{\'a}lez-Delgado, Javier and Gonz{\'a}lez-Sanz, Alberto and Cort{\'e}s, Juan and Neuvial, Pierre},
  title = {Two-sample goodness-of-fit tests on the flat torus based on {Wasserstein} distance and their relevance to structural biology},
  journal = {Electronic Journal of Statistics},
  year = {2023},
  volume = {17},
  number = {1},
  pages = {1547--1586},
  doi = {10.1214/23-EJS2135}
}

@article{HMS21,
  author = {Hallin, Marc and Mordant, Gilles and Segers, Johan},
  title = {Multivariate goodness-of-fit tests based on {Wasserstein} distance},
  journal = {Electronic Journal of Statistics},
  year = {2021},
  volume = {15},
  number = {1},
  pages = {1328--1371},
  doi = {10.1214/21-EJS1816}
}

@article{Hoe48,
  author = {Hoeffding, Wassily},
  title = {A class of statistics with asymptotically normal distribution},
  journal = {The Annals of Mathematical Statistics},
  year = {1948},
  volume = {19},
  number = {3},
  pages = {293--325},
  doi = {10.1214/aoms/1177730196}
}

@unpublished{Led19,
  author = {Ledoux, Michel},
  title = {A fluctuation result in dual {Sobolev} norm for the optimal matching problem},
  year = {2019},
  note = {Unpublished manuscript},
  url = {https://perso.math.univ-toulouse.fr/ledoux/files/2022/03/matchingclt.pdf}
}

@article{McC01,
  author = {McCann, R. J.},
  title = {Polar factorization of maps on {Riemannian} manifolds},
  journal = {Geometric and Functional Analysis},
  year = {2001},
  volume = {11},
  number = {3},
  pages = {589--608},
  doi = {10.1007/PL00001679}
}

@book{Nathanson00,
  author = {Nathanson, Melvyn B.},
  title = {Elementary Methods in Number Theory},
  series = {Graduate Texts in Mathematics},
  volume = {195},
  publisher = {Springer},
  address = {New York},
  year = {2000},
  doi = {10.1007/b98870}
}

@article{PC19,
  author = {Peyr{\'e}, Gabriel and Cuturi, Marco},
  title = {Computational optimal transport: With applications to data science},
  journal = {Foundations and Trends in Machine Learning},
  year = {2019},
  volume = {11},
  number = {5--6},
  pages = {355--607},
  doi = {10.1561/2200000073}
}

@book{Ser80,
  author = {Serfling, Robert J.},
  title = {Approximation Theorems of Mathematical Statistics},
  publisher = {John Wiley \& Sons},
  address = {New York},
  year = {1980},
  doi = {10.1002/9780470316481}
}

@article{SM18,
  author = {Sommerfeld, Max and Munk, Axel},
  title = {Inference for empirical {Wasserstein} distances on finite spaces},
  journal = {Journal of the Royal Statistical Society: Series B (Statistical Methodology)},
  year = {2018},
  volume = {80},
  number = {1},
  pages = {219--238},
  doi = {10.1111/rssb.12236}
}

@book{Tal14,
  author = {Talagrand, Michel},
  title = {Upper and Lower Bounds for Stochastic Processes: Modern Methods and Classical Problems},
  series = {Ergebnisse der Mathematik und ihrer Grenzgebiete. 3. Folge},
  volume = {60},
  publisher = {Springer},
  address = {Berlin, Heidelberg},
  year = {2014},
  doi = {10.1007/978-3-642-54075-2}
}

@article{TGR24,
  author = {Torous, William and Gunsilius, Florian and Rigollet, Philippe},
  title = {An optimal transport approach to estimating causal effects via nonlinear difference-in-differences},
  journal = {Journal of Causal Inference},
  year = {2024},
  volume = {12},
  number = {1},
  pages = {20230004},
  doi = {10.1515/jci-2023-0004}
}

@book{Vil09,
  author = {Villani, C{\'e}dric},
  title = {Optimal Transport: Old and New},
  series = {Grundlehren der mathematischen Wissenschaften},
  volume = {338},
  publisher = {Springer},
  address = {Berlin, Heidelberg},
  year = {2009},
  doi = {10.1007/978-3-540-71050-9}
}

@article{Sho86,
  author = {Shor, Peter W.},
  title = {The average-case analysis of some on-line algorithms for bin packing},
  journal = {Combinatorica},
  volume = {6},
  number = {2},
  pages = {179--200},
  year = {1986},
  doi = {10.1007/BF02579171}
}

@article{LS89,
  author = {Leighton, T. and Shor, P.},
  title = {Tight bounds for minimax grid matching with applications to the average case analysis of algorithms},
  journal = {Combinatorica},
  volume = {9},
  number = {2},
  pages = {161--187},
  year = {1989},
  doi = {10.1007/BF02124678}
}

@article{SY91,
  author = {Shor, P. W. and Yukich, J. E.},
  title = {Minimax grid matching and empirical measures},
  journal = {The Annals of Probability},
  volume = {19},
  number = {3},
  pages = {1338--1348},
  year = {1991},
  doi = {10.1214/aop/1176990347}
}

@article{Tal92,
  author = {Talagrand, Michel},
  title = {Matching random samples in many dimensions},
  journal = {The Annals of Applied Probability},
  volume = {2},
  number = {4},
  pages = {846--856},
  year = {1992},
  doi = {10.1214/aoap/1177005578}
}

@article{TY93,
  author = {Talagrand, M. and Yukich, J. E.},
  title = {The integrability of the square exponential transportation cost},
  journal = {The Annals of Applied Probability},
  volume = {3},
  number = {4},
  pages = {1100--1111},
  year = {1993},
  doi = {10.1214/aoap/1177005274}
}

@article{Tal94a,
  author = {Talagrand, M.},
  title = {Matching theorems and empirical discrepancy computations using majorizing measures},
  journal = {Journal of the American Mathematical Society},
  volume = {7},
  number = {2},
  pages = {455--537},
  year = {1994},
  doi = {10.1090/S0894-0347-1994-1227476-X}
}

@article{Tal94b,
  author = {Talagrand, M.},
  title = {The transportation cost from the uniform measure to the empirical measure in dimension {$\geq 3$}},
  journal = {The Annals of Probability},
  volume = {22},
  number = {2},
  pages = {919--959},
  year = {1994},
  doi = {10.1214/aop/1176988735}
}

@article{DY95,
  author = {Dobri{\'c}, V. and Yukich, J. E.},
  title = {Asymptotics for transportation cost in high dimensions},
  journal = {Journal of Theoretical Probability},
  volume = {8},
  number = {1},
  pages = {97--118},
  year = {1995},
  doi = {10.1007/BF02213456}
}

@book{Yuk98,
  author = {Yukich, Joseph E.},
  title = {Probability Theory of Classical {Euclidean} Optimization Problems},
  series = {Lecture Notes in Mathematics},
  volume = {1675},
  publisher = {Springer-Verlag},
  address = {Berlin},
  year = {1998},
  isbn = {978-3-540-63666-3},
  doi = {10.1007/BFb0093472}
}

@article{AG19,
  author = {Ambrosio, Luigi and Glaudo, Federico},
  title = {Finer estimates on the {$2$}-dimensional matching problem},
  journal = {Journal de l'{\'E}cole polytechnique --- Math{\'e}matiques},
  volume = {6},
  pages = {737--765},
  year = {2019},
  doi = {10.5802/jep.105},
  url = {https://jep.centre-mersenne.org/articles/10.5802/jep.105/}
}

@article{AGT19,
  author = {Ambrosio, Luigi and Glaudo, Federico and Trevisan, Dario},
  title = {On the optimal map in the {$2$}-dimensional random matching problem},
  journal = {Discrete and Continuous Dynamical Systems},
  volume = {39},
  number = {12},
  pages = {7291--7308},
  year = {2019},
  doi = {10.3934/dcds.2019304},
  url = {https://www.aimsciences.org/article/doi/10.3934/dcds.2019304}
}

@article{GHO21,
  author = {Goldman, Michael and Huesmann, Martin and Otto, Felix},
  title = {Quantitative linearization results for the {Monge--Amp\`ere} equation},
  journal = {Communications on Pure and Applied Mathematics},
  volume = {74},
  number = {12},
  pages = {2483--2560},
  year = {2021},
  doi = {10.1002/cpa.21994},
  url = {https://onlinelibrary.wiley.com/doi/10.1002/cpa.21994}
}

@incollection{GHO25,
  author = {Goldman, Michael and Huesmann, Martin and Otto, Felix},
  title = {Almost sharp rates of convergence for the average cost and displacement in the optimal matching problem},
  editor = {Ehrnstr{\"o}m, Mats and Holden, Helge and Jakobsen, Espen R.},
  booktitle = {Partial Differential Equations: Waves, Nonlinearities and Nonlocalities: The Abel Symposium 2023},
  series = {Abel Symposia},
  volume = {18},
  pages = {93--109},
  publisher = {Springer},
  address = {Cham},
  year = {2025},
  doi = {10.1007/978-3-031-91282-5_5},
  url = {https://link.springer.com/chapter/10.1007/978-3-031-91282-5_5},
  eprint = {2312.07995},
  archivePrefix = {arXiv},
  primaryClass = {math.AP}
}

@misc{AGGT26,
  author = {Armegioiu, Victor and Goldman, Michael and Grotto, Francesco and Trevisan, Dario},
  title = {Exact asymptotics for the {2D Euclidean} random matching problem},
  year = {2026},
  eprint = {2609.08717},
  archivePrefix = {arXiv},
  primaryClass = {math.PR},
  note = {Preprint, version 1, 8 September 2026},
  url = {https://arxiv.org/abs/2609.08717}
}

@article{MBNW24,
  author  = {Manole, Tudor and Balakrishnan, Sivaraman
             and Niles-Weed, Jonathan and Wasserman, Larry},
  title   = {Plugin estimation of smooth optimal transport maps},
  journal = {The Annals of Statistics},
  year    = {2024},
  volume  = {52},
  number  = {3},
  pages   = {966--998},
  doi     = {10.1214/24-AOS2379}
}

\end{document}